\documentclass[11pt]{article}
\usepackage[margin=1in]{geometry}
\usepackage{amsmath,amssymb,amsthm,mathrsfs,enumerate,bm,xcolor,multirow,pbox,mathtools}
\usepackage{graphicx,color,framed,tikz,enumitem}
\usepackage{subfigure,float}
\usepackage[normalem]{ulem} % for strikethrough text

\setlist{leftmargin=5mm}
\usepackage[colorlinks,linkcolor=red,citecolor=blue,urlcolor=blue]{hyperref}
\allowdisplaybreaks[4]

\AtBeginDocument{%
	\def\MR#1{}
}

\numberwithin{equation}{section}

\newcommand{\R}{\mathbb{R}}
\newcommand{\RR}{\mathbb{R}}

\newcommand{\E}{\mathbb{E}}
\newcommand{\Prob}{\mathbb{P}}

\newcommand{\td}{\widetilde}

\renewcommand{\cal}{\mathcal}
\newcommand{\var}{\text{Var}}

\newcommand{\Mod}{\text{Mod}}
\newcommand{\KL}{\text{KL}}

\newcommand{\ora}{{\rm oracle}}
\newcommand{\pen}{\text{pen}}
\newcommand{\ada}{\text{adapt}}
\newcommand{\ldot}{\underline{\odot}}
\newcommand{\udot}{\overline{\odot}}

\newcommand{\pnorm}[2]{\lVert#1\rVert_{#2}}

\newcommand{\abs}[1]{\lvert#1\rvert}

\newcommand{\biggabs}[1]{\bigg\lvert#1\bigg\rvert}

\renewcommand{\epsilon}{\varepsilon}

\renewcommand{\d}[1]{\mathrm{d}#1}

\newcommand{\floor}[1]{\left\lfloor #1 \right\rfloor}
\newcommand{\ceil}[1]{\left\lceil #1 \right\rceil}

\newcommand{\beq}{\begin{equation}}
\newcommand{\eeq}{\end{equation}}
\newcommand{\beqa}{\begin{equation} \begin{aligned}}
\newcommand{\eeqa}{\end{aligned} \end{equation}}
\newcommand{\beqas}{\begin{equation*} \begin{aligned}}
\newcommand{\eeqas}{\end{aligned} \end{equation*}}

\newcommand{\bit}{\begin{itemize}}
	\newcommand{\eit}{\end{itemize}}
\newcommand{\bmat}{\begin{bmatrix}}
	\newcommand{\emat}{\end{bmatrix}}
\def\undertilde#1{\mathord{\vtop{\ialign{##\crcr
$\hfil\displaystyle{#1}\hfil$\crcr\noalign{\kern1.5pt\nointerlineskip}
$\hfil\tilde{}\hfil$\crcr\noalign{\kern1.5pt}}}}}

\theoremstyle{definition}\newtheorem{problem}{Problem}[section]
\theoremstyle{definition}
\theoremstyle{remark}

\theoremstyle{remark}\newtheorem{remark}[problem]{Remark}
\theoremstyle{definition}\newtheorem{example}[problem]{Example}
\theoremstyle{plain}\newtheorem{theorem}[problem]{Theorem}
\theoremstyle{plain}
\theoremstyle{plain}\newtheorem{lemma}[problem]{Lemma}
\theoremstyle{plain}
\theoremstyle{plain}\newtheorem{proposition}[problem]{Proposition}
\theoremstyle{plain}
\theoremstyle{plain}

\DeclarePairedDelimiter{\parr}{(}{)}

\DeclarePairedDelimiter{\co}{{\rm Coef}[}{]}

\DeclareMathOperator*{\argmin}{arg\,min}

\RequirePackage{lineno}
\newcommand*\patchAmsMathEnvironmentForLineno[1]{%
	\expandafter\let\csname old#1\expandafter\endcsname\csname #1\endcsname
	\expandafter\let\csname oldend#1\expandafter\endcsname\csname end#1\endcsname
	\renewenvironment{#1}%
	{\linenomath\csname old#1\endcsname}%
	{\csname oldend#1\endcsname\endlinenomath}}% 
\newcommand*\patchBothAmsMathEnvironmentsForLineno[1]{%
	\patchAmsMathEnvironmentForLineno{#1}%
	\patchAmsMathEnvironmentForLineno{#1*}}%
\AtBeginDocument{%
	\patchBothAmsMathEnvironmentsForLineno{equation}%
	\patchBothAmsMathEnvironmentsForLineno{align}%
	\patchBothAmsMathEnvironmentsForLineno{flalign}%
	\patchBothAmsMathEnvironmentsForLineno{alignat}%
	\patchBothAmsMathEnvironmentsForLineno{gather}%
	\patchBothAmsMathEnvironmentsForLineno{multline}%
}
\begin{document}
	
\title{\LARGE On a phase transition in general order spline regression}

\date{\today}

\author{Yandi Shen\thanks{Department of Statistics, University of Washington, Seattle, WA 98105, USA. E-mail: \tt{ydshen@uw.edu} },~  Qiyang Han\thanks{Department of Statistics, Rutgers University, Piscataway, NJ 08854, USA. E-mail: \tt{qh85@stat.rutgers.edu}},~ and Fang Han\thanks{Department of Statistics, University of Washington, Seattle, WA 98105, USA. E-mail: \tt{fanghan@uw.edu}} }
\maketitle
	
%\title[Phase transition in splines]{On a phase transition in general order spline regression}

%\thanks{The research of Q. Han is partially supported by DMS-1916221. \red{Enter here funding resource.} }
%

%\author[Y. Shen]{Yandi Shen}
%%
%%
%\address[Y. Shen]{
%	Department of Statistics, University of Washington, Seattle, WA 98105, USA.
%}
%\email{ydshen@uw.edu}
%
%\author[Q. Han]{Qiyang Han}
%%
%%
%\address[Q. Han]{
%	Department of Statistics, Rutgers University, Piscataway, NJ 08854, USA.
%}
%\email{qh85@stat.rutgers.edu}
%%
%\author[F. Han]{Fang Han}
%%
%\address[F. Han]{
%	Department of Statistics, University of Washington, Seattle, WA 98105, USA.
%}
%\email{fanghan@uw.edu}

%\date{\today}
%\keywords{splines, minimax rate, phase transition, law of iterated logarithm, shape constraint}
%\subjclass[2000]{60F17, 62E17}

\begin{abstract}
In the Gaussian sequence model $Y= \theta_0 + \varepsilon$ in $\RR^n$, we study the fundamental limit of approximating the signal $\theta_0$ by a class $\Theta(d,d_0,k)$ of (generalized) splines with free knots. Here $d$ is the degree of the spline, $d_0$ is the order of differentiability at each inner knot, and $k$ is the maximal number of pieces. We show that, given any integer $d\geq 0$ and $d_0\in\{-1,0,\ldots,d-1\}$, the minimax rate of estimation over $\Theta(d,d_0,k)$ exhibits the following phase transition:
\begin{equation*}
\begin{aligned}
\inf_{\td{\theta}}\sup_{\theta\in\Theta(d,d_0, k)}\E_\theta\pnorm{\td{\theta} - \theta}{}^2 \asymp_d
\begin{cases}
k\log\log(16n/k), & 2\leq k\leq k_0,\\
k\log(en/k), & k \geq k_0+1.
\end{cases}
\end{aligned}
\end{equation*}
The transition boundary $k_0$, which takes the form $\floor{(d+1)/(d-d_0)} + 1$, demonstrates the critical role of the regularity parameter $d_0$ in the separation between a faster $\log \log(16n)$ and a slower $\log(en)$ rate. We further show that, once encouraging an additional `$d$-monotonicity' shape constraint (including monotonicity for $d = 0$ and convexity for $d=1$), the above phase transition is eliminated and the faster $k\log\log(16n/k)$ rate can be achieved for all $k$. These results provide theoretical support for developing $\ell_0$-penalized (shape-constrained) spline regression procedures as useful alternatives to $\ell_1$- and $\ell_2$-penalized ones.
\end{abstract}

%\maketitle

% \tableofcontents

\section{Introduction}

\subsection{Overview}
Consider the regression model 
\begin{align}\label{eq:model}
Y_i = f_0(i/n)+ \epsilon_i, \qquad i=1,\ldots,n,
\end{align}
where $f_0:[0,1]\to\mathbb{R}$ is an unknown function and $\epsilon_i$'s are independent normal random variables with mean zero and variance $\sigma^2$. Throughout the paper, we reserve the notation $\theta_0$ for the truth in (\ref{eq:model}), i.e., $(\theta_0)_i \equiv f_0(i/n)$. The main goal of this paper is to study the approximation of $\theta_0$ by splines with free knots.

Consider the (generalized) spline space with the following three parameters: $d$, the degree of the spline; $d_0$, the level of continuity; $k$, the maximal number of pieces. More formally, $(d,d_0,k)$-splines are defined as (exact definition in Section \ref{sec:gen_spline}):
\begin{align}\label{def:spline_fcn_intro}
\Big\{&f:[0,1]\rightarrow \R: f \text{ has at most $k+1$ knots, is a degree $d$ polynomial} \\
&\text{  between knots, and is $d_0$-times differentiable at each inner knot}\Big\}. \notag
\end{align}
For any fixed degree $d$, $d_0$ takes value in $\{-1,0,\ldots,d-1\}$, with $d_0 = d-1$ being the smoothest case and $d_0 = -1$ allowing for discontinuity between pieces. To avoid degeneracy to global polynomials, we only consider the case $k\geq 2$ in this paper. The corresponding sequence space is defined as
\begin{align}\label{def:spline_seq_intro}
\Theta(d,d_0,k)\equiv \Big\{\theta\in\RR^n:\theta_i = f(i/n)\text{ for some  $(d,d_0,k)$-spline } f\Big\}.
\end{align}
Compared to splines in more classical settings \cite{de1978practical,
green1993nonparametric,
wahba1990spline}, the above parameter space does not fix the knots a priori and thus provides more flexibility. Previously, general order splines with free knots have been studied in, e.g., \cite{mammen1997locally,tibshirani2014adaptive,baranowski2019narrowest}.

Splines of the forms (\ref{def:spline_fcn_intro}) and (\ref{def:spline_seq_intro}) have frequently emerged in nonparametric curve estimation problems. For example, the classical smoothing splines \cite{wahba1990spline} arise from minimizing the least squares criterion with an $\ell_2$ roughness penalty. In the $\ell_1$ world, splines are closely related to total variation regularization or denoising studied in, e.g.,  \cite{rudin1992nonlinear,mammen1997locally,chen2001atomic,davis2001local,tibshirani2005sparsity,steidl2006splines,rinaldo2009properties,harchaoui2010multiple,hutter2016optimal,dalalyan2017prediction}. In recent years, these methods with the spline space (\ref{def:spline_seq_intro}) received a revival of interest under the name \emph{trend filtering}; cf. \cite{kim2009ell_1,tibshirani2014adaptive,wang2014falling,guntuboyina2020adaptive}. 

Despite the long history and large volume of works related to the spline spaces (\ref{def:spline_fcn_intro})-(\ref{def:spline_seq_intro}), their fundamental statistical limits have remained largely unexplored. Our first main result in this paper reveals the following intriguing phase transition in the minimax rate of estimation error over $\Theta(d, d_0, k)$:
\begin{equation}\label{eq:minimax_noshape_intro}
\begin{aligned}
\inf_{\td{\theta}}\sup_{\theta\in\Theta(d,d_0, k)}\E_\theta\pnorm{\td{\theta} - \theta}{}^2 \asymp_d
\begin{cases}
\sigma^2k\log\log(16n/k), & 2\leq k\leq k_0,\\
\sigma^2k\log(en/k), & k \geq k_0+1.
\end{cases}
\end{aligned}
\end{equation}
Here, $\|\cdot\|$ denotes the Euclidean norm and $\asymp_d$ denotes equivalence in order up to some positive constant that only depends on $d$. The transition boundary $k_0$, which takes the form $\floor{(d+1)/(d-d_0)}+1$ with $\floor{\cdot}$ denoting the floor function, governs the maximal number of pieces above which the optimal dependence of the estimation error on the sample size $n$ changes from the faster $\log\log (16n)$ rate to the slower $\log (en)$ rate. Notably, for any fixed degree $d$, $k_0$ is an increasing function of the regularity parameter $d_0$. In the two extreme cases, we have $k_0 = d+2$ if $d_0 = d-1$ (smoothest) and $k_0 = 2$ if $d_0 = -1$ (roughest). In other words, \emph{the driving factor behind the phase transition in (\ref{eq:minimax_noshape_intro}) is the regularity due to the differentiability structure encoded in $d,d_0$}.

The minimax rate in (\ref{eq:minimax_noshape_intro}) is achieved by the $\ell_0$-constrained spline least squares estimator (LSE) $\widehat{\theta}\equiv \widehat{\theta}(\Theta(d,d_0,k),Y)$, with $Y\equiv (Y_1,\ldots,Y_n)^\top$ and
\begin{align}\label{eq:lse}
\widehat{\theta}(\Theta, Y) \equiv \argmin_{\theta\in\Theta}\pnorm{Y - \theta}{2}^2\quad \text{ for any }\Theta\subset\RR^n.
\end{align}
In fact, a more general oracle inequality allowing for arbitrary model mis-specification can be proved for $\widehat{\theta}$. Due to the non-convexity of $\Theta(d,d_0,k)$, the solution to  \eqref{eq:lse} with $\Theta=\Theta(d,d_0,k)$ may not be unique and we choose any $\widehat{\theta}$ that achieves the minimum. Among the three parameters, we take $d$ and $d_0$ to be fixed in advance and consider $k$ as a tuning parameter to balance the approximation error of $\theta_0$ in \eqref{eq:model} by $\Theta(d,d_0,k)$ and the complexity of the latter space. The estimator in \eqref{eq:lse} with $\Theta = \Theta(d,d_0,k)$ can therefore be viewed as a class of $\ell_0$-splines in their constrained form.

The minimax rate in \eqref{eq:minimax_noshape_intro} and the rate-optimality of $\ell_0$-constrained spline LSE are interesting from at least two very different angles. First, the minimax rate in \eqref{eq:minimax_noshape_intro} is particularly useful in penalty selection for the adaptive version of the $\ell_0$-constrained spline LSE $\widehat{\theta}(\Theta(d,d_0,k),Y)$. Specifically, suppose $\theta_0\in\Theta(d,d_0,k^*)$ in \eqref{eq:model} with $d$ and $d_0 $ fixed in advance and an unknown $k^*$ on the number of pieces. Our aim is to find an adaptive version of $\widehat{\theta}$ that does not require the knowledge of $k^*$ but remains minimax optimal in estimation. Using the classical approach in \cite{birge1993rates,barron1999risk, birge2001gaussian,massart2007concentration}, this can be done by resorting to the penalized spline LSE $\widehat{\theta}_{\ada}$, where 
\begin{align}\label{eq:pen_lse_def_intro}
\widehat{\theta}_{\ada} &\equiv \widehat{\theta}(\Theta(d,d_0,\widehat{k}), Y)
\end{align}
with some data-driven $\widehat{k}$:
\begin{align}\label{eq:khat_intro}
\widehat{k}&\equiv \argmin_{1\leq k\leq n}\Big\{\pnorm{Y-\widehat{\theta}(\Theta(d,d_0,k),Y)}{}^2 + \pen(k;d,d_0)\Big\}
\end{align}
for some penalty function $\pen(\cdot; d,d_0)$. The estimator $\widehat{\theta}_\ada$ can thus be viewed as a class of $\ell_0$-penalized splines. Similar $\ell_0$-penalized procedures have previously been studied in \cite{kohler1999nonparametric,boysen2009consistencies,fan2018approximate,jewell2018exact}. When the penalty $\pen(\cdot; d,d_0)$ is chosen to be proportional to the minimax rate established in \eqref{eq:minimax_noshape_intro}, $\widehat{\theta}_{\ada}$ is guaranteed to be adaptively minimax optimal over $\Theta(d,d_0,k)$ for all values of $k$.

Second, \eqref{eq:minimax_noshape_intro} suggests some interesting comparison between $\ell_0$- and $\ell_1$-regularizers in spline regression. For expository purpose, let us consider the simplest piecewise constant class $\Theta(0,-1,k)$, where the transition boundary is given by $k_0 = 2$. There, while the $\ell_0$-constrained spline LSE, as defined in \eqref{eq:lse} with $\Theta = \Theta(0,-1,2)$, is able to achieve the faster $\log\log (16n)$ rate with $2$ pieces, the same rate has been proven to be un-attainable by the $\ell_1$ trend filtering, even with an additional minimum spacing condition that could be substantially improved with $\ell_0$-splines \cite{van2018tight,fan2018approximate, guntuboyina2020adaptive}. Computationally, unlike the context of sparse linear regression where the $\ell_0$ problem of best-subset selection is provably NP-hard \cite{natarajan1995sparse}, efficient dynamic programming algorithms do exist for implementing \eqref{eq:lse}, at least in the discontinuous case ($d_0 = -1$) \cite{auger1989algorithms,winkler2002smoothers,jackson2005algorithm,friedrich2008complexity} and the first-order continuous case ($d_0 = 0$) \cite{fearnhead2019detecting}.  Our results hence suggest that \emph{the $\ell_0$-constrained spline LSE could be an attractive alternative to its $\ell_1$ counterparts in spline regressions}.

To motivate the second main result of this paper, we recall the following minimax result from \cite{gao2017minimax}: for all $k\geq 2$,
\begin{align}\label{eq:minimax_GHZ_intro}
\inf_{\td{\theta}^*}\sup_{\theta^*\in\Theta^*(0,k)}\E_{\theta^*}\|\td{\theta}^* - \theta^*\|^2 \asymp \sigma^2k\log\log(16n/k),
\end{align}
where $\Theta^\ast(0,k)$ is the sub-class of $\Theta(0,-1,k)$ with non-decreasing signals. Comparing (\ref{eq:minimax_noshape_intro}) with $d=0,d_0=-1$ and (\ref{eq:minimax_GHZ_intro}) above, we see that the phase transition from the faster rate $\log\log(16n)$ to the slower rate $\log (en)$ in (\ref{eq:minimax_noshape_intro}) is eliminated in (\ref{eq:minimax_GHZ_intro}) under the additional monotonicity shape constraint. This raises the natural questions of whether a similar gain by shape constraints applies to higher-order splines, and if so, which type of shape constraints should be encouraged. As shape-constrained models repeatedly prove their usefulness in various applications, answering the above questions is of both practical and theoretical interests.

To this end, following  \cite{balabdaoui2007estimation,chatterjee2015risk}, we consider the following sub-class of $(d,d_0,k)$-splines with an additional `$d$-monotone' shape constraint (exact definition in Section \ref{sec:gen_shape}):
\begin{align}\label{eq:fun-class-intro}
\Big\{f:[0,1]\rightarrow\RR:& \text{ $f$ is a $(d,d-1,k)$-spline with non-decreasing}\\ 
&\text{highest-order polynomial coefficients}\Big\}. \notag
\end{align}
Two canonical examples are $d = 0$ and $1$, with the former corresponding to non-decreasing signals with at most $k$ constant pieces, and the latter corresponding to convex signals with at most $k$ linear pieces. Both classes have been extensively studied in the literature; cf. \cite{zhang2002risk,chatterjee2015risk,bellec2018sharp,gao2017minimax} for the case $d = 0$ and \cite{guntuboyina2013global,chatterjee2015risk,bellec2018sharp} for the case $d = 1$. Define the sequence space corresponding to \eqref{eq:fun-class-intro} as $\Theta^*(d,k)$. 

As a special case of our second main result, we show an analogue of (\ref{eq:minimax_GHZ_intro}) under the convexity (=1-monotone) shape constraint: for all $k\geq 2$,
\begin{align}\label{eq:minimax_shape_intro}
\inf_{\td{\theta}^*}\sup_{\theta^*\in\Theta^*(1,k)}\E_{\theta^*}\|\td{\theta}^* - \theta^*\|^2 \asymp \sigma^2k\log\log(16n/k).
\end{align}
The same upper bound actually holds for the general $d$-monotone class $\Theta^*(d,k)$, with a complementary lower bound showing that the $\log\log (16n)$ rate cannot be further improved even with only two pieces. Comparing (\ref{eq:minimax_noshape_intro}) and (\ref{eq:minimax_shape_intro}), it is hence clear that \emph{a higher-order `$d$-monotonicity' shape constraint eliminates the phase transition in (\ref{eq:minimax_noshape_intro}) for general $d$ in that the faster $k\log\log(16n/k)$ rate can now be achieved for all $k$.} The $d$-monotonicity therefore offers an attractive non-parametric sub-class $\Theta^*(d,k)$ of the general $\Theta(d,d-1,k)$ over which additional gain can be obtained in estimating the underlying signal. 

Finally, we remark on the technical challenges in proving (\ref{eq:minimax_noshape_intro}) and (\ref{eq:minimax_shape_intro}). Unlike the relatively straightforward proof of the $\log(en)$ part in \eqref{eq:minimax_noshape_intro}, the derivation of the correct transition boundary $k_0$ and the faster $\log\log(16n)$ rate requires non-trivial efforts from both analytical and probabilistic angles. The analytic step is to derive sharp enough controls for the magnitudes of the polynomial coefficients of signals in $\Theta(d,d_0,k)$ and $\Theta^\ast(d,k)$, which, in a certain sense, need be `tied' to either the left-most or the right-most knot of the signal. This is possible either due to the strong regularity inherited in the differentiability structure of $\Theta(d,d_0,k)$ for $k\leq k_0$, or to the global regularity within the $d$-monotonicity shape constraint. Once the above controls are obtained, a generalized version of the law of iterated logarithm (LIL), which we will develop in Section \ref{sec:lil}, can be applied to obtain the iterated logarithmic rates in (\ref{eq:minimax_noshape_intro}) and (\ref{eq:minimax_shape_intro}). 

The rest of the paper is organized as follows. Sections \ref{sec:gen_spline}
and \ref{sec:gen_shape} are devoted to the study of unshaped splines $\Theta(d,d_0,k)$ and shaped splines $\Theta^*(d,k)$, respectively. A general version of the LIL in expectation is developed in Section \ref{sec:lil}. Main proofs of the results are presented in Sections \ref{sec:noshape_upper} and \ref{sec:shape_upper}, with the remaining technical lemmas collected in the Appendix.

\subsection{Notation} 

For any $x\in \R$, write $(x)_+ \equiv\max\{x,0\}$. Let $\bm{1}_\cdot$ denote the indicator function. For any non-negative integers $a,b$, we use $[a;b]$ to denote the set $\{a,\ldots,b\}$ and $(a;b]$ to denote the set $\{a+1,\ldots, b\}$. For any two positive integers $a,b$, let $\Mod(a;b)$ be the remainder of $a$ divided by $b$. For any two real numbers $a,b$, define $a\vee b\equiv \max\{a,b\}$ and $a\wedge b\equiv \min\{a,b\}$. For any positive integers $m \geq n$, let $\ldot(m;n)\equiv m(m-1)\ldots(m-n+1)$ and $\udot(m;n)\equiv m(m+1)\ldots(m+n-1)$.  Let $\mathbb{Z}_+$ denote the set of positive integers and $\mathbb{Z}_{\geq 0}\equiv \mathbb{Z}_{+}\cup\{0\}$. For any $d\in \mathbb{Z}_+$, let $\mathbb{S}^d\subset\RR^{d+1}$ stand for the unit sphere. We write $\E_{\theta_0}$ as expectation under the experiment \eqref{eq:model} with truth $\theta_0$.

Let $C^m([0,1])$ denote the set of all $m$-times differentiable functions on $[0,1]$. For any $f\in C^m([0,1])$ and integer $0\leq \ell \leq m$, let $f^{(0)}(x)\equiv f(x)$ and $(D^{(\ell)}f)(x) \equiv f^{(\ell)}(x)$ be the $\ell$-th derivative of $f$ at point $x$. For any function $f$ defined on $[0,1]$, $\tau\in[0,1]$, and real number $c$, define the first-order integral $(I^1_{c;\tau}f)(x) \equiv \int_\tau^x f(y)\ \d y + c$ for $x\in[0,1]$, and the  $m$-th order integral iteratively as $(I^m_{c_0,\ldots,c_{m-1};\tau}f)(x) \equiv \big(I^1_{c_0;\tau}(I^{m-1}_{c_1,\ldots,c_{m-1};\tau}f)\big)(x)$ for any positive integer $m\geq 2$ and real sequence $\{c_\ell\}_{\ell=0}^{m-1}$. For any real function $f$, let $f(x_-)$ and $f(x_+)$ denote the left and right limits at $x$, respectively. 

For two non-negative sequences $\{a_n\}$ and $\{b_n\}$, we write $a_n\lesssim_d b_n$ (resp. $a_n\gtrsim_d b_n$) if $a_n \leq Cb_n$ (resp. $a_n\geq cb_n$) for some $C,c>0$ that only depend on $d$. We also write $a_n\asymp_d b_n$ if both $a_n\lesssim_d b_n$ and $a_n\gtrsim_d b_n$ hold. In the following, we will suppress $d$ in $\lesssim_d$, $\gtrsim_d$, and $\asymp_d$ when no confusion is possible. For any given constants $a_1,a_2,\ldots$, we write $C(a_1,a_2,\ldots)$ and $c(a_1,a_2,\ldots)$ to denote positive constants that only depend on $a_1,a_2,\ldots$.

\section{General-order spline regression}\label{sec:gen_spline}

We start with an exact definition of the general-order spline space in \eqref{def:spline_fcn_intro}: 
\begin{align*}
&\cal{F}_n(d,d_0,k)\equiv \Big\{f:[0,1]\to \RR: \text{there exist } 0\equiv n_0\leq \ldots \leq n_k\equiv n \text{ such that }\\ 
&\qquad n_0,\ldots, n_k \in \mathbb{Z}_{\geq 0}, \qquad n_i - n_{i-1}\geq (d+1)\bm{1}_{n_i >  n_{i-1}},\\
&\qquad f\text{ is a $d$-degree polynomial on each interval }(n_{i-1}/n, n_i/n], \text{ and } \\
&\qquad f^{(\ell)}\big((n_i/n)_-\big) = f^{(\ell)}\big((n_i/n)_+\big) \text{ for all } i\in[1;k-1] \text{ and } \ell\in[0;d_0]\Big\}.
\end{align*}
For any fixed degree $d\geq 0$, the range of $d_0$ is $[-1;d-1]$, with $d_0 = -1$ allowing the spline $f$ to be completely discontinuous.  The numbers $n_0/n,\ldots,n_{k}/n$ are the \textit{knots} of $f$, with the middle $(k-1)$ ones as \textit{inner knots}. Define the corresponding sequence space
\begin{align}\label{eq:spline_seq}
\Theta_n(d,d_0,k) \equiv \Big\{\theta\in\RR^n: \theta_i = f(i/n)\text{ for some } f\in\cal{F}_n(d,d_0,k)\Big\};
\end{align}
in what follows, we suppress the subscript $n$ of $\Theta_n(d,d_0,k)$ when no confusion is possible and name $n_0,\ldots,n_k$ in its corresponding spline $f\in\cal{F}_n(d,d_0,k)$ the {\it knots} of $\theta$.

Two remarks regarding the above spline class are in line.
\begin{itemize}
    \item[(i)] The function space $\cal{F}_n(d,d_0,k)$ enforces the inner knots of the spline to be positioned among the design points. This is due to two reasons. First, it ensures the existence of the LSE as defined in \eqref{eq:lse} with $\Theta = \Theta(d,d_0,k)$. Indeed, the minimization can be first taken over at most $(n+1)^{k-1}$ configurations of the inner knots, after which the problem becomes strictly convex with respect to the rest of the polynomial coefficients and thus has a unique solution. Second, it facilitates fast computation of the LSE via dynamic programming algorithms; see \cite{fearnhead2019detecting} for detailed illustration of the piecewise linear case. 
    \item[(ii)] The gap $d+1$ between $n_i$ and $n_{i-1}$ in the above definition is necessary for the identifiability of $f$ in the discontinuous case. This minimum spacing condition improves substantially over existing ones made in a class of $\ell_1$ methods; see Remark \ref{remark:compare_l1} ahead for more details.  
\end{itemize}
For any fixed $d\in \mathbb{Z}_{\geq 0}$ and $d_0\in[-1;d-1]$, let
\begin{align}\label{eq:boundary}
k_0 \equiv k_0(d,d_0)\equiv \floor{\frac{d+1}{d-d_0}} + 1.
\end{align}

Our first main result is the following oracle inequality. Recall that we only consider the case $k\ge 2$ in this paper and the analysis of global polynomials (corresponding to $k=1$) is rather straightforward.
\begin{theorem}\label{thm:noshape_upper}
Fix any $\theta_0\in\RR^n$. Let $\widehat{\theta}\equiv \widehat{\theta}(\Theta(d,d_0,k),Y)$ be the LSE as defined in \eqref{eq:lse} under the experiment \eqref{eq:model} with truth $\theta_0$. Then, for any $\delta > 0$, there exists some $C = C(d,\delta)$ such that the following statements hold for any $n\geq \underline{n}$ with some $\underline{n} = \underline{n}(d)$. If $2\leq k\leq k_0$,
\begin{align*}
\E_{\theta_0}\|\widehat{\theta} - \theta_0\|^2 \leq (1+\delta)\inf_{\theta\in\Theta(d,d_0,k)}\|\theta - \theta_0\|^2 + C\sigma^2k\log\log(16n/k),
\end{align*}
and if $k\geq k_0+1$,
\begin{align*}
\E_{\theta_0}\|\widehat{\theta} - \theta_0\|^2 \leq (1+\delta)\inf_{\theta\in\Theta(d,d_0,k)}\|\theta - \theta_0\|^2 + C\sigma^2k\log(en/k).
\end{align*}
\end{theorem}
The following lower bound result shows that Theorem \ref{thm:noshape_upper} is optimal in the minimax sense. 
\begin{proposition}\label{prop:noshape_lower}
Under the experiment \eqref{eq:model}, there exists some $c = c(d)$ such that the following statements hold for all $n\geq \underline{n}$ with some $\underline{n} = \underline{n}(d)$. If $2\leq k\leq k_0$,
\begin{align*}
\inf_{\td{\theta}}\sup_{\theta\in\Theta(d,d_0,k)} \E_\theta\|\td{\theta} - \theta\|^2 \geq c\sigma^2k\log\log(16n/k),
\end{align*}
and if $k\geq k_0 + 1$,
\begin{align*}
\inf_{\td{\theta}}\sup_{\theta\in\Theta(d,d_0,k)} \E_\theta\|\td{\theta} - \theta\|^2 \geq c\sigma^2k\log(en/k),
\end{align*}
where the infimum over $\td{\theta}$ in both displays is taken over all measurable functions of $Y$.
\end{proposition}
The proof of Theorem \ref{thm:noshape_upper} is presented in Section  \ref{sec:noshape_upper}, and the proof of Proposition \ref{prop:noshape_lower} can be found in Appendix \ref{subsec:noshape_lower}.

\begin{remark}\label{remark:naive_upper}
The above two results imply, in particular, the minimax rates in \eqref{eq:minimax_noshape_intro}. There, the upper bound $k\log(en/k)$ above the transition boundary $k_0$ is not essentially new and can be proved via straightforward modifications of the classical arguments in, e.g., \cite{donoho1994minimax,birge2001gaussian}. Rather, our main contribution lies in establishing the sharp transition boundary $k_0$ and the faster $\log\log (16n)$ rate below this boundary. 

In practice when the number of pieces $k$ is unknown, the minimax rates in \eqref{eq:minimax_noshape_intro} provide guidance for penalty selection in the adaptive version \eqref{eq:pen_lse_def_intro} of the $\ell_0$-constrained spline LSE. Precisely, one can choose $\widehat{k}$ as in \eqref{eq:khat_intro} with the penalty 
\begin{align*}
\pen(k;d,d_0)\equiv \tau\sigma^2\bigg[ \bm{1}_{k=1} + k\log\log (16n/k)\cdot\bm{1}_{2\leq k\leq k_0} +  k\log(en/k)\cdot\bm{1}_{k > k_0} \bigg]
\end{align*}
for some sufficiently large universal $\tau > 0$. Then, standard arguments \cite{birge1993rates,barron1999risk, birge2001gaussian,massart2007concentration} guarantee that $\widehat{\theta}_{\ada}$ is adaptively minimax optimal over $\Theta(d,d_0,k)$ for all $k\in\mathbb{Z}_+$. Details are accordingly skipped.
\end{remark}

\begin{remark}\label{remark:compare_gao}

It is important to mention here one crucial difference between our perspective for the phase transition results and the $\log\log (16n)$ rates and the one taken in \cite{gao2017minimax}. There, the faster $\log\log (16n)$ rate for $\Theta(0,-1,2)$ follows immediately from the general iterated logarithmic rates for $\Theta^\ast(0,k)$, the class of piecewise constant and non-decreasing signals with at most $k$ pieces (formally defined in Section \ref{sec:gen_shape}). In other words, the $\log\log (16n)$ rate for $\Theta(0,-1,2)$ is perceived in \cite{gao2017minimax} as a consequence of the monotonicity shape constraint. In contrast, the $\log\log (16n)$ rate for $\Theta(d,d_0,k)$ in (\ref{eq:minimax_noshape_intro}) in the regime $k\leq k_0$ is inherited from the strong regularity in the signal parametrized by the degree $d$ and the level of continuity $d_0$, rather than any explicit shape constraint. In the regime $k>k_0$, the $\log\log (16n)$ rate is not possible due to insufficient regularity in $\Theta(d,d_0,k)$, unless additional shape constraints are enforced; see Section \ref{sec:gen_shape} ahead for more details.

\end{remark}

\begin{remark}\label{remark:compare_l1}
Recently, \cite{guntuboyina2020adaptive} studied the theoretical properties of trend filtering (TF), a class of $\ell_1$-regularized discrete spline methods. More precisely, under the experiment \eqref{eq:model}, the $d$-th order TF estimator is 
\begin{align}\label{eq:tf}
\widehat{\theta}_{\text{TF}}^d\equiv \min_{\theta\in\RR^n} \Big\{\|Y-\theta\|^2 + \lambda\|D^{(d)}\theta\|_1\Big\},
\end{align}
where $\|\cdot\|_1$ denotes the vector $\ell_1$ norm, $\lambda > 0$ is a tuning parameter, and $D^{(d)}:\RR^n\rightarrow\RR^{n-d}$, when applied to vectors, represents the $d$-th order discrete difference operator defined as $D^{(0)}\theta\equiv \theta$, $D^{(1)}\theta\equiv (\theta_2-\theta_1,\ldots,\theta_n-\theta_{n-1})^\top$, and $D^{(r)}\theta \equiv D^{(1)}(D^{(r-1)}\theta)$ for $r\geq 2$. Equation \eqref{eq:tf} is a convex problem and can be solved efficiently via algorithms designed for lasso-type problems \cite{tibshirani2014adaptive}.

For any $\theta_0\in\Theta(d,d-1,k)$ in \eqref{eq:model}, Corollary 2.11 in \cite{guntuboyina2020adaptive} proved that, upon choosing the tuning parameter $\lambda$ properly and assuming a minimum spacing condition to be detailed below, 
\begin{align}\label{eq:l1_bound}
\E_{\theta_0}\|\widehat{\theta}_{\text{TF}}^{d+1} - \theta_0\|^2 \leq C\sigma^2\Big(k\log(en/k) + k^{2(d+1)} \bm{1}_{d>0}\Big).
\end{align}
for some $C=C(d)$. Comparing \eqref{eq:l1_bound} with our Theorem \ref{thm:noshape_upper} and Proposition \ref{prop:noshape_lower}, we see
the following distinctions between $\ell_0$-regularized splines and their $\ell_1$ counterparts.
\begin{itemize}
    \item[(i)] The bound \eqref{eq:l1_bound} requires a  minimum spacing condition that regulates, for non-vanishing pieces $(n_i;n_{i+1}]$ between knots with different signs (see Page 210 of \cite{guntuboyina2020adaptive} for their definition for the signs of knots), $n_{i+1} - n_i\geq cn/k$ for some $c = c(d)$. This is stronger than the constant gap condition assumed in $\Theta(d,d_0,k)$. Moreover, Theorem 4.2 in \cite{fan2018approximate} suggests that this minimum spacing condition is essential to the TF estimators, namely, the performance of \eqref{eq:tf} could deteriorate to $\sqrt{n}$ (up to some polylogarithmic factors) without it. 
    \item[(ii)] Over the class $\Theta(d,d-1,k)$ with transition boundary $k_0 = d + 2$, the $\ell_1$ TF estimator in \eqref{eq:tf} is in general rate sub-optimal below the boundary, even with the additional minimum spacing condition mentioned above. Specifically, in the constant space $\Theta(0,-1,k)$, the minimax rate of estimation is $\log\log(16n)$ with $k = 2$ pieces, but the TF estimator in \eqref{eq:tf} with $d = 1$ can only achieve the slower $\log(en)$ rate in view of Lemma 2.4 of \cite{guntuboyina2020adaptive}. 
\end{itemize}
\end{remark}

\begin{remark}\label{remark:computation}
For the computation of the $\ell_0$-constrained spline LSE $\widehat{\theta}$ and its adaptive version \eqref{eq:pen_lse_def_intro}, the major difficulty in the development of efficient algorithms is measured by the regularity parameter $d_0$. For $d_0=-1$, both estimators can be computed efficiently using standard dynamic programming algorithms \cite{auger1989algorithms,winkler2002smoothers,jackson2005algorithm,friedrich2008complexity} along with more refined pruning arguments \cite{killick2012optimal,maidstone2017optimal}. For the first-order continuous case $(d_0 = 0)$, \cite{fearnhead2019detecting} recently introduced for the linear case ($d=1$) a novel dynamic programming algorithm with linear to quadratic time complexity, which can be readily extended to arbitrary order $d\in\mathbb{Z}_+$. We expect that the above method could potentially be extended to the case of general $d$ and $d_0$, but this will be left as the subject of future research. 
\end{remark}

Lastly, we provide some intuition for the form of $k_0$ defined in (\ref{eq:boundary}). This will mostly be clear from the perspective of minimax lower bounds in Proposition \ref{prop:noshape_lower}. There, the situation is somewhat similar to the derivation of minimax lower bounds in the sparse linear regression setting \cite{donoho1994minimax,ye2010rate,raskutti2011minimax}, in that we only have to find, for each fixed $d$ and $d_0$, the minimum value of $k$ such that a subset $S$ of $1$-sparse vectors can be constructed in $\Theta(d,d_0,k)$ with cardinality $|S|\geq cn$ for some $c = c(d)$. Heuristically, this value can be found via the following degree-of-freedom (DOF) calculation:
\begin{align}\label{eq:dof}
(k-2)(d+1) \geq (k - 1)(d_0 + 1) + 1.
\end{align}
Here, the left-hand side is the DOF for any $1$-sparse $\theta\in\Theta(d,d_0,k)$ with the two end pieces being constantly zero, as each of the middle $(k-2)$ pieces has $(d+1)$ DOF arising from the $d$-degree polynomial. On the right-hand side, the first term $(k-1)(d_0+1)$ results from the $(d_0+1)$ continuity constraints at each of the $(k-1)$ inner knots, and the additional $1$ DOF excludes the possibility that $\theta\equiv 0$. Solving \eqref{eq:dof} yields that $k\geq 1 + \ceil{(d+2)/(d - d_0)}$, which indeed holds for $k = k_0+1$ as defined in \eqref{eq:boundary}, with equality when $d_0 = d-1$.

Figure \ref{fig:sparse} demonstrates the minimum number $k$ of pieces needed for $d\in[0;2]$ and $d_0\in[-1;d-1]$ so that a $1$-sparse vector can be constructed in general position. The minimum value of $k$ in each scenario matches $k_0 + 1$ as defined in \eqref{eq:boundary}.

\begin{figure}[!ht]
\centering
\includegraphics[width=12cm]{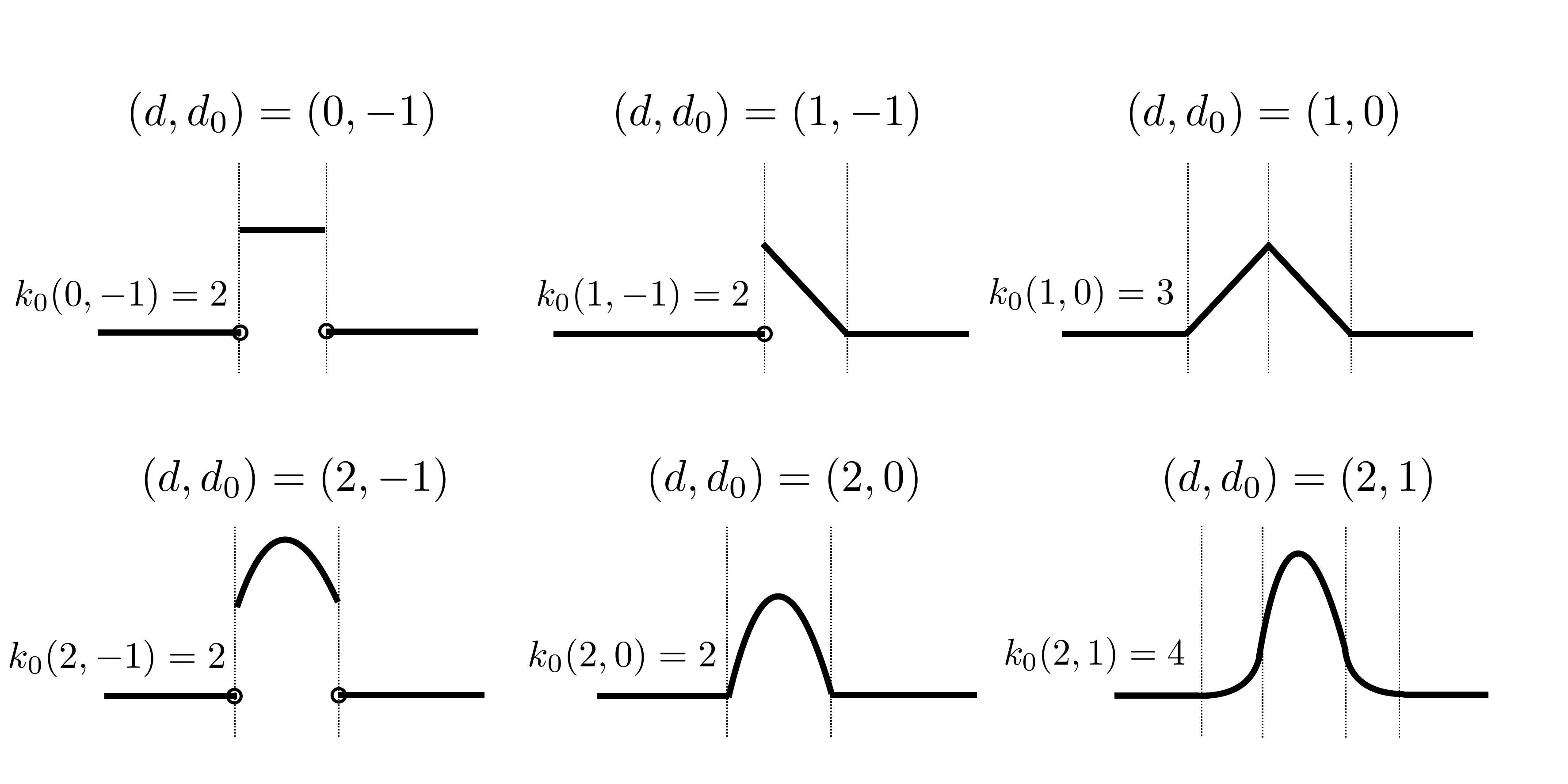}
\caption{Minimum number of $k = k_0+1$ pieces required to construct $1$-sparse vectors with general position in $\Theta(d,d_0,k)$ for $d\in[0;2]$ and $d_0\in[-1;d-1]$.}
\label{fig:sparse}
\end{figure}

\section{General-order splines with shape constraint}\label{sec:gen_shape}

As mentioned in \eqref{eq:minimax_GHZ_intro} in the introduction, in contrast to the phase transition in \eqref{eq:minimax_noshape_intro}, the faster $\log\log (16n)$ rate of estimation becomes universal in the class $\Theta^\ast(0,k)$ that contains all piecewise constant non-decreasing signals. This section derives higher-order analogues of this result. We start with the convexity constraint in the linear case in Section \ref{subsec:cvx_linear}, and then generalize to higher-order splines in Section \ref{subsec:shape_higher}. 

\subsection{Convex piecewise linear regression}\label{subsec:cvx_linear}

Convex regression is one of the central topics in shape constrained regression; see, e.g., \cite{guntuboyina2013global,chatterjee2015risk,bellec2018sharp} for global risk bounds and adaptation properties of the convex LSE. 

We start by defining the function space of convex piecewise linear functions: 
\begin{equation}\label{eq:cvx_linear_def}
\begin{aligned}
\cal{F}^*_n(1,k)\equiv\Big\{f\in\cal{F}_n(1,0,k): f\text{ has non-decreasing slopes on } [0,1]\Big\},
\end{aligned}
\end{equation}
and the space on the sequence level:
\begin{align}\label{eq:cvx_class}
\Theta_n^*(1,k)\equiv \{\theta^*\in\RR^n: \theta^*_i = f^*(i/n)\text{ for some } f^*\in\cal{F}^*_n(1,k)\},
\end{align} 
with the subscript $n$ in $\Theta_n^*(1,k)$ suppressed in the sequel. The following two results show that the convexity shape constraint eliminates the phase transition in $\Theta(1,0,k)$. 
 
\begin{proposition}\label{prop:cvx_linear}
Fix any $\theta_0\in\RR^n$.  Let $\widehat{\theta}^*\equiv\widehat{\theta}(\Theta^*(1,k),Y)$ be the LSE as defined in \eqref{eq:lse} under the experiment \eqref{eq:model} with truth $\theta_0$. Then, for any $\delta > 0$, there exists some $C = C(\delta)$ such that for any $n\geq \underline{n}$ with some universal $\underline{n}$ and $k\geq 2$,
\begin{align}\label{eq:cvx_linear}
\E_{\theta_0}\|\widehat{\theta}^* - \theta_0\|^2 \leq (1+\delta)\inf_{\theta^*\in\Theta^*(1,k)}\|\theta^* - \theta_0\|^2 + C\sigma^2k\log\log (16n/k).
\end{align}
\end{proposition}

\begin{proposition}\label{prop:cvx_linear_lower}
Under the experiment \eqref{eq:model}, there exists some universal constant $c$ such that for all $n\geq \underline{n}$ with some universal $\underline{n}$ and $k\geq 2$,
\begin{align*}
&\inf_{\td{\theta}^*}\sup_{\theta^*\in\Theta^*(1,k)} \E_{\theta^*}\|\td{\theta}^* - \theta^*\|^2 \geq c\sigma^2k\log\log(16n/k),
\end{align*}
where the infimum over $\td{\theta}^*$ is taken over all measurable functions of $Y$.
\end{proposition}

Proposition \ref{prop:cvx_linear} follows from its more general version in Theorem \ref{thm:shape_upper} ahead. The proof of Proposition \ref{prop:cvx_linear_lower} will be presented in Appedix \ref{subsec:shape_lower}.

\begin{remark}\label{remark:compare_bellec}
The in-expectation version of Theorem 4.3 in \cite{bellec2018sharp} proved a similar oracle inequality for the convex LSE:
\begin{align}\label{eq:bellec}
\E_{\theta_0}\|\widehat{\theta}(\Theta^*,Y) - \theta_0\|^2 \leq \inf_{\theta^*\in\Theta^*}\Big(\|\theta^* - \theta_0\|^2 + Ck(\theta^*)\log(en/k(\theta^*))\Big)
\end{align}
for some universal constant $C > 0$, where $\Theta^*\equiv \Theta^*(1,n)$ is the larger class of equispaced realizations of general convex functions on $[0,1]$, and $k(\theta^*)$ is the number of linear pieces of $\theta^*$, i.e., $k(\theta^*)\equiv \sum_{i=2}^n\bm{1}_{2\theta^*_i<\theta^*_{i-1} + \theta^*_{i+1}}$. Note that $\Theta^*$, as opposed to $\Theta^*(1,k)$, is a closed convex cone in $\RR^n$. The bounds \eqref{eq:cvx_linear} and \eqref{eq:bellec} are complementary in nature: the bound \eqref{eq:bellec} exploits the convexity of $\Theta^*$ to obtain a sharp oracle inequality (in the sense of leading constant $1$ before $\inf_{\theta^*\in\Theta^*}\|\theta^* - \theta_0\|^2$), but only achieves a slower worst-case $k\log(en/k)$ rate over the smaller class $\Theta^*(1,k)$; the bound \eqref{eq:cvx_linear}, or its adaptive version modified in a similar way as \eqref{eq:pen_lse_def_intro}, is minimax optimal over $\Theta^*(1,k)$ but loses the sharp leading constant $1$. 
\end{remark}

\subsection{General-order spline regression with shape constraint}\label{subsec:shape_higher}
Following \cite{balabdaoui2007estimation,chatterjee2015risk}, we consider the class of $d$-monotone splines defined as follows. Let
\begin{align*}
\cal{F}_n^*(0,k)\equiv \Big\{f\in \cal{F}_n(0,-1,k): f\text{ is non-decreasing on } [0,1]\Big\}
\end{align*}
be the $0$-monotone class. Next, for any $d\in\mathbb{Z}_+$, define
\begin{align*}
\cal{F}_n^*(d,k) \equiv \Big\{&f:[0,1]\to \RR: ~ f(x) = (I^{d}_{r_0,\ldots,r_{d-1};0}f_\circ)(x)\\
& \text{ for some } f_\circ\in\cal{F}_n^*(0,k) \text{ and real sequence } \{r_\ell\}_{\ell=0}^{d-1}\Big\}.
\end{align*}
Define the sequence version of the above space as
\begin{align}\label{eq:shape_class_seq}
\Theta_{n}^*(d,k)\equiv \Big\{\theta^*\in\RR^n: \theta^*_i = f^*(i/n) \text{ for some } f^*\in\cal{F}^*_n(d,k)\Big\},
\end{align}
shorthanded as $\Theta^*(d,k)$. One can readily check that for $d = 0$, $\Theta^*(0,k)$ is the class of $k$-piece isotonic signals studied in \cite{gao2017minimax}; for $d = 1$, $\Theta^*(1,k)$ coincides with the convex piecewise linear class in \eqref{eq:cvx_class}. Moreover, two facts follow immediately from the above definitions: (i) For any $d\geq 1$, $f^*\in \cal{F}^*_n(d,k)\subset C^{d-1}([0,1])$ so that $\Theta^*(d,k)\subset \Theta(d,d-1,k)$ with the latter defined in \eqref{eq:spline_seq}; (ii) For any $d\geq 1$ and $\ell\in[1;d]$, it holds that $(f^*)^{(\ell)}\in \cal{F}^*_n(d-\ell,k)$. 

The following result, with Proposition \ref{prop:cvx_linear} as a special case, shows that $d$-monotonicity eliminates the phase transition in the general spline space $\Theta(d,d-1,k)$. Its proof is given in Section \ref{sec:shape_upper}.
\begin{theorem}\label{thm:shape_upper}
Fix any $\theta_0\in\RR^n$. Let $\widehat{\theta}^*\equiv\widehat{\theta}(\Theta^*(d,k),Y)$ be the LSE as defined in \eqref{eq:lse} under the experiment \eqref{eq:model} with truth $\theta_0$. Then, for any $\delta > 0$, there exists some $C = C(d,\delta)$ such that for any $n\geq \underline{n}$ with some $\underline{n} = \underline{n}(d)$ and $k\geq 2$,
\begin{align*}
\E_{\theta_0}\|\widehat{\theta}^* - \theta_0\|^2 \leq (1+\delta)\inf_{\theta^*\in \Theta^\ast(d,k)}\|\theta^\ast - \theta_0\|^2 + C\sigma^2k\log\log(16n/k).
\end{align*}
Moreover, there exists some $c = c(d)$ such that for all $n\geq \underline{n}$ and $k\geq 2$,
\begin{align*}
\inf_{\td{\theta}^*}\sup_{\theta^*\in\Theta^*(d,k)} \E_{\theta^*}\|\td{\theta}^* - \theta^*\|^2 \geq c\sigma^2\log\log (16n),
\end{align*}
where the infimum over $\td{\theta}^*$ is taken over all measurable functions of $Y$.
\end{theorem}

\begin{remark}\label{remark:compare_GHZ}
The essential technical difficulties in proving Theorem \ref{thm:shape_upper} and Proposition \ref{prop:cvx_linear_lower} over the oracle inequality version of (\ref{eq:minimax_GHZ_intro}) (cf. Theorem 2.1 of \cite{gao2017minimax}) rest in the additional regularity of $\Theta^*(d,k)$ over $\Theta^*(0,k)$.

\begin{itemize}
    \item[(i)] For the upper bound, \cite{gao2017minimax} made essential use of the fact that $\widehat{\theta}(\Theta^*(0,k))$ is the sample average given the estimated knots; cf. Lemma 5.1 therein. The analogous property is, unfortunately, not true even for $\widehat{\theta}(\Theta^*(1,k))$. Instead, we provide a completely different proof which is based on a new parametrization for general-order splines with shape constraint (cf. Lemma \ref{lemma:shape_para} ahead). We further observe that this new proof technique, when applied to the setting of \cite{gao2017minimax}, significantly simplifies their proof; see Section \ref{subsec:shape_main} for details.
    \item[(ii)] For the lower bound in Proposition \ref{prop:cvx_linear_lower}, the continuity constraint in $\Theta^*(1,k)$ requires a much more delicate construction of least favorable signals that achieves the $k\log\log(16n/k)$ rate, compared to $\Theta^*(0,k)$; see Appendix \ref{subsec:shape_lower} for more details. This lower bound construction can actually be extended to yield the optimal $k\log\log(16n/k)$ rate over the quadratic class $\Theta^*(2,k)$, but a general lower bound of the order $k\log\log(16n/k)$ is still lacking for higher-order $d$-monotone splines.
\end{itemize}

\end{remark}

\section{A generalized law of iterated logarithm}\label{sec:lil}

In this section, we present a generalized law of iterated logarithm (LIL) in expectation that underlies the $\log\log (16n)$ rates derived in Sections \ref{sec:gen_spline} and \ref{sec:gen_shape}. Recall that a centered random variable $X$ is said to be \emph{sub-Gaussian} with parameter $\tau$, if there exists some $K > 0$ such that $\E\exp(\lambda X)\leq K\exp(\lambda^2\tau^2/2)$ for any $\lambda \in\RR$.

\begin{theorem}\label{thm:lil}
Fix positive integers $d\geq 1$ and $n\geq 2$. Let $\{\varepsilon_i\}_{i=1}^n$ be a sequence of independent and identically distributed centered sub-Gaussian random variables with parameter $1$. Let $\psi:\RR_+\rightarrow\RR_+$ be a strictly increasing continuous function with inverse $\psi^{-1}$. Let 
\begin{align*}
Z\equiv \max_{1\leq n_1 < n_2\leq n}\frac{\big|\sum_{i\in(n_1;n_2]} (i-n_1)^d\varepsilon_i\big|}{(n_2-n_1)^d(n_2\wedge (n-n_1))^{1/2}}.
\end{align*}
Then, provided that 
\begin{align}\label{cond:lil}
\int_1^\infty e^{-c_0(\psi^{-1}(t))^2}\ \d{t} < \infty
\end{align}
for some sufficiently small $c_0 = c_0(d)$, there exist some $C_1 = C_1(\psi,d)>0$ and $C_2=C_2(d)>0$ such that
\begin{align*}
\E\psi(Z)\leq C_1\big[\psi\big((C_2\log\log (16n))^{1/2}\big)\vee 1\big].
\end{align*}
\end{theorem}

The proof of the above theorem can be found in Appendix \ref{sec:proof_lil}. Here are some choices of $\psi$'s that will be relevant in the proofs of results in Sections \ref{sec:gen_spline} and \ref{sec:gen_shape}.

\begin{example}
Let $\psi(t)=t^\alpha$ where $\alpha>0$. Then $\psi^{-1}(t)=t^{1/\alpha}$, so clearly (\ref{cond:lil}) holds.
 \end{example}

\begin{example}
Let $\psi(t)=e^{c t^\alpha}-1$ where $\alpha,c>0$. Then $\psi^{-1}(t) = \big(\log(1+t)/c\big)^{1/\alpha}$. So 
\begin{align*}
\int_1^\infty e^{-(c_0/c^{2/\alpha}) (\log(1+t))^{2/\alpha} }\ \d{t} 
&
\begin{cases}
<\infty, & \alpha \in (0,2], c \in (0,c_0\bm{1}_{\alpha=2}+\infty\bm{1}_{\alpha \in (0,2)}),\\
=\infty, & \textrm{otherwise}.
\end{cases}
\end{align*}
Note that a law of iterated logarithm in expectation fails in general for the choice $\psi(t)=e^{ct^\alpha}-1$ whenever $\alpha>2$, as $\alpha=2$ corresponds to the maximal integrability of Gaussian random variables.
\end{example}

\section{Proof of Theorem \ref{thm:noshape_upper}}\label{sec:noshape_upper}

Starting from this section, unless otherwise specified, we will focus on the case $\sigma^2 = 1$; the extension to an arbitrary $\sigma^2 > 0$ is straightforward and hence not recorded here. We will also omit the proof for the $k\log(en/k)$ part of Theorem \ref{thm:noshape_upper} as it follows essentially from the classical arguments in \cite{donoho1994minimax,birge2001gaussian} by completely ignoring the regularity constraints. For the rest of the section, we focus on illustrating the form of $k_0$ in \eqref{eq:boundary} from the upper bound perspective and proving the faster $\log\log (16n)$ rate below the transition boundary. Section \ref{sec:intuition} provides a proof outline with illustrative simple cases discussed at first. Section \ref{sec:cw} reduces the proof of Theorem \ref{thm:noshape_upper} to the bound of complexity width in Proposition \ref{prop:width}. The key ingredients to the proof of this proposition will be presented in Sections \ref{subsec:noshape_ground} and \ref{subsec:noshape_key}, followed by the main proof in Section \ref{subsec:noshape_main}.

%The rest of this section is organized as follows. 

\subsection{Proof outline}\label{sec:intuition}

\subsubsection{Piecewise linear case}\label{sec:linear-case-S5}

We first consider the piecewise linear case $d=1, d_0 \in \{-1,0\}$, and assume $\theta_0=0$ in (\ref{eq:model}) for simplicity of discussion. Here, the transition boundary in \eqref{eq:boundary} is $k_0 =2$ for $d_0 = -1$ and  $k_0=3$ for $d_0 = 0$, beyond which the $\log\log(16n)$ rate cannot be attained. We focus on the case of $k=3$ pieces and illustrate the difference between $d_0 = -1$ and $d_0= 0$. To start, a standard reduction to complexity width in Proposition \ref{prop:reduction} ahead yields that for some universal constant $C>0$,
\begin{align*}
\E_{\theta_0}\pnorm{\widehat{\theta}-\theta_0}{}^2 -C \pnorm{\theta_\ora - \theta_0}{}^2 \leq C\cdot  \E\sup_{\theta\in\Theta(1,d_0,3):\pnorm{\theta}{}\leq 1}\big(\varepsilon \cdot \theta\big)^2 \equiv C\cdot  \E Z^2,
\end{align*}
where $\theta_\ora$ is any oracle in $\Theta(1,d_0,3)$ such that $\inf_{\theta\in\Theta(1,d_0,3)}\pnorm{\theta-\theta_0}{}^2$ is achieved, and $\E Z^2$ is termed the `complexity width' of $\Theta(1,d_0,3)$. To bound $\E Z^2$, we use the following parametrization for any given $f\in\cal{F}_n(1,d_0,3)$ with knots $0=n_0/n \leq n_1/n\leq n_2/n\leq n_3/n=1$: for $i \in \{0,1,2\}$,
%for any $x\in[0,1]$,
\begin{align}\label{eq:linear_para}
f(x) = a_i + b_i(x-n_{i}/n), \quad x\in\bigg(\frac{n_i}{n}, \frac{n_{i+1}}{n}\bigg].
%= \sum_{i=0}^{2}\big(a_i + b_i(x-n_{i}/n)\big)\bm{1}_{(n_{i}/n, n_{i+1}/n]}(x).
\end{align}
Under the additional continuity constraint when $d_0=0$, one has
\begin{align}\label{eq:linear_constraint}
a_1 = a_0 + b_0(n_1 - n_0)/n\quad \text{ and }\quad a_2 = a_1 + b_1(n_2 - n_1)/n.
\end{align}
Under the parametrization \eqref{eq:linear_para}, the supremum within the complexity width can be bounded by
\begin{align*}
 Z &\leq \sup_{\theta\in\Theta(1,d_0,3):\pnorm{\theta}{}\leq 1}\sum_{i=0}^{2}\bigg(|a_i| \biggabs{\sum_{j\in(n_i;n_{i+1}]}\varepsilon_j} + \biggabs{\frac{b_i}{n}} \biggabs{\sum_{j\in(n_i;n_{i+1}]}(j-n_i)\varepsilon_j}\bigg).
\end{align*}
The magnitudes of $\{a_i\}$ and $\{b_i\}$ can be drastically different for $d_0=-1$ and $d_0=0$. We illustrate this on the middle piece $(n_1;n_2]$.
\begin{itemize}
\item ($d_0=-1$). The constraint $1\geq \pnorm{\theta}{} \geq \pnorm{\theta}{(n_1; n_2]}$ directly yields the following estimates for $a_{1}$ and $b_{1}$ with some universal $C > 0$:
\begin{align}\label{eq:discon_est}
|a_{1}| \leq C(n_2 - n_1)^{-1/2} \quad \text{ and }\quad |b_{1}/n|\leq C(n_2-n_1)^{-3/2}.
\end{align}
Such estimates cannot be improved for, e.g., $f(x) = c(L^{-1/2} - nL^{-3/2}(x-1/2))\bm{1}_{(1/2,1/2+L/n]}(x)$ for small $c > 0$ and $L\geq 2$.
\item ($d_0=0$). With the additional continuity constraint in \eqref{eq:linear_constraint}, refined estimates can be obtained:
\begin{align}\label{eq:con_est}
|a_{1}| \leq Cn_2^{-1/2} \quad \text{ and }\quad |b_{1}/n|\leq C(n_2\wedge (n-n_1))^{-3/2}.
\end{align}
These estimates only hold up to $k = 3$ pieces. For $k\geq 4$, the best possible estimates are of type (\ref{eq:discon_est}) by considering, e.g., $f(x) = c\Big(nL^{3/2}\big(x-(1/2-L/n)\big)\bm{1}_{(1/2-L/n,1/2]}(x) - nL^{3/2}\big(x-(1/2+L/n)\big)\bm{1}_{(1/2,L/n+1/2]}(x)\Big)$ for small $c>0$ and $L\geq 2$. 
\end{itemize}
The crucial difference here is that estimates of type (\ref{eq:con_est}) enable a law of iterated logarithm (cf. Theorem \ref{thm:lil}) with $\E Z^2\lesssim \log\log (16n)$, while those of (\ref{eq:discon_est}) correspond to the maxima of $O(n)$ independent Gaussian random variables with $\E Z^2\lesssim \log (en)$.

\subsubsection{General case}\label{sec:general-case-S5}

Similar to the linear case discussed above, the key step is to prove
\begin{align}\label{eq:width}
\E\sup_{\theta\in\Theta(d,d_0,k_0),\|\theta\|\leq 1} \parr*{\varepsilon \cdot \theta}^2 \leq C\log\log (16n),
\end{align}
and we need to obtain estimates of type (\ref{eq:con_est}). For simplicity, we consider the smoothest case $d_0 = d-1$ so that $k_0 = d+2$. 

Fix a degree $d$, and any $f\in\cal{F}_n(d,d-1,d+2)$ along with the corresponding $\theta\in\Theta(d,d-1,d+2)$ of unit norm and knots $0=n_0\leq n_1\leq \ldots \leq n_{d+2} = n$. We use the following parametrization of $f$:
\begin{align}\label{eq:spline_para}
f(x) = \sum_{\ell=1}^{d+1}a^i_\ell\bigg(x-\frac{n_i}{n}\bigg)^{\ell-1},\quad x\in\bigg(\frac{n_i}{n}, \frac{n_{i+1}}{n}\bigg],
\end{align} 
and focus on a generic piece $(n_i;n_{i+1}]$ at the sequence level. Here the superscript $i$ represents `the $(i+1)$-th piece $(n_i;n_{i+1}]$' and the subscript $\ell$ represents `the $\ell$-th coefficient' in the polynomial. We aim at obtaining the following estimates:
\begin{align}\label{eq:demo_estimate}
1\geq c\cdot (a^i_\ell)^2((n_{i+1} - n_i)/n)^{2(\ell-1)}(n_{i+1}\wedge (n-n_i)),\quad \ell \in[1;d+1],
\end{align}
with some $c = c(d)$. Once these estimates are obtained, one can immediately apply Theorem \ref{thm:lil} to obtain a $\log\log (16n)$ bound on the complexity width on $(n_i;n_{i+1}]$.

In \eqref{eq:demo_estimate}, the $(d-1)$-th order differentiability at each inner knot naturally divides the coefficients into two groups, the `shared coefficients' $\{a_\ell^i\}_{\ell \in [1;d]}$ and the `nuisance coefficient' $a_{d+1}^i$. This suggests the following two-step proof strategy:
\begin{enumerate}
	\item[(i)] First, we show that the estimate for the second group, $a^i_{d+1}$, follows from that of the first group; cf. Lemma \ref{lemma:last_coef} ahead.
	\item[(ii)] Second, we obtain estimates in \eqref{eq:demo_estimate} for $\ell\in[1;d]$ with the choice $k_0=d+2$; cf. Lemma \ref{lemma:coef_est_1} ahead. 
\end{enumerate}
In the proof below, we will see clearly why $k_0 = d+2$ is the maximal number of pieces where the estimates in (\ref{eq:demo_estimate}) are achievable. At a high level, the coefficient estimates $\{a_\cdot^i\}$ on the piece $(n_i; n_{i+1}]$ necessarily depend on coefficient estimates at locations to the both sides of $i$. The passage of such information, for example from the rightmost knot, is precisely characterized in Lemma \ref{lemma:quad_forms} ahead through a set of quadratic forms, which are obtained via `iterative cancellation' to be detailed in Section \ref{subsec:noshape_ground}. The transition boundary $k_0$ is then determined via `counting of quadratic forms' (cf. (\ref{ineq:proof_boundary}) in the main proof ahead) that mirrors the DOF calculation in \eqref{eq:dof}, thereby unifying the heuristics in the upper and lower bounds.

\subsection{Reduction to complexity width}\label{sec:cw}

 We first introduce some notation. For any fixed $\theta_0\in \R^n$, let 
 $\theta_\ora\equiv \theta_\ora(\theta_0)\in\Theta(d,d_0,k)$ be an oracle such that $\inf_{\theta\in\Theta(d,d_0,k)}\pnorm{\theta - \theta_0}{}$ is achieved, 
 with knots $0 = n_0\leq n_1\leq \ldots\leq n_k = n$. For each $\theta\in\RR^n$, define $\theta_{[j]}$ as the sub-vector $(\theta_i)_{i\in(n_j;n_{j+1}]}$ and $v_j(\theta)\equiv v_j(\theta;\theta_\ora)\equiv (\theta-\theta_\ora)_{[j]}/\|(\theta-\theta_\ora)_{[j]}\|$.

The following result is a standard reduction principle for the LSE tailored to the class of splines. Its proof can be found in Appendix \ref{sec:proof_noshape}.

\begin{proposition}\label{prop:reduction}
Fix any $\theta_0\in\RR^n$.  Let $\widehat{\theta}\equiv \widehat{\theta}(\Theta(d,d_0,k),Y)$ be the LSE as defined in \eqref{eq:lse} under the experiment \eqref{eq:model} with truth $\theta_0$. Then, for any $\delta>0$, there exists some $C=C(\delta)>0$ such that
\begin{align*}
\E_{\theta_0}\|\widehat{\theta} - \theta_0\|^2 \leq (1+\delta)\|\theta_\ora - \theta_0\|^2 + C\cdot\E\sup_{\theta\in\Theta(d,d_0,k)}\sum_{j=0}^{k-1}\big(\varepsilon_{[j]}\cdot v_j(\theta)\big)^2.
\end{align*}
\end{proposition}

Now, note that each $v_j(\theta)$ is also a spline with unit norm and the same parameters $(d,d_0,k)$ (rigorously speaking, the two end pieces of $v_j(\theta)$ may have length smaller than $d+1$, but these pieces are negligible since there are at most $2k$ of them and each only contributes a constant (up to $d$) factor to the complexity width). Therefore, in view of Proposition \ref{prop:reduction}, the $\log\log (16n)$ part of Theorem \ref{thm:noshape_upper} for $k\leq k_0$ is immediately implied by the following result by noticing that $\Theta(d,d_0,k)\subset \Theta(d,d_0,k_0)$ for all $k\leq k_0$. 

\begin{proposition}\label{prop:width}
There exists some $C = C(d)$ such that
\begin{align*}
\E\sup_{\theta\in\Theta(d,d_0,k_0),\|\theta\|\leq 1} \parr*{\varepsilon \cdot \theta}^2 \leq C\log\log (16n).
\end{align*}
\end{proposition}
The following two subsections present the main ingredients to the proof of Proposition \ref{prop:width}, whose details will be presented in Section \ref{subsec:noshape_main}.

\subsection{Groundwork}\label{subsec:noshape_ground}

Fix any $f\in\cal{F}_n(d,d_0,k_0)$ with knots $0=n_0/n\leq n_1/n\leq \ldots\leq n_{k_0}/n =1$ and recall the parametrization (\ref{eq:spline_para}). Due to the regularity constraints, similar relations as the linear equations of the type \eqref{eq:linear_constraint} exist between adjacent knots. We use the notation $\co{a^i_p;a^{i-1}_q}$ to denote the coefficient of $a^{i-1}_q$ in the linear equation of $a^i_p$, i.e., $a^i_p = \sum_q \co{a^i_p;a^{i-1}_q} a^{i-1}_q$. The following lemma makes explicit this dependence. Its proof and proofs for other lemmas in this subsection are contained in Appendix \ref{sec:proof_noshape}. We write 
\begin{align}\label{eq:nxx}
n_{i;j}\equiv(n_i-n_j)/n.
\end{align}

\begin{lemma}\label{lemma:para_coef}
For any $i\in[1;k_0-1]$, $p\in[1;d_0+1]$, and $q\in[1;d+1]$,
\begin{align*}
\co{a^i_p; a^{i-1}_q} = {q-1\choose p-1}n_{i;i-1}^{q-p}\bm{1}_{q\geq p}.
\end{align*}
\end{lemma}

The next Lemma \ref{lemma:quad_forms} provides, as described in the proof outline in Section \ref{sec:general-case-S5}, the exact forms of the quadratic forms obtained by `iterative cancellation' from right. These quadratic forms lay the foundation for coefficient estimates of type \eqref{eq:demo_estimate}. For the rest of this section, we reserve the notation $s$ for the number of `iterative cancellation' performed. 

Before stating the general formulation in Lemma \ref{lemma:quad_forms}, we first present the illustrative case of cubic spline $(d = 3, k_0 = 5)$ in the sequence space with unit norm. We detail below the starting point $(s= 0)$ and the first two steps of cancellation $(s\in\{1,2\})$. Following the proof outline in Section \ref{sec:general-case-S5}, we separate the quadratic forms that only involve the `shared coefficients' $\{a_\ell^i\}_{\ell \in [1;3]}$ and those that also involve the `nuisance coefficient' $a_4^i$.
\begin{itemize}	
	\item ($s=0$). 
	The $\ell_2$ constraint on $(n_4;n_5]$ for the signal ($\|\theta\|_{(n_4;n_5]}\leq \|\theta\|= 1$) provides control on the following $4$ quadratic forms of length $1$:
	\begin{align*}
	1 \geq c\cdot\bigg[\Big\{(n-n_4)(a^4_1)^2 + \frac{(n-n_4)^3}{n^2}(a^4_2)^2 + \frac{(n-n_4)^5}{n^4}(a^4_3)^2\Big\}+\frac{(n-n_4)^7}{n^6}(a^4_4)^2\bigg].
	\end{align*}
	\item ($s=1$). For the first cancellation, we have, by Lemma \ref{lemma:para_coef},
	\begin{align}\label{ineq:general_sketch_2}
	\begin{pmatrix}
	a_1^4\\
	a_2^4\\
	a_3^4
	\end{pmatrix}
	= 
	\begin{pmatrix}
	1 & n_{4;3} & n_{4;3}^2 & n_{4;3}^3\\
	0 & 1 & 2n_{4;3} & 3n_{4;3}^2\\
	0 & 0 & 1 & 3n_{4;3}
	\end{pmatrix}
	\begin{pmatrix}
	a_1^3\\
	a_2^3\\
	a_3^3\\
	a_4^3
	\end{pmatrix}
	.
	\end{align}
	The identity \eqref{ineq:general_sketch_2} enables us to first find a linear combination of $(a^4_2,a^4_3)$ to cancel $a^3_4$, and then to find another linear combination of $(a^4_1,a^4_2,a^4_3)$ to cancel both $a^3_3$ and $a^3_4$. These, along with direct expansion of the term $(a^4_3)^2(n-n_4)^5/n^4$ using \eqref{ineq:general_sketch_2}, leave us with $3$ quadratic forms of length $2$:
	\begin{align*}
	1\geq c\cdot\bigg[\Big\{(n-n_4)\big(3a^3_1 + n_{4;3}a^3_2\big)^2 + \frac{(n-n_4)^3}{n^2}\big(a^3_2 + n_{4;3}a^3_3\big)^2\Big\} +\frac{(n-n_4)^5}{n^4}\big(a^3_3 + 3n_{4;3}a^3_4\big)^2 \bigg].
	\end{align*}
	\item ($s=2$). For the second cancellation, we have, by Lemma \ref{lemma:para_coef} again,
	\begin{align*}
	\begin{pmatrix}
	a_1^3\\
	a_2^3\\
	a_3^3
	\end{pmatrix}
	= 
	\begin{pmatrix}
	1 & n_{3;2} & n_{3;2}^2 & n_{3;2}^3\\
	0 & 1 & 2n_{3;2} & 3n_{3;2}^2\\
	0 & 0 & 1 & 3n_{3;2}
	\end{pmatrix}
	\begin{pmatrix}
	a_1^2\\
	a_2^2\\
	a_3^2\\
	a_4^2
	\end{pmatrix}
	.
	\end{align*}
	Then, finding a linear combination of $(a^3_1,a^3_2,a^3_3)$ to cancel $a^2_4$ and directly expanding $\big(a^3_2 + n_{4;3}a^3_3\big)^2(n-n_4)^3/n^2$, we obtain $2$ quadratic forms of length $3$:
	\begin{equation*}
	\begin{aligned}
	1 \geq c\cdot &\bigg[ (n-n_4)\bigg(3a^2_1 + (2n_{3;2} + n_{4;3})a^2_2 + (n_{3;2}^2+ n_{3;2}n_{4;3})a^2_3\bigg)^2 \\
	&+\frac{(n-n_4)^3}{n^2}\bigg(a^2_2 + (2n_{3;2} + n_{4;3})a^2_3 + (3n_{3;2}^2+ 3n_{3;2}n_{4;3})a^2_4\bigg)^2\bigg].
	\end{aligned}
	\end{equation*}
\end{itemize}

To state the above cancellation scheme for general $d$ and $d_0$, some further notation is introduced. Fix $d,d_0$, and the resulting $k_0$ as defined in \eqref{eq:boundary}. Define the sequence $\{\overline{\beta}^s_j\}$, $s\in[0;\floor{(d_0+1)/(d-d_0)}]$ recursively as follows. Let $\overline{\beta}^s_0\equiv 1$, 
\begin{align}\label{eq:beta_right}
\overline{\beta}^s_j &\equiv \sum_{\ell=0}^j \binom{s(d-d_0)-\ell}{j-\ell} n_{k_0-s;k_0-1-s}^{j-\ell} \overline{\beta}_{\ell}^{s-1}
\end{align}
for $j\in[1;s(d-d_0)]$, and $\overline{\beta}^s_j \equiv 0$ for $j > s(d-d_0)$. Further define, for every $i\in[1;(s+1)d_0 - sd + 1]$ and $j\in[0;s(d-d_0)]$,
\[
D(i,0)\equiv 1, \quad D(i,j) \equiv \frac{\overline{\odot}(i;j)}{\underline{\odot}(d+1-i;j)} \quad \text{ for }j\geq 1.
\]
Lastly, let $\overline{\beta}^s_{i,j}\equiv D(i,j)\overline{\beta}^s_j$.

We work under the extra condition that 
\begin{align}\label{cond:end_long}
n_{1;0}\wedge n_{k_0;k_0-1} \geq \max\{n_{2;1},\ldots,n_{k_0-1;k_0-2}\}.
\end{align}
We remark that condition (\ref{cond:end_long}) is made merely for presentational simplicity; see the comments after Lemma \ref{lemma:coef_est_1} ahead for detailed discussion of this condition. 

\begin{lemma}\label{lemma:quad_forms}
Suppose (\ref{cond:end_long}) holds. Fix $d$, $d_0$, and $k_0$ as defined in \eqref{eq:boundary}, and any $\theta\in\Theta(d,d_0,k_0)$ such that $\|\theta\|\leq 1$. Then, there exists some $c = c(d)$ such that, for any $s\in[0;\floor{(d_0+1)/(d-d_0)}]$,
\begin{equation}\label{eq:quad_forms_right}
\begin{aligned}
1&\geq c\bigg\{\sum_{i=1}^{(s+1)d_0-sd+1}+\sum_{i=(s+1)d_0-sd+2}^{sd_0-(s-1)d+1} \bigg\} \frac{(n-n_{k_0-1})^{2i-1}}{n^{2(i-1)}}\bigg(\sum_{j=0}^{s(d-d_0)} \overline{\beta}_{i,j}^s a_{i+j}^{k_0-1-s} \bigg)^2.
\end{aligned}
\end{equation}
\end{lemma}
\begin{remark}\label{remark:quad_forms}
Several remarks for the quadratic forms above are in order.
\begin{itemize}
    \item[(i)] The quadratic forms in \eqref{eq:quad_forms_right} are obtained via iterative cancellation from knot $n_{k_0-1}$. 
    \item[(ii)]  In a generic $\overline{\beta}^s_{i,j}$, the superscript $s$ marks the counts of cancellations already performed, $i$ indicates the $i$-th quadratic form, and $j$ indicates the coefficient for the $j$-th component in this quadratic form.
    \item[(iii)] In \eqref{eq:quad_forms_right}, we intentionally separate the indices $i\in[1;(s+1)d_0 - sd + 1]$ and $i\in[(s+1)d_0 - sd + 2;sd_0 - (s-1)d + 1]$ since the first set of quadratic forms only involves the `shared coefficients' $a^{\cdot}_j$ with $j\in[1;d_0+1]$.
    \item[(iv)] Every time $s$ grows by $1$, the first summand of \eqref{eq:quad_forms_right} has $(d-d_0)$ fewer quadratic forms with each one comprising of $(d-d_0)$ more components. 
\end{itemize}
\end{remark}

\subsection{Key estimates}\label{subsec:noshape_key}

Recall the coefficient sequence $\{a^i_\ell\}_{i\in[0;k_0-1],\ell\in[1;d+1]}$ defined in \eqref{eq:spline_para}. As described in Section \ref{sec:intuition}, we aim to obtain sharp estimates of type (\ref{eq:demo_estimate}). For any $a,b\in[1;n]$, define
\begin{align*}
M(a,b)\equiv (a\wedge (n-b))^{1/2}.
\end{align*}
The first result below reduces the task of obtaining (\ref{eq:demo_estimate}) for all the coefficients down to estimating only the `shared coefficients' $\{a^\cdot_\ell\}_{\ell\in[1;d_0+1]}$, from which the estimates for `nuisance coefficients' $\{a^\cdot_\ell\}_{\ell\in [d_0+2;d+1]}$ can be derived. Its proof can be found in Appendix \ref{sec:proof_noshape}.
\begin{lemma}\label{lemma:last_coef} 
Fix any $i\in[1;k_0-2]$. Suppose there exists some $c = c(d)$ such that for every $\ell\in[1;d_0+1]$, it holds that $1\geq c(a^i_{\ell})^2n_{i+1,i}^{2(\ell-1)}M^2(n_{i+1},n_i)$. Then, there exists some $c^\prime = c^\prime(d)$ such that
\begin{align*}
1\geq c^\prime(a^i_\ell)^2n_{i+1,i}^{2(\ell-1)}M^2(n_{i+1},n_i)
\end{align*}
for every $\ell\in [d_0+2; d+1]$.
\end{lemma}

 Following the preceding lemma, the next result, which builds on the groundwork derived in Lemma \ref{lemma:quad_forms}, makes use of an inductive argument to derive sharp estimates of the type \eqref{eq:demo_estimate} for $\{a^{i+1}_\ell\}_{\ell\in[1;d_0+1]}$ on a fixed target piece $(n_{i+1};n_{i+2}]$. To make the notation more accessible, we present here the special case $d_0 = d - 1$ (so that $k_0=d+2$) and defer the case of general $d_0$ to Appendix \ref{subsec:coef_est_1}.

\begin{lemma}\label{lemma:coef_est_1}
Suppose $d_0=d-1$ and (\ref{cond:end_long}) holds. Fix $i\in[0;d-1]$.  For some $c = c(d)$, the following estimates hold for all locations $1\leq j\leq i+1$:
\begin{align*}
1 &\geq c\max_{1\leq \ell\leq d} \bigg\{(a^j_\ell)^2\cdot  n_{i+2;j}^{2 \{(d-i) \wedge (\ell-1)\} } \cdot \bigg(\prod_{k=d-i+2}^{(d-j+2)\wedge \ell }n^2_{d+3-k; j}\bigg) \cdot n_{j+1;j}^{2(\ell-(d-j+2))_+} \cdot M^2(n_{j+1},n_{j}) \bigg\}.
\end{align*}
Here $\prod_{k=k_1}^{k_2}\equiv 1$ for $k_2<k_1$. In particular, for $j=i+1$:
\begin{align}\label{ineq:coef_est_prelim}
1 & \geq  c\max_{1\leq \ell\leq d} \bigg\{(a^{i+1}_\ell)^2 \cdot n_{i+2;i+1}^{2(\ell-1)} \cdot M^2(n_{i+2},n_{i+1}) \bigg\}.
\end{align}	
\end{lemma}
The proof of the above lemma is presented in the next subsection. We emphasize that the condition (\ref{cond:end_long}) is made only for presentational simplicity, as we explain below. If it does not hold, we can adopt the following partition of the pieces $\{(n_0;n_1],\ldots,(n_{d+1};n_{d+2}]\}$ via general length constraints. Fix a target piece $(n_{i+1};n_{i+2}]$ with $i\in[0;d-1]$. 
\begin{itemize}
\item[S1.] First locate among all pieces the longest one denoted as $(n_{i^*_1}; n_{i^*_1 + 1}]$ with $i_1^*\in[0;d+1]$. If this is the target piece, then we can directly apply Lemma \ref{lemma:diagonal} in Appendix \ref{sec:auxiliary} to this piece to obtain the desired estimates in \eqref{ineq:coef_est_prelim}.
\item[S2.] If not, assume without loss of generality that the target piece is to the left of this longest piece, i.e., $i+1<i^*_1$. Then, we can locate the longest piece among $\{(n_0;n_1],\ldots,(n_{i_1^*-1};n_{i_1^*}]\}$, which we denote as $(n_{i_2^*};n_{i_2^*+1}]$ with $i_2^*\in[0;i_1^*-1]$. If the target piece is among $\{(n_{i_2^*};n_{i_2^*+1}], \ldots, (n_{i_1^*-1};n_{i_1^*}]\}$, we can then make the following two modifications of Lemmas \ref{lemma:quad_forms} and \ref{lemma:coef_est_1}: (i) choose location $n_{i^*_1}$ (instead of the current $n_{d+1}$) as the starting point for the cancellation of the quadratic forms; (ii) choose location $i_2^*+1$ (instead of the current location $1$) as the starting point for the induction in Lemma \ref{lemma:coef_est_1}. These two modifications will yield the desired estimates for $\{a^{i+1}_\ell\}$ in \eqref{ineq:coef_est_prelim}. 
\item[S3.] If this is not the case, i.e., $(n_{i+1};n_{i+2}]\in\{(n_1;n_2], \ldots, (n_{i_2^*-1};n_{i_2^*}]\}$, we can then iterate S2 with $i_2^*$ in place of $i_1^*$. This partitioning will terminate in a finite number of steps. 
\end{itemize}
Condition \eqref{cond:end_long} (with $n_{1;0}\leq n_{d+2;d+1}$), along with the current versions of Lemmas \ref{lemma:quad_forms} and \ref{lemma:coef_est_1}, correspond to the above partitioning scheme with an early stop at S2 with $i_1^* = d+1$ and $i_2^* = 0$. On the other hand, condition \eqref{cond:end_long} represents the most difficult case in the sense that the maximal gap  $i_1^* - i_2^* = d+1$ activates the condition $k\leq k_0 =d+2$ as seen in \eqref{ineq:proof_boundary} in the proof ahead.

\subsection{Main proof}\label{subsec:noshape_main}
The main step in the proof of Proposition \ref{prop:width} is the set of coefficient estimates in Lemma \ref{lemma:coef_est_1}, with its more general version stated in Appendix \ref{subsec:coef_est_1}. We present the proof of this lemma in the special case $d_0 = d-1$; the proof for the general case is completely analogous. 

\begin{proof}[Proof of Lemma \ref{lemma:coef_est_1}]
Let
\begin{align*}
Q_{j}^2(\ell) = n_{i+2;j}^{2 \{(d-i) \wedge (\ell-1)\} } \cdot  \bigg(\prod_{k=d-i+2}^{(d-j+2)\wedge \ell }n^2_{d+3-k; j}\bigg) \cdot n_{j+1;j}^{2(\ell-(d-j+2))_+}.
\end{align*}
For the rest of the proof, empty $\prod$ is to be understood as $1$ and empty $\bigvee$ is to be understood as $0$. We will prove (a slightly stronger version with $M(n_1,n_0)$ instead of $M(n_{j+1},n_j)$)
\begin{align}\label{eq:stronger}
1 &\gtrsim \max_{1\leq \ell\leq d} \big\{ (a_\ell^j)^2 Q_{j}^2(\ell)\big\} \cdot M^2(n_1,n_0)
\end{align}
 by induction on $j\in [1;i+1]$. The baseline case $j = 1$ clearly holds by the condition \eqref{cond:end_long} and application of Lemma \ref{lemma:diagonal} to the piece $(n_0;n_1]$. Now, suppose the induction holds up to some location $j \in [1;i]$, and we will prove the iteration at location $j+1$.

 \noindent \textbf{(Part I).} We deal with $\{a^{j+1}_\ell\}_{\ell=1}^{d-j+1}$ in this part. For this, we first obtain estimates for $a^j_{d+1}$ and then use triangle inequality. Applying Lemma \ref{lemma:quad_forms} with $d_0 = d-1$ and $s = d-j$, the $j$-th term in the first summand therein yields that
\begin{align*}
1\gtrsim &\frac{(n-n_{d+1})^{2j-1}}{n^{2(j-1)}} \bigg(\sum_{\ell=0}^{d-j}  \overline{\beta}_{j,\ell}^{d-j} a_{j+\ell}^{j+1} \bigg)^2 \\
&= \frac{(n-n_{d+1})^{2j-1}}{n^{2(j-1)}} \bigg(\sum_{\ell=0}^{d-j}  \overline{\beta}_{j,\ell}^{d-j} \sum_{k= \ell}^{d-j+1} \binom{k+j-1}{\ell+j-1} n_{j+1;j}^{k-\ell}a_{k+j}^{j} \bigg)^2\\
&\equiv \frac{(n-n_{d+1})^{2j-1}}{n^{2(j-1)}} \bigg(\sum_{k=0}^{d-j+1} \bar{\gamma}_{j,k}^{d-j+1} a_{k+j}^{j} \bigg)^2,
\end{align*}
where we used Lemma \ref{lemma:para_coef} and $
\bar{\gamma}_{j,k}^{d-j+1} \equiv \sum_{q=0}^{(d-j)\wedge k}  \overline{\beta}_{j,q}^{d-j}  \binom{k+j-1}{q+j-1}n_{j+1;j}^{k-q}$. Note that for a generic number of $k$ pieces, when $j = 1$, we need to take $d_0 = d-1$ and $s = (k-1)-(j+1) = k - 3$ in Lemma \ref{lemma:quad_forms}, in which case the first summand is non-void if and only if 
\begin{align}\label{ineq:proof_boundary}
d - s  = d - k + 3 \geq 1 \iff k\leq d+2. 
\end{align}
This explains the transition boundary $k_0 = d+2$ as in \eqref{eq:boundary}. 

Combining the above estimate with the estimates for $\{a^j_k\}_{k=j}^d$ from the induction assumption, and using Lemma \ref{lemma:wedge} to cancel everything but $a_{d+1}^j$, we have
\begin{align*}
1&\gtrsim \frac{(n-n_{d+1})^{2j-1}}{n^{2(j-1)}} \bigg(\sum_{k=0}^{d-j+1} \bar{\gamma}_{j,k}^{d-j+1} a_{k+j}^{j} \bigg)^2 + \sum_{k=j}^d \bigg[(a^j_k)^2\cdot Q_j^2(k) \cdot M^2(n_1,n_0) \bigg]\\
&\gtrsim \big(a_{d+1}^j\big)^2 \bigg\{\frac{(n-n_{d+1})^{2j-1} (\bar{\gamma}_{j,d-j+1}^{d-j+1})^2}{n^{2(j-1)}}\wedge \bigwedge_{k=j}^d  \bigg[ Q_j^2(k) M^2(n_1,n_0)\frac{(\bar{\gamma}_{j,d-j+1}^{d-j+1})^2}{(\bar{\gamma}_{j,k-j}^{d-j+1})^2} \bigg]\bigg\}\\
&\equiv \big(a_{d+1}^j\big)^2\bigg\{A_j\wedge \bigwedge_{k=j}^d B_{j,k}\bigg\}.
\end{align*}
As $A_j/(\bar{\gamma}_{j,d-j+1}^{d-j+1})^2 = n_{d+2;d+1}^{2(j-1)}(n-n_{d+1}) \gtrsim B_{j,j}/(\bar{\gamma}_{j,d-j+1}^{d-j+1})^2$ by the assumption that the two end pieces are longer than any middle pieces, we only need to bound from below $\wedge_{k=j}^d B_{j,k}$. By definition of $\bar{\gamma}_{\cdot,\cdot}^\cdot$ and non-negativity of $\overline{\beta}^\cdot_{\cdot,\cdot}$, for any $j\leq k \leq d$,
\begin{align*}
\frac{(\bar{\gamma}_{j,d-j+1}^{d-j+1})^2}{(\bar{\gamma}_{j,k-j}^{d-j+1})^2}&\asymp \frac{ \big(\sum_{q=0}^{d-j}  \overline{\beta}_{j,q}^{d-j} n_{j+1;j}^{(d-j+1)-q}\big)^2 }{ \big(\sum_{q=0}^{k-j}  \overline{\beta}_{j,q}^{d-j} n_{j+1;j}^{k-j-q}\big)^2 }\quad (\textrm{by definition})\\
&\asymp  \bigvee_{p=0}^{d-k}\frac{ \bigvee_{q=0}^{k-j}  (\overline{\beta}_{j,p+q}^{d-j})^2 n_{j+1;j}^{2\{(d-j+1)-(p+q)\}} }{ \bigvee_{q=0}^{k-j}  (\overline{\beta}_{j,q}^{d-j})^2 n_{j+1;j}^{2(k-j-q)} }\quad (\textrm{by rearranging the numerator}) \\
&\geq \bigvee_{p=0}^{d-k} \bigg\{n_{j+1;j}^{2(d-k+1-p)}\bigwedge_{q=0}^{k-j} \bigg(\frac{\overline{\beta}_{j,p+q}^{d-j}}{\overline{\beta}_{j,q}^{d-j}}\bigg)^2 \bigg\} \quad (\textrm{by Lemma \ref{lemma:wedge}})\\
&\gtrsim  \bigvee_{p=0}^{d-k} \bigg\{n_{j+1;j}^{2(d-k+1-p)}\bigwedge_{q=0}^{k-j} \prod_{r=1+q}^{p+q}n_{d+2-r;j+1}^2 \bigg\} \quad (\textrm{by Lemma \ref{lemma:property_beta}})\\
& = \bigvee_{p=0}^{d-k} \bigg\{n_{j+1;j}^{2(d-k+1-p)} \prod_{r=1+k-j}^{p+k-j}n_{d+2-r;j+1}^2 \bigg\} \quad (\textrm{minimum at }q=k-j).
\end{align*}
Hence
\begin{align*}
1\gtrsim (a_{d+1}^j)^2 \bigg[\bigwedge_{k=j}^d Q_j^2(k) \bigvee_{p=0}^{d-k} \bigg\{n_{j+1;j}^{2(d-k+1-p)} \prod_{r=1+k-j}^{p+k-j}n_{d+2-r;j+1}^2 \bigg\}\bigg] M^2(n_1,n_0).
\end{align*}
This implies that for $1\leq \ell \leq d$, by taking $p=(\ell-k)_+$ above and Lemma \ref{lemma:para_coef}, 
\begin{align*}
& M(n_1,n_0) |a^{j+1}_\ell|\lesssim M(n_1,n_0)\bigg[\sum_{k = \ell}^d n_{j+1;j}^{k-\ell}|a^j_k| + n_{j+1;j}^{d+1-\ell}|a^j_{d+1}|\bigg]\\
& \lesssim \sum_{k=\ell}^d n_{j+1;j}^{k-\ell} Q_j^{-1}(k) + \bigvee_{k=j}^d \bigg\{Q_j^{-1}(k)  n_{j+1;j}^{k-\ell+(\ell-k)_+} \prod_{r=1}^{(\ell-k)_+ }n_{d+2+j-k-r;j+1}^{-1} \bigg\}\\
& = \sum_{k=\ell}^d n_{j+1;j}^{k-\ell} Q_j^{-1}(k) + \bigvee_{j\leq k< \ell\vee j} \bigg\{Q_j^{-1}(k)  \prod_{r=1}^{\ell\vee j-k }n_{d+2+j-k-r;j+1}^{-1} \bigg\}+ \bigvee_{\ell\vee j\leq k\leq d} \bigg\{Q_j^{-1}(k) n^{k-\ell}_{j+1;j} \bigg\}.
\end{align*}
Using that $k\mapsto n_{j+1;j}^{k-\ell } Q_j^{-1}(k)$ is non-increasing, the first and third terms in the above display are on the same order as $Q_j^{-1}(\ell) +  Q_j^{-1}(\ell \vee j) n_{j+1;j}^{\ell \vee j -\ell} \asymp Q_j^{-1}(\ell)$. Hence we only need to verify for all $1\leq \ell\leq d-j+1$, $1\leq j\leq i$,
\begin{align}\label{ineq:general_i_1}
\mathfrak{Q}_{j,1}(\ell)+ \mathfrak{Q}_{j,2}(\ell) \equiv  Q_j^{-1}(\ell) + \bigvee_{j\leq k< \ell\vee j} \bigg\{Q_j^{-1}(k)  \prod_{r=1}^{\ell\vee j-k }n_{d+2+j-k-r;j+1}^{-1} \bigg\}  \lesssim Q_{j+1}^{-1}(\ell).
\end{align}
\noindent \textbf{(Case 1).} If $1\leq \ell\leq d-i+1$, $Q_j^{-1}(\ell)= n_{i+2;j}^{-(\ell-1)}$ and $Q_j^{-1}(\ell)= n_{i+2;j+1}^{-(\ell-1)}$, so:
\begin{itemize}
    \item (first term) $\mathfrak{Q}_{j,1}(\ell) = n_{i+2;j}^{-(\ell-1)} \leq n_{i+2;j+1}^{-(\ell-1)} = Q_{j+1}^{-1}(\ell)$.
    \item (second term) without loss of generality we assume $\ell>j$ (otherwise this term does not exist):
    \begin{align*}
    \mathfrak{Q}_{j,2}(\ell)&= \bigvee_{j\leq k< \ell} \bigg\{n_{i+2;j}^{-(k-1)} \prod_{r=1}^{\ell-k }n_{d+2+j-k-r;j+1}^{-1} \bigg\}\\
    &\leq \bigvee_{j\leq k< \ell} \bigg\{n_{i+2;j}^{-(k-1)} n_{d+2+j-\ell;j+1}^{-(\ell-k)} \bigg\}\\
    &\leq \bigvee_{j\leq k< \ell} \bigg\{n_{i+2;j+1}^{-(k-1)} n_{i+2;j+1}^{-(\ell-k)} \bigg\} = n_{i+2;j+1}^{-(\ell-1)}= Q_{j+1}^{-1}(\ell),
    \end{align*}
    where the first equality follows since $k<\ell\leq d-i+1$ so that $Q^{-1}_{j}(k)= n_{i+2;j}^{-(k-1)}$, and the second inequality follows by noting that $\ell\leq d-i+1$ implies $d+2+j-\ell \geq i+2$. 
\end{itemize}
\noindent \textbf{(Case 2).} If $d-i+2\leq \ell\leq d-j+1$, $Q_j^{-1}(\ell)= n_{i+2;j}^{-(d-i)}\prod_{s =d-i+2}^\ell n_{d+3-s;j}^{-1}$ and $Q_{j+1}^{-1}(\ell)= n_{i+2;j+1}^{-(d-i)}\prod_{s =d-i+2}^\ell n_{d+3-s;j+1}^{-1}$, so:
\begin{itemize}
    \item (first term) similarly as above, 
    \begin{align*}
    \mathfrak{Q}_{j,1}(\ell) &= n_{i+2;j}^{-(d-i)}\prod_{s =d-i+2}^\ell n_{d+3-s;j}^{-1} \leq n_{i+2;j+1}^{-(d-i)}\prod_{s =d-i+2}^\ell n_{d+3-s;j+1}^{-1} = Q_{j+1}^{-1}(\ell).
    \end{align*}
    \item (second term) similarly as above we assume $\ell>j$, then
    \begin{align*}
    \mathfrak{Q}_{j,2}(\ell)&= \bigvee_{j\leq k<\ell} \bigg\{ n_{i+2;j}^{-(d-i)}\prod_{s =d-i+2}^k n_{d+3-s;j}^{-1}  \prod_{r=1}^{\ell-k }n_{d+2+j-k-r;j+1}^{-1}  \bigg\}\\
    &\leq n_{i+2;j+1}^{-(d-i)}  \bigvee_{j\leq k<\ell} \bigg\{\prod_{u=d+3-k}^{i+1} n^{-1}_{u;j+1} \prod_{u=d+2+j-\ell}^{d+1+j-k} n^{-1}_{u;j+1} \bigg\}.
    \end{align*}
    Note that $\prod_{u=d+2+j-\ell}^{d+1+j-k} n^{-1}_{u;j+1} \leq \prod_{u=d+2+j-\ell-(j-1)}^{d+1+j-k-(j-1)} n^{-1}_{u;j+1} = \prod_{u = d+3-\ell}^{d+2-k}n^{-1}_{u;j+1}$, where the inequality follows by $j\geq 1$ and $\ell \leq d-j+1$, so the above display can be further bounded by 
    \begin{align*}
    \mathfrak{Q}_{j,2}(\ell)\leq n_{i+2;j+1}^{-(d-i)} \prod_{u = d+3-\ell}^{i+1} n_{u;j+1}^{-1} = Q_{j+1}^{-1}(\ell).
    \end{align*}
\end{itemize}
Hence (\ref{ineq:general_i_1}) is verified and we have finished the proof for Part I.

\noindent \textbf{(Part II).} We deal with $\{a^{j+1}_\ell\}_{\ell= d-j+2}^d$ in this step. Applying Lemma \ref{lemma:quad_forms} with $d_0 = d-1$ and $s = d-j$, the last $(j-1)$ terms in the first summand therein take the form
\begin{align*}
1&\gtrsim \frac{(n-n_{d+1})^{3}}{n^2}\Big(\overline{\beta}^{d-j}_{2,0}a^{j+1}_2  + \ldots + \overline{\beta}^{d-j}_{2,d-j}a^{j+1}_{d-j+2}\Big)^2 \tag{R.$2$}\\
&+ \frac{(n-n_{d+1})^5}{n^4}\Big(\overline{\beta}^{d-j}_{3,0}a^{j+1}_3 +  \ldots + \overline{\beta}^{d-j}_{3,d-j}a^{j+1}_{d-j+3}\Big)^2\tag{R.$3$}\\
& \ldots\\
&+\frac{(n-n_{d+1})^{2j-1}}{n^{2(j-1)}}\Big(\overline{\beta}^{d-j}_{j,0}a^{j+1}_j + \ldots + \overline{\beta}^{d-j}_{j,d-j}a^{j+1}_d\Big)^2\tag{R.$j$}.
\end{align*}
Combining (R.$2$) with the estimates for $\{a^{j+1}_{\ell}\}_{\ell = 2}^{d-j+1}$ obtained in Part I, and using Lemma \ref{lemma:two_quad} iteratively to cancel everything but $a^{j+1}_{d-j+2}$, we obtain
\begin{align*}
1 &\gtrsim \frac{(n-n_{d+1})^3}{n^{2}}\bigg(\sum_{k=0}^{d-j}\overline{\beta}^{d-j}_{2,k}a^{j+1}_{k+2}\bigg)^2 + \sum_{k =2}^{d-j+1} \bigg[(a_k^{j+1})^2\cdot Q_{j+1}^2(k)\cdot M^2(n_1,n_0)\bigg]\\
&\gtrsim \big(a_{d-j+2}^{j+1}\big)^2 \bigg\{ \frac{(n-n_{d+1})^3 (\bar{\beta}_{2,d-j}^{d-j})^2}{n^{2}} \wedge \bigwedge_{k=2}^{d-j+1}\bigg[Q_{j+1}^2(k)M^2(n_1,n_0)\frac{(\bar{\beta}_{2,d-j}^{d-j})^2}{(\bar{\beta}_{2,k-2}^{d-j})^2 }\bigg] \bigg\}\\
&\equiv \big(a_{d-j+2}^{j+1}\big)^2\bigg\{A_j^{(2)}\wedge \bigwedge_{k=2}^{d-j+1} B_{j,k}^{(2)}\bigg\}.
\end{align*}
Similar to Part I, we only need to get a lower bound for $\bigwedge_{k=2}^{d-j+1} B_{j,k}^{(2)}$. As 
$(\bar{\beta}_{2,d-j}^{d-j})^2/(\bar{\beta}_{2,k-2}^{d-j})^2\gtrsim \prod_{r=k-1}^{d-j}n_{d+2-r;j+1}^2$ by Lemma \ref{lemma:property_beta}, it follows that 
\begin{align*}
1&\gtrsim \big(a_{d-j+2}^{j+1}\big)^2 \bigwedge_{k=2}^{d-j+1}\bigg[Q_{j+1}^2(k)\prod_{r=k-1}^{d-j}n_{d+2-r;j+1}^2 \bigg] M^2(n_1,n_0).
\end{align*}
As $k\mapsto Q_{j+1}^2(k)\prod_{r=k-1}^{d-j}n_{d+2-r;j+1}^2 =Q_{j+1}^2(d-j+1) n_{d+3-k;j+1}^2$ is non-increasing on $k \in [2;d-j+1]$, the minimum is taken at $k=d-j+1$ in the above display. Since $Q_{j+1}^2(d-j+1)n_{j+2;j+1}^2 = Q_{j+1}^2(d-j+2)$, we arrive at
\begin{align*}
1&\gtrsim \big(a_{d-j+2}^{j+1}\big)^2 Q_{j+1}^2(d-j+2) M^2(n_1,n_0),
\end{align*}
which is the desired estimate for $a_{d-j+2}^{j+1}$. Now iterate along (R.$3$)-(R.$j$) to complete the proof for Part II. This completes the proof.
\end{proof}

\begin{proof}[Proof of Proposition \ref{prop:width}]
We shorthand $\Theta(d,d_0,k_0)$ as $\Theta$, and the sample points will be indexed using $\iota$. For any $\theta\in\Theta$, let $\{n_j\}_{j=0}^{k_0}$ be its knots: $0 = n_0\leq n_1\leq\ldots\leq n_{k_0} = n$. The overall complexity width can then be bounded piece by piece:
\begin{align*}
\E\sup_{\theta\in\Theta}\parr*{\varepsilon\cdot \theta}^2 &= \E\sup_{\theta\in\Theta}\bigg(\sum_{i=1}^{k_0}\parr*{\varepsilon\cdot \theta}_{(n_{i-1};n_i]}\bigg)^2 \leq C\sum_{i=1}^{k_0}\E\sup_{\theta\in\Theta}\parr*{\varepsilon\cdot \theta}_{(n_{i-1};n_i]}^2.
\end{align*}
We will prove that each summand in the above display can be bounded by a constant multiple of $\log\log (16n)$.

We start with the first piece $(n_0;n_1]$. Let $f\in\cal{F}_n(d,d_0,k)$ be a generating spline of $\theta$, i.e., $\theta_\iota = f(\iota/n)$ for $\iota\in[1;n]$. For this piece, we use the following parametrization of $f(\cdot)$ slightly different from \eqref{eq:spline_para}: for any $x\in(0,n_1/n]$,
\begin{align}\label{eq:para_p1}
f(x)= \sum_{\ell=1}^{d_0+1}\td{a}^1_\ell\bigg(x-\frac{n_1}{n}\bigg)^{\ell-1} + \sum_{\ell=d_0+2}^{d+1}a^0_{\ell}\bigg(x-\frac{n_1}{n}\bigg)^{\ell-1}.
\end{align}
Then, the complexity width in question can be written as
\begin{align*}
(\varepsilon\cdot \theta)_{(n_0;n_1]}=\sum_{\ell=1}^{d_0+1}\sum_{\iota\in(n_0;n_1]}\td{a}^1_\ell\bigg(\frac{\iota-n_1}{n}\bigg)^{\ell-1}\varepsilon_{\iota} + \sum_{\ell=d_0+2}^{d+1}\sum_{\iota\in(n_0;n_1]}a^0_{\ell}\bigg(\frac{\iota-n_1}{n}\bigg)^{\ell-1}\varepsilon_{\iota}.
\end{align*}
Applying Lemma \ref{lemma:diagonal} to the piece $(n_0;n_1]$, we have
\begin{align}\label{eq:estimate_p1}
\sum_{\ell=1}^{d_0+1}(\td{a}^1_\ell)^2\frac{n_1^{2\ell-1}}{n^{2(\ell-1)}} + \sum_{\ell=d_0+2}^{d+1}(a^0_\ell)^2\frac{n_1^{2\ell-1}}{n^{2(\ell-1)}}\lesssim 1.
\end{align}
Thus the complexity width over the first piece $(n_0;n_1]$ can be bounded by
\begin{align*}
&\E\sup_{\theta\in\Theta}\parr*{\varepsilon\cdot \theta}^2_{(n_0;n_1]}\\
&\lesssim \sum_{\ell=1}^{d_0+1} \E\sup_{1\leq n_1\leq n}\sup_{(\td{a}^1_\ell)^2\frac{n_1^{2\ell-1}}{n^{2(\ell-1)}} \leq 1}\frac{(\td{a}^1_\ell)^2}{n^{2(\ell-1)}}\bigg(\sum_{\iota\in(n_0;n_1]}(\iota-n_1)^{\ell-1}\varepsilon_{\iota}\bigg)^2 \\
&+\sum_{\ell=d_0+2}^{d+1}\E\sup_{1\leq n_1\leq n}\sup_{(a^0_\ell)^2\frac{n_1^{2\ell-1}}{n^{2(\ell-1)}} \leq 1}\frac{(a^0_\ell)^2}{n^{2(\ell-1)}}\bigg(\sum_{\iota\in(n_0;n_1]}(\iota-n_1)^{\ell-1}\varepsilon_{\iota}\bigg)^2 \\
&\leq C\log\log (16n),
\end{align*}
where the second inequality is due to Theorem \ref{thm:lil} with $\psi(x) = x^2$ therein. The complexity width over the last piece $(n_{k_0-1};n_{k_0}]$ can be handled similarly.

Starting from the second until the second last piece, we use the parametrization \eqref{eq:spline_para} on the piece $(n_{i+1};n_{i+2}]$, yielding
\begin{align*}
(\varepsilon\cdot \theta)_{(n_{i+1};n_{i+2}]}=\sum_{\ell=1}^{d+1}\sum_{\iota\in(n_{i+1};n_{i+2}]} a^{i+1}_\ell\bigg(\frac{\iota-n_{i+1}}{n}\bigg)^{\ell-1}\varepsilon_{\iota}.
\end{align*}
Thus the complexity width in question can be bounded by
\begin{align*}
&\E\sup_{\theta\in\Theta}\parr*{\varepsilon\cdot \theta}^2_{(n_{i+1};n_{i+2}]} \lesssim \sum_{\ell=1}^{d+1} \E\sup_{\theta\in\Theta} \frac{(a_\ell^{i+1})^2}{n^{2(\ell-1)}}\bigg(\sum_{\iota\in(n_{i+1};n_{i+2}]}(\iota-n_{i+1})^{\ell-1}\varepsilon_{\iota}\bigg)^2\\
&\lesssim \sum_{\ell=1}^{d+1}\E\sup_{ \substack{n_{i+1}<n_{i+2},\\ (a^{i+1}_\ell)^2n_{i+2;i+1}^{2(\ell-1)}M^2(n_{i+2},n_{i+1}) \leq 1 } }\frac{\parr*{a_\ell^{i+1}}^2}{n^{2(\ell-1)}}\bigg(\sum_{\iota\in(n_{i+1};n_{i+2}]}(\iota-n_{i+1})^{\ell-1}\varepsilon_{\iota}\bigg)^2\\
&\leq C\log\log (16n),
\end{align*}
where the second inequality is by plugging in the estimates $a^{i+1}_{\ell}$, $\ell\in[1;d+1]$ from Lemma \ref{lemma:coef_est_1-general} (the general version of Lemma \ref{lemma:coef_est_1} with general $d_0\in[-1;d-1]$), and the third inequality is by applying Theorem \ref{thm:lil} with $\psi(x) = x^2$ therein. The proof is thus complete.
\end{proof}

\section{Proof of Theorem \ref{thm:shape_upper}, upper bound}\label{sec:shape_upper}

\subsection{Proof outline}

For expository purpose, we focus on the convex linear case $\Theta^*(1,k)$ with truth $\theta_0 = 0$ in \eqref{eq:model}. Using the reduction Proposition \ref{prop:reduction}, the key ingredient is to show
\begin{align}\label{eq:cvx_linear_complexity}
\E\sup_{\theta^*\in\Theta^*(1,k):\pnorm{\theta^*}{}\leq 1}\big(\varepsilon \cdot \theta^*\big)^2\leq C \log\log (16n).
\end{align}
To control the complexity width, we may parametrize any $\theta^\ast \in \Theta^\ast(1,k)$ by 
\begin{align}\label{eq:linear_canonical}
\theta^*_i = c_0 + \sum_{j=1}^{j^*}a_j\bigg(\frac{n_j - i}{n}\bigg)_+ + \sum_{j = j^*}^{k-1}b_j\bigg(\frac{i-n_j}{n}\bigg)_+,
\end{align}
where
\begin{itemize}
    \item $j^*$ is the index of the knot where the slope of the underlying convex function $f^*$ crosses zero if it does, and is otherwise set to be $k$; 
    \item $\{a_j\}$ and $\{b_j\}$ are two \emph{non-negative} real sequences parametrizing the \emph{change of slope}, in the two regions where $f^\ast$ has negative and positive slopes, respectively.
\end{itemize}

With the parametrization (\ref{eq:linear_canonical}), proving (\ref{eq:cvx_linear_complexity}) then reduces to obtaining sharp estimates for $\{a_j\}, \{b_j\}$, and $c_0$. These estimates are obtained in rather different ways:
\begin{itemize}
    \item For the coefficients $\{a_j\},\{b_j\}$, the non-negativity property turns out to be the key in obtaining sharp estimates for their magnitudes. Combined with the LIL (cf. Theorem \ref{thm:lil}), these coefficients contribute the desired $\log\log (16n)$ factor to the complexity width (\ref{eq:cvx_linear_complexity}).
    \item For the coefficient $c_0$, an \emph{a priori} estimate $|c_0|\leq C/\sqrt{n}$ is obtained (cf. Lemma \ref{lemma:shape_coef}) under the assumed (convexity) shape constraint and the $\ell_2$ constraint on the signal. This means that the coefficient $c_0$ only contributes a constant factor to the complexity width (\ref{eq:cvx_linear_complexity}).
\end{itemize}

It should be noted that for the larger class $\Theta(1,0,k)$ without the convexity shape constraint, a parametrization in the form of (\ref{eq:linear_canonical}) still holds but \emph{without the non-negativity constraint on $\{a_j\},\{b_j\}$}. The lack of such sign constraints unfortunately makes this representation not quite useful in obtaining LIL for $\Theta(1,0,3)$, so a different representation (cf. (\ref{eq:linear_para})) and a different proof strategy (cf. Section \ref{sec:intuition}) are adopted for $\Theta(1,0,3)$.

\subsection{Groundwork}\label{subsec:shape_ground}

The first result establishes a canonical parametrization for general-order $d$-monotone splines. By definition, the polynomial coefficient of the highest order for a $d$-monotone spline is increasing and thus crosses zero at most once. In the following parametrization, we choose this cross point as the pivot.
\begin{lemma}\label{lemma:shape_para}
For any $f^*\in\cal{F}_n^*(d,k)$, there exists some integer $j^*\in[0;k]$ and real sequences $\{a_j\}_{j=1}^{j^*}$, $\{b_j\}_{j=j^*}^{k-1}$, and $\{c_\ell\}_{\ell=0}^{d-1}$ such that $a_j(-1)^{d+1} \geq 0$, $b_j\geq 0$, and
\begin{align}\label{eq:can_func}
f^*(x) = \sum_{j=1}^{j^*}a_j\Big(\frac{n_j}{n} - x\Big)_+^{d} + \sum_{j=j^*}^{k-1}b_j\Big(x-\frac{n_j}{n}\Big)_+^d + \sum_{\ell=0}^{d-1}\frac{c_\ell}{\ell!}x^\ell
\end{align} 
for $x\in (0,1]$, where $\{n_j/n\}_{j=0}^{k}$ are the knots of $f^*$. On the sequence level, we have for every $\theta^*\in\Theta^*(d,k)$:
\begin{align}\label{eq:can_seq}
\theta^*_i = \sum_{j=1}^{j^*}a_j\bigg(\frac{n_j - i}{n}\bigg)_+^{d} + \sum_{j=j^*}^{k-1}b_j\bigg(\frac{i-n_j}{n}\bigg)_+^{d} + \sum_{\ell=0}^{d-1}\frac{c_\ell}{\ell!}(i/n)^\ell.
\end{align} 
\end{lemma}

The next result generalizes the bound $|c_0|\leq C/\sqrt{n}$ in the previous proof outline, indicating that all lower-order polynomial coefficients of a $d$-monotone spline can be well-controlled.
\begin{lemma}\label{lemma:shape_coef}
For any $\theta^*\in \Theta^*(d,k)$ with $\|\theta^*\|^2\leq 1$, there exists some $C = C(d)$ such that, in its canonical form \eqref{eq:can_seq}, $|c_\ell|\leq C/\sqrt{n}$ for every $\ell\in[0;d-1]$.
\end{lemma}

The proof of the above lemmas can be found in Appendix \ref{sec:proof_shape}.

\subsection{Main proof}\label{subsec:shape_main}

\begin{proof}[Proof of Theorem \ref{thm:shape_upper} (upper bound)]
Throughout the proof, we will shorthand $\Theta^*(d,k)$ as $\Theta^*$. We start with a slight modification of the reduction principle in Proposition \ref{prop:reduction}. 

Let $L_0\equiv n/k$ be an integer without loss of generality. Let $\theta^*_\ora$ be an oracle in $\Theta^*$ that achieves the infimum. Let $n_j \equiv n_j(\theta^*_\ora)$, $0\leq j\leq k$ be the knots of $\theta^*_\ora$: $0 = n_0 \leq n_1\leq \ldots\leq n_k = n$. For each $j\in[0;k-1]$, let $m_j \equiv m_j(\theta^*_\ora) \equiv \ceil{(n_{j+1} - n_j)/L_0}$, $n_{j,p} \equiv n_{j,p}(\theta^*_\ora)\equiv n_j + p\cdot L_0$ for $p\in[0;m_j-1]$ so that $n_{j,0} = n_j$ and $n_{j,m_j}\equiv n_{j,m_j}(\theta^*_\ora) \equiv n_{j+1}$. Lastly, for any $\theta^*\in\Theta^*$, let $s_{j,p}\equiv s_{j,p}(\theta^*, \theta^*_\ora)$ be the number of knots of $\theta^*-\theta^*_\ora$ on the segment $(n_{j,p},n_{j,p+1}]$, so that $\sum_{j=0}^{k-1}\sum_{p=0}^{m_j-1}s_{j,p}\leq k$.
Under the above notation, define, for each $\theta\in\RR^n$, $(\theta)_{[j,p]}$ as the sub-vector $(\theta_i)_{i\in(n_{j,p},n_{j,p+1}]}$. 

Following the same line of proof as Proposition \ref{prop:reduction} on this finer resolution $\{n_{j,p}\}$, we have, for any $\delta > 0$ and then some $C = C(\delta)$,
\begin{align*}
\E_{\theta_0}\|\widehat{\theta} - \theta_0\|^2 \leq (1+\delta)\|\theta^*_\ora - \theta_0\|^2 + C\cdot\E\sup_{\theta^*\in\Theta^*}\sum_{j=0}^{k-1}\sum_{p=0}^{m_j-1}\big(\varepsilon_{[j,p]}\cdot v_{j,p}(\theta^*)\big)^2, 
\end{align*}
where $v_{j,p}(\theta^*)\equiv v_{j,p}(\theta^*;\theta^*_\ora)\equiv (\theta^*-\theta^*_\ora)_{[j,p]}/\|(\theta^*-\theta^*_\ora)_{[j,p]}\|$.

We now prove that the second term on the right side can be bounded by a constant multiple of $k\log\log(16n/k)$. Some extra notation is hence needed. For any $\theta^*\in\Theta^*$, denote the set of $s_{j,p}$ knots of $v_{j,p}(\theta^*)$ as $n_{j,p,1},\ldots,n_{j,p,s_{j,p}}$. Also define $n_{j,p,0}\equiv n_{j,p,0}(\theta^*_\ora) \equiv n_{j,p}$ and $n_{j,p,s_{j,p}+1}\equiv n_{j,p,s_{j,p}+1}(\theta^*_\ora)\equiv n_{j,p+1}$. Moreover, in view of the canonical parametrization of shape-constrained splines in Lemma \ref{lemma:shape_para}, let for each fixed $j\in[0;k-1]$ and $p\in[0;m_j]$ the index $q\equiv q(\theta^*,\theta^*_\ora)\in[0;s_{j,p}]$ be such that, on $(n_{j,p},n_{j,p+1}]$, $(n_{j,p,q^*-1}, n_{j,p,q^*}]$ is the last piece on which the sign of the highest order polynomial component of $\theta^*-\theta^*_\ora$ is negative. 

Under the above notation, we have $v_{j,p}(\theta^*)\in\Theta^*_{n_{j,p+1}-n_{j,p}}(d,s_{j,p}+1)$ (here we assume without loss of generality that the two end pieces of $\theta^*-\theta^*_\ora$ adjacent to $n_{j,p}$ and $n_{j,p+1}$ also have length at least $d+1$ since there are at most $2k$ such pieces and each only contributes a constant factor to the complexity width). Thus Lemma \ref{lemma:shape_para} entails that there exist real sequences $\{c_{j,p,\ell}\}\equiv \{c_{j,p,\ell}(\theta^*,\theta^*_\ora)\}$, and some $q^*\in[1;s_{j,p}]$ along with sequences of equal sign $\{a_{j,p,q}\}_{q=1}^{q^*}\equiv \{a_{j,p,q}(\theta^*,\theta^*_\ora)\}_{q=1}^{q^*}$, $\{b_{j,p,q}\}_{q=q^*}^{s_{j,p}}\equiv \{b_{j,p,q}(\theta^*,\theta^*_\ora)\}_{q=q^*}^{s_{j,p}}$ such that 
\begin{align}\label{eq:shape_formula}
\notag\big(v_{j,p}(\theta^*)\big)_i&= \sum_{q=1}^{q^*}a_{j,p,q}\bigg(\frac{n_{j,p,q}-(i+n_{j,p})}{n}\bigg)^d_+ + \sum_{q=q^*}^{s_{j,p}}b_{j,p,q}\bigg(\frac{(i+n_{j,p})-n_{j,p,q}}{n}\bigg)^d_+  \\
&\quad +\sum_{\ell=0}^{d-1}\frac{c_{j,p,\ell}}{\ell!}\bigg(\frac{i-n_{j,p}}{n}\bigg)^\ell \equiv \big(v^1_{j,p}(\theta^*)\big)_i + \big(v^2_{j,p}(\theta^*)\big)_i,
\end{align}
where $(v^2_{j,p}(\theta^*))_i\equiv \sum_{\ell=0}^{d-1}c_{j,p,\ell}\big((i-n_{j,p})/n\big)^\ell/\ell!$. Therefore, we have
\begin{align*}
&\E\sup_{\theta^*\in\Theta^*}\sum_{j=0}^{k-1}\sum_{p=0}^{m_j-1}\big(\varepsilon_{[j,p]}\cdot v_{j,p}(\theta^*) \big)^2\\
&\leq 2\bigg(\E\sup_{\theta^*\in\Theta^*}\sum_{j=0}^{k-1}\sum_{p=0}^{m_j-1}\big(\varepsilon_{[j,p]}\cdot v_{j,p}^1(\theta^*)\big)^2 + \E\sup_{\theta^*\in\Theta^*}\sum_{j=0}^{k-1}\sum_{p=0}^{m_j-1}\big(\varepsilon_{[j,p]}\cdot v_{j,p}^2(\theta^*)\big)^2\bigg)\\
&\equiv 2\Big((I) + (II)\Big).
\end{align*}
We first upper bound $(II)$. Since for each $j,p$ and $\theta^*\in\Theta^*$, $v_{j,p}(\theta^*)\in \Theta^*_{n_{j,p+1} - n_{j,p}}(d,s_{j,p}+1)$ and has unit norm, Lemma \ref{lemma:shape_coef} entails that there exists some $C = C(d)$ such that $|c_{j,p,\ell}|\leq C/\sqrt{n_{j,p+1} - n_{j,p}}$ for $j\in[0;k-1]$, $p\in[0;m_j]$, and $\ell\in[0;d-1]$. Let $\Delta n_{j,p}\equiv n_{j,p+1} - n_{j,p}$. Then, we have
\begin{align*}
(II) &\leq C\cdot \E\sup_{|c_{j,p,\ell}|\leq C/\sqrt{\Delta n_{j,p}}}\sum_{j=0}^{k-1}\sum_{p=0}^{m_j-1}\sum_{\ell=0}^{d-1} \frac{c_{j,p,\ell}^2}{n^{2\ell}(\ell!)^2}\bigg(\sum_{i\in(n_{j,p};n_{j,p+1}]}(i-n_{j,p})^\ell\varepsilon_i\bigg)^2\\
&\leq C\cdot \sum_{j=0}^{k-1}\sum_{p=0}^{m_j-1}\sum_{\ell=0}^{d-1}(\Delta n_{j,p})^{-1} \frac{\E\big[\sum_{i\in(n_{j,p};n_{j,p+1}]}(i-n_{j,p})^\ell\varepsilon_i\big]^2}{n^{2\ell}}\\
&\leq C\cdot\sum_{j=0}^{k-1}\sum_{p=0}^{m_j-1}1 = C\cdot\sum_{j=0}^{k-1} m_j \leq Ck.
\end{align*}

Next, we bound $(I)$. Some extra notation is needed. Define the following partition of $(n_{j,p};n_{j,p+1}]$ with intervals
\begin{align*}
I^B_{j,p,\ell}\equiv \bigg(n_{j,p} + \ceil{(1-2^{-(\ell-1)})\Delta n_{j,p}};n_{j,p} + \ceil{(1-2^{-\ell})\Delta n_{j,p}}\bigg]
\end{align*}
for $\ell\in[1;t_{j,p}]$ and $t_{j,p}\equiv \ceil{\log_2\Delta n_{j,p}}$, and similarly,
\begin{align*}
I^A_{j,p,\ell}\equiv \bigg(n_{j,p+1} - \ceil{(1-2^{-\ell})\Delta n_{j,p}};n_{j,p+1} - \ceil{(1-2^{-(\ell-1)})\Delta n_{j,p}}\bigg].
\end{align*}
From this definition, we immediately have (with analogous conclusions for $I^A_{j,p,\ell}$): (i) $|I^B_{j,p,\ell}| \leq \ceil{2^{-\ell}\Delta n_{j,p}}$; (ii) $2(\sum_{\ell > \ell_0}|I^B_{j,p,\ell}| + 1)\geq \sum_{\ell\geq \ell_0}|I^B_{j,p,\ell}|$ for any $\ell_0\in[1;t_{j,p}]$. Then, let 
\begin{align*}
B_{j,p,\ell}&\equiv B_{j,p,\ell}(\theta^*,\theta^*_\ora)\equiv \sum_{q=q^*}^{s_{j,p}} b_{j,p,\ell}\bm{1}_{n_{j,p,q}\in I^B_{j,p,\ell}},\\
\delta^B_{j,p,\ell} &\equiv \delta^B_{j,p,\ell}(\theta^*, \theta^*_\ora)\equiv \max\Big\{q^*\leq q\leq s_{j,p}: \bm{1}_{n_{j,p,q}\in I^B_{j,p,\ell}}\Big\}.
\end{align*}
In words, $\delta^B_{j,p,\ell}$ equals to $1$ if and only if among the knots $\{n_{j,p,q}\}_{q=q^*}^{s_{j,p}}$, there is at least one that lies in the interval $I^B_{j,p,\ell}$, and if such is the case, $B_{j,p,\ell}$ returns the block sum. We omit the similar definitions for $A_{j,p,\ell}$ and $\delta^A_{j,p,\ell}$. By definition, we immediately have $\sum_{\ell=1}^{t_{j,p}}\delta^B_{j,p,\ell}\leq s_{j,p}$.

In the parametrization \eqref{eq:shape_formula}, using the constraint $\|v_{j,p}(\theta^*)\|\leq 1$ and the bounds $|c_{j,p,\ell}|\leq C/\sqrt{n_{j,p+1} - n_{j,p}}$ for $\ell\in[0;d-1]$, we have $\pnorm{v_{j,p}^1(\theta^\ast)}{}\leq C$ (recall the definition of $v^1_{j,p}$ in \eqref{eq:shape_formula}) for some $C = C(d)$. Hence for some sufficiently small $c = c(d)$,
\begin{align*}
1 &\geq c\cdot\sum_{i\in(n_{j,p};n_{j,p+1}]}\bigg[\sum_{q=1}^{q^*}a_{j,p,q}\bigg(\frac{n_{j,p,q}-i}{n}\bigg)^d_+ + \sum_{q=q^*}^{s_{j,p}}b_{j,p,q}\bigg(\frac{i-n_{j,p,q}}{n}\bigg)^d_+\bigg]^2\\
&\geq c\cdot\sum_{i\in(n_{j,p};n_{j,p+1}]}\bigg[\sum_{q=1}^{q^*}a_{j,p,q}\bigg(\frac{n_{j,p,q}-i}{n}\bigg)^d_+\bigg]^2\vee \bigg[\sum_{q=q^*}^{s_{j,p}}b_{j,p,q}\bigg(\frac{i-n_{j,p,q}}{n}\bigg)^d_+\bigg]^2,
\end{align*} 
where the second inequality follows from the fact that the interaction term between the two summands in the first inequality is $0$ for each $i$.

Now, starting from the constraint $
1\geq c\cdot \sum_{i\in(n_{j,p};n_{j,p+1}]}\big[\sum_{q=q^*}^{s_{j,p}}b_{j,p,q}\big(\frac{i-n_{j,p,q}}{n}\big)^d_+\big]^2$,
we will obtain estimates for $B_{j,p,\ell}$. Fix $j,p$. By the disjointness of $I^B_{j,p,\ell}$ and the non-negativeness of $\{b_{j,p,q}\}$, we have
\begin{align}\label{eq:B_estimate}
\notag1& \geq c\cdot \sum_{i\in(n_{j,p};n_{j,p+1}]}\bigg[\sum_{\ell=1}^{t_{j,p}}\sum_{q=q^*}^{s_{j,p}}b_{j,p,q}\bm{1}_{n_{j,p,q} \in I_{j,p,\ell}^B}\bigg(\frac{i-n_{j,p,q}}{n}\bigg)^d_+\bigg]^2  \\
\notag &\geq c\cdot \sum_{i\in(n_{j,p};n_{j,p+1}]}\bigg[\sum_{\ell=1}^{t_{j,p}}\sum_{q=q^*}^{s_{j,p}}b_{j,p,q}\bm{1}_{n_{j,p,q} \in I_{j,p,\ell}^B}\bigg(\frac{i-(I^B_{j,p,\ell})_+  }{n}\bigg)^d_+\bigg]^2\\
\notag &= c\cdot \sum_{i\in(n_{j,p};n_{j,p+1}]}\bigg[\sum_{\ell=1}^{t_{j,p}}B_{j,p,\ell}\bigg(\frac{i-(I^B_{j,p,\ell})_+}{n}\bigg)^d_+\bigg]^2\\
\notag &\geq c\cdot \sum_{\ell=1}^{t_{j,p}}B_{j,p,\ell}^2\sum_{i\in(n_{j,p};n_{j,p+1}]}\bigg(\frac{i-(I^B_{j,p,\ell})_+ }{n}\bigg)^{2d}_+\\
%\notag&\geq c\cdot \sum_{\ell=1}^{t_{j,p}}B_{j,p,\ell}^2\frac{(n_{j,p+1} - n_{j,p,q(\ell)})^{2d+1}}{n^{2d}}\\
&\geq c\cdot \sum_{\ell=1}^{t_{j,p}}B_{j,p,\ell}^2\frac{(n_{j,p+1} - (I^B_{j,p,\ell})_+)^{2d+1}}{n^{2d}} \geq c\cdot \sum_{\ell=1}^{t_{j,p}}B_{j,p,\ell}^2\frac{(n_{j,p+1} - (I^B_{j,p,\ell})_-)^{2d+1}}{n^{2d}}, 
\end{align}
where 
%$q(\ell)\equiv \max\{q^*\leq q\leq s_{j,p}: n_{j,p,q}\in I^B_{j,p,\ell}\}$, 
$(I^B_{j,p,\ell})_+$ ($(I^B_{j,p,\ell})_-$) is defined to be the right (left) endpoint of $I^B_{j,p,\ell}$, and the last inequality follows from property (ii) of the partition $I^B_{j,p,\ell}$.

We are now ready to bound the term $(I)$. First by the vanishing of interaction terms, we have $(I) = (I_1) + (I_2)$, where
\begin{align*}
(I_1)&\equiv \E\sup_{\theta^*\in\Theta^*}\sum_{j=0}^{k-1}\sum_{p=0}^{m_j-1}\bigg[\sum_{q=0}^{q^*}a_{j,p,q}\bigg(\sum_{i\in(n_{j,p};n_{j,p,q}]}\bigg(\frac{n_{j,p,q} - i}{n}\bigg)^d_{+}\varepsilon_i\bigg)\bigg]^2,\\
(I_2)&\equiv \E\sup_{\theta^*\in\Theta^*}\sum_{j=0}^{k-1}\sum_{p=0}^{m_j-1}\bigg[\sum_{q=q^*}^{s_{j,p}}b_{j,p,q}\bigg(\sum_{i\in(n_{j,p,q};n_{j,p+1}]} \bigg(\frac{i-n_{j,p,q}}{n}\bigg)^d_{+}\varepsilon_i\bigg)\bigg]^2.
\end{align*}
Due to symmetry, we only bound $(I_2)$ as follows:
\begin{align*}
(I_2)&= \E\sup_{\theta^*}\sum_{j,p}\bigg[\sum_{q=q^*}^{s_{j,p}} \sum_{\ell=1}^{t_{j,p}}\bm{1}_{n_{j,p,q}\in I^B_{j,p,\ell}}b_{j,p,q}\bigg(\sum_{i\in(n_{j,p,q};n_{j,p+1}]}\bigg(\frac{i-n_{j,p,q}}{n}\bigg)^d\varepsilon_i\bigg)\bigg]^2\\
&\leq \E\sup_{\theta^*}\sum_{j,p}\bigg[\sum_{\ell=1}^{t_{j,p}}\bigg\{\sum_{q=q^*}^{s_{j,p}}\bm{1}_{n_{j,p,q}\in I^B_{j,p,\ell}}b_{j,p,q}\bigg\}\max_{\tau\in I^B_{j,p,\ell}}\bigg|\sum_{i\in(\tau;n_{j,p+1}]}\bigg(\frac{i-\tau}{n}\bigg)^d\varepsilon_i\bigg|\bigg]^2\\
&= \E\sup_{\theta^*}\sum_{j,p}\bigg[\sum_{\ell=1}^{t_{j,p}}B_{j,p,\ell}\max_{\tau\in I^B_{j,p,\ell}}\bigg|\sum_{i\in(\tau;n_{j,p+1}]}\bigg(\frac{i-\tau}{n}\bigg)^d\varepsilon_i\bigg|\bigg]^2\\
&\leq \E \max_{ \{\delta_{j,p,\ell}^B\} \in \Delta^B  }\sum_{j=0}^{k-1}\sum_{p=0}^{m_j-1}\sum_{\ell=1}^{t_{j,p}}\delta^B_{j,p,\ell}\max_{\tau\in I^B_{j,p,\ell}}\frac{\Big(\sum_{i\in(\tau;n_{j,p+1}]}(i-\tau)^d\varepsilon_i\Big)^2}{(n_{j,p+1} - (I^B_{j,p,\ell})_-)^{2d+1}}.
\end{align*}
Here, the first inequality follows from the non-negativity of $\{b_{j,p,q}\}$, the second equality follows from the definition of $B_{j,p,\ell}$, and the last inequality follows from Cauchy-Schwarz along with the estimates for $B_{j,p,\ell}$ in \eqref{eq:B_estimate}. Furthermore, we define
\begin{align*}
    \Delta^B\equiv \bigg\{\{\delta^B_{j,p,\ell}\}: \delta^B_{j,p,\ell}\in\{0,1\}, \sum_{j=1}^k\sum_{p=1}^{m_j}\sum_{\ell=1}^{t_{j,p}} \delta^B_{j,p,\ell}\leq k\bigg\}
\end{align*}
to be the admissible set for the sequence $\{\delta^B_{j,p,\ell}\}$. As $\sum_{j=1}^k\sum_{p=1}^{m_j}\sum_{\ell=1}^{t_{j,p}} 1 =\sum_{j=1}^k\sum_{p=1}^{m_j} \ceil{\log_2(n_{j,p+1}-n_{j,p})} \leq Ck\ceil{\log_2(n/k)}$, a combinatorial estimate yields that $|\Delta^B|\leq \binom{Ck\ceil{\log_2(n/k)}}{k}\leq (Ce\ceil{\log_2(n/k)})^k$. 

Now, using the basic inequality $(a+b)^2\leq 2(a^2 + b^2)$, it suffices to bound by the order $k\log\log(16 n/k)$ the following two terms:
\begin{align}\label{eq:term1}
\E \max_{ \{\delta_{j,p,\ell}^B\} \in \Delta^B }\sum_{j=0}^{k-1}\sum_{p=0}^{m_j-1}\sum_{\ell=1}^{t_{j,p}}\delta^B_{j,p,\ell}\max_{\tau\in I^B_{j,p,\ell}}\frac{\Big(\sum_{i\in(\tau;(I^B_{j,p,\ell})_+]}(i-\tau)^d\varepsilon_i\Big)^2}{\Big(n_{j,p+1} - (I^B_{j,p,\ell})_-\Big)^{2d+1}}
\end{align}
and
\begin{align}\label{eq:term2}
\E \max_{ \{\delta_{j,p,\ell}^B\} \in \Delta^B }\sum_{j=0}^{k-1}\sum_{p=0}^{m_j-1}\sum_{\ell=1}^{t_{j,p}}\delta^B_{j,p,\ell}\max_{\tau\in I^B_{j,p,\ell}}\frac{\Big(\sum_{i\in((I^B_{j,p,\ell})_+;n_{j,p+1}]}(i-\tau)^d\varepsilon_i\Big)^2}{\Big(n_{j,p+1} - (I^B_{j,p,\ell})_-\Big)^{2d+1}}.
\end{align}
From here on, in view of Theorem \ref{thm:lil}, the proof is essentially the same as that of Lemma 5.2 in \cite{gao2017minimax} (our \eqref{eq:term1} and \eqref{eq:term2} correspond to their (42) and (43)). For the sake of completeness, we will present the proof for the bound of \eqref{eq:term1}; the bound for \eqref{eq:term2} follows from essentially the proof of (43) in \cite{gao2017minimax}. 

Denote the variable in \eqref{eq:term1} as $Z$, i.e.,
\begin{align*}
Z\equiv  \max_{ \{\delta_{j,p,\ell}^B\} \in \Delta^B } \sum_{j=0}^{k-1}\sum_{p=0}^{m_j-1}\sum_{\ell=1}^{t_{j,p}}\delta^B_{j,p,\ell}\max_{\tau\in I^B_{j,p,\ell}}\frac{\Big(\sum_{i\in(\tau;(I^B_{j,p,\ell})_+]}(i-\tau)^d\varepsilon_i\Big)^2}{\Big(n_{j,p+1} - (I^B_{j,p,\ell})_-\Big)^{2d+1}}.
\end{align*}
We bound the tail probability of $Z$ as follows. For any $u\geq 0$ and small enough $c>0$,
\begin{align*}
&\Prob(Z > u)\\
&\leq \sum_{\{\delta_{j,p,\ell}^B\} \in \Delta^B} \Prob\bigg[\sum_{j=0}^{k-1}\sum_{p=0}^{m_j-1}\sum_{\ell=1}^{t_{j,p}}\delta^B_{j,p,\ell}\max_{\tau\in I^B_{j,p,\ell}}\frac{\Big(\sum_{i\in(\tau; (I_{j,p,\ell}^B)_+ ]}(i-\tau)^d\varepsilon_i\Big)^2}{\Big(n_{j,p+1} - (I_{j,p,\ell}^B)_-\Big)^{2d+1}} > u\bigg]\\
&\leq \sum_{\{\delta_{j,p,\ell}^B\} \in \Delta^B} e^{-cu} \prod_{j,p,\ell}\E\exp\bigg[c\delta^B_{j,p,\ell}\max_{\tau\in I_{j,p,\ell}^B }\frac{\Big(\sum_{i\in(\tau;(I_{j,p,\ell}^B)_+ ]}(i-\tau)^d\varepsilon_i\Big)^2}{\Big(n_{j,p+1} - (I_{j,p,\ell}^B)_-\Big)^{2d+1}}\bigg]\\
&\lesssim \sum_{\{\delta_{j,p,\ell}^B\} \in \Delta^B} e^{-cu}\cdot \exp\bigg(\sum_{j=0}^{k-1}\sum_{p=0}^{m_j-1}\sum_{\ell=1}^{t_{j,p}}C\delta^B_{j,p,\ell}\log\log\big(16(n_{j,p+1} - n_{j,p})\big)\bigg)\\
&\leq \exp\big(\log \abs{\Delta^B} - cu + C k\log\log(16n/k)\big)\leq \exp(-cu + Ck\log\log(16n/k)\big).
\end{align*}
Here, the second inequality follows from the independence of the partial sum processes over the partition $\{I^B_{j,p,\ell}\}$, the third inequality follows by choosing $c$ to be sufficiently small and then applying Theorem \ref{thm:lil} with $\psi(x) = \exp(cx^2)-1$ therein, and the fourth inequality follows from the fact that $n_{j,p+1} - n_{j,p}\leq n/k$ and that $\sum_{j=0}^{k-1}\sum_{p=0}^{m_j-1}\sum_{\ell=1}^{t_{j,p}}\delta^B_{j,p,\ell} \leq k$ for any $\{\delta_{j,p,\ell}^B\} \in \Delta^B$. The proof is now complete by integrating the tail estimate. 
\end{proof}

\appendix

\section{Proof of lower bounds}\label{sec:proof_lower}

\subsection{Lower bound in Section \ref{sec:gen_spline}}\label{subsec:noshape_lower}
\begin{proof}[Proof of Proposition \ref{prop:noshape_lower}]
We start with the first claim. In view of the fact that minimax rate over $\Theta(d,d_0,k)$ is non-decreasing in $k$ and $\Theta(d,d-1,k)\subset\Theta(d,d_0,k)$ for any $d_0\in[-1; d-1]$, it suffices to show that
\begin{align*}
\inf_{\td{\theta}}\sup_{\theta\in\Theta(d,d-1,2)} \E_{\theta}\|\td{\theta} - \theta\|^2 \geq c\log\log (16n).
\end{align*}
For this, we will apply a standard reduction argument to multiple hypothesis testing (cf. Theorem 2.5 of \cite{tsybakov2008introduction}). Define the following series of splines. Let $M\equiv \floor{\log_2 (n/(d+1))}$, and for each $\ell\in[1;M]$, $\tau_\ell\equiv \floor{(1-2^{-\ell})n}$ and $f^\ell(x)\equiv \alpha_\ell(x-\tau_\ell/n)_+^d$ with $\alpha_\ell \equiv c(2^\ell)^{(2d+1)/2}\sqrt{\log\log (16n)/n}$ for some sufficiently small $c$. Further define $f^0(x)\equiv 0$ on $[0,1]$, and the induced vectors $\theta^\ell_i \equiv f^\ell(i/n)$ for $i\in[1;n]$ and $\ell\in[0; M]$. Denote the corresponding joint distribution of $\{Y_i\}_{i=1}^n$ under the experiment \eqref{eq:model} with truth $\theta^\ell$ as $P_\ell$, $\ell\in[0;M]$. It can be readily verified that $\theta^\ell\in\Theta(d,d-1,2)$, and the Kullback-Leibler divergence between $P_0$ and each $P_\ell$, denoted as $\KL(P_0, P_\ell)$, satisfies
\begin{align*}
\KL(P_0, P_\ell) =\|\theta^0 - \theta^\ell\|^2/2 = \|\theta^\ell\|^2/2 \asymp \log\log (16n)
\end{align*}
for every $\ell\in[1;M]$. Moreover, for any $1\leq j<k\leq M$, it holds by direct calculation that 
\begin{align*}
d(P_j, P_k)&\equiv \|\theta^j - \theta^k\|^2 \geq \sum_{i\in(\tau_j,\tau_k]}(\theta^j_i - \theta^k_i)^2 \asymp \alpha_j^2\frac{(\tau_k - \tau_j)^{2d+1}}{n^{2d}}\\
&\asymp \alpha_j^2\frac{(n - \tau_j)^{2d+1}}{n^{2d}} \asymp \alpha_j^2 \frac{2^{-j(2d+1)}}{n} \asymp \log\log (16n).
\end{align*}  
Theorem 2.5 in \cite{tsybakov2008introduction} therefore entails the desired lower bound. 

Next, we prove the second claim. By following the same reduction as in the previous claim, it suffices to show that for any $k\geq k_0+1$, there exists some nonzero $f\in\cal{F}_n(d,d_0,k_0+1)$ such that $f(x) = 0$ for $x\in[0,c]\cup[1-c,1]$ with some universal $c$. Take $c = 1/3$. Let $\tau_0 \equiv 0$, $\tau_j \equiv 1/3+(j-1)/(3(k_0-1))$ for $j\in[1; k_0]$, and $\tau_{k_0+1} \equiv 1$. Define 
\begin{align*}
f(x)\equiv \Big(\sum_{j=1}^{k_0-1}\sum_{\ell=d_0+1}^d c^j_\ell(x-\tau_j)_+^\ell\Big)\cdot\bm{1}_{[1/3,2/3]}(x), \quad x\in[0,1].
\end{align*}
By definition, $f$ vanishes on $[0,1/3]\cup[2/3,1]$. Moreover, it can be readily checked that, for any real sequence $\{c_\ell^j\}_{j \in [1;k_0-1],\ell \in [d_0+1;d]}$, $f^{(\ell)}((\tau_j)_-) = f^{(\ell)}((\tau_j)_+)$ for $j\in[1;k_0-1]$ and $\ell\in[0;d_0]$. Therefore, in order to show that $f\in\cal{F}_n(d,d_0,k_0+1)$ and is non-zero, it suffices to show that there exists a non-zero realization of the sequence $\{c^j_\ell\}_{
j\in[1;k_0-1],\ell\in[d_0+1; d]}$ such that $f^{(\ell)}((\tau_{k_0)})_-) = f^{(\ell)}((\tau_{k_0})_+) = 0$ for all $\ell \in [0;d_0]$. This is equivalent to finding a non-zero solution for the homogeneous linear system $\boldsymbol{A}\boldsymbol{c}= \boldsymbol{b}$, where $\boldsymbol{c}\equiv \{c^j_\ell\}_{j\in[1;k_0-1],\ell\in[d_0+1;d]}\in \RR^{(k_0-1)(d-d_0)}$, $\boldsymbol{b}\equiv \boldsymbol{0}_{(k_0-1)(d-d_0)}$, and
\begin{align*}
\boldsymbol{A}\equiv
\begin{bmatrix}
\boldsymbol{A}_1 &  \boldsymbol{A}_2 & \ldots \, \boldsymbol{A}_{k_0-1}
\end{bmatrix}
\end{align*}
with
\begin{align*}
\bm{A}_j\equiv
\begin{bmatrix}
\underline{\odot}(d_0+1;0)\tau_{k_0,j}^{d_0+1} & \underline{\odot}(d_0+2;0)\tau_{k_0,j}^{d_0+2} & \ldots & \underline{\odot}(d;0)\tau_{k_0,j}^d & \\
\underline{\odot}(d_0+1;1)\tau_{k_0,j}^{d_0} & \underline{\odot}(d_0+2;1)\tau_{k_0,j}^{d_0+1} & \ldots & \underline{\odot}(d;1)\tau_{k_0,j}^{d-1} \\
&\ldots\\
\underline{\odot}(d_0+1;d_0)\tau_{k_0,j} & \underline{\odot}(d_0+2;d_0)\tau_{k_0,j}^2 &  \ldots & \underline{\odot}(d;d_0)\tau_{k_0,j}^{d-d_0}
\end{bmatrix}
\end{align*}
and $\tau_{j_1,j_2}\equiv \tau_{j_1} - \tau_{j_2}$. Note that the coefficient matrix $\boldsymbol{A}$ has $d_0+1$ rows and $(k_0-1)(d-d_0)$ columns, where, by definition of $k_0$, 
\begin{align*}
(k_0-1)(d-d_0) \geq d_0+2 \iff \floor{\frac{d+1}{d-d_0}} + 1\geq \frac{d+2}{d-d_0}.
\end{align*}
The above equivalence indeed holds since if $(d+1)/(d-d_0)$ is an integer, then
\begin{align*}
\floor{\frac{d+1}{d-d_0}} + 1 = \frac{d+1+(d-d_0)}{d-d_0}\geq \frac{d+2}{d-d_0},
\end{align*}
and if not
\begin{align*}
\floor{\frac{d+1}{d-d_0}} + 1\geq\ceil{\frac{d+1}{d-d_0}}\geq
\frac{d+2}{d-d_0}.
\end{align*}
This entails that the solution space of the linear system $\boldsymbol{A}\boldsymbol{c}= \boldsymbol{b}$ is of dimension at least one and thus the system is guaranteed to have a non-trivial solution. The proof is thus complete.
\end{proof}

\subsection{Lower bound in Section \ref{sec:gen_shape}}\label{subsec:shape_lower}

\begin{proof}[Proof of Proposition \ref{prop:cvx_linear_lower}]

We will continue to adopt the standard reduction to multiple testing (cf. Theorem 2.5 of \cite{tsybakov2008introduction}) as in the proof of Proposition \ref{prop:noshape_lower}. We first introduce a set of basis functions. Let $\td{k}\equiv k/3$ which we assume without loss of generality to be an integer, $\ell_0\equiv \floor{\log_2(n/(2\td{k}))}$, and  $\tau_\ell\equiv (1-2^{-(\ell-1)})/\td{k}$ for $\ell\in[1;\ell_0 + 1]$. Next, for $x\in[0,1/\td{k}]$, let $\td{f}_\ell(x)\equiv c(2^{\ell-1})^{3/2}\sqrt{\log\log(16n/k)/n}(x-\tau_\ell)_+$ for $\ell\in[1;\ell_0]$ and $f_{\text{ref}}(x)\equiv c(2^{\ell_0})^{3/2}\sqrt{\log\log(16n/k)/n}(x-\tau_{\ell_0+1})_+$ (here the subscript ``ref"  stands for ``reference" and $f_{\text{ref}}$ will be pieced together later to be the true signal underlying the distribution $P_0$ in Theorem 2.5 of \cite{tsybakov2008introduction}). Then let $f_\ell(x) \equiv \td{f}_\ell(x)\vee f_{\text{ref}}(x)$, and it can be verified that $f_\ell(x) = \td{f}_\ell(x)$ on $[0,\tau_{\ell_0+1}]$. The above set of functions resembles those constructed in the proof of Proposition \ref{prop:noshape_lower}, and satisfies the similar properties
\begin{align}\label{eq:separation}
\sum_{i:(i/n)\in(0, 1/\td{k}]} \big(f_{\ell}(i/n) - f_{\ell^\prime}(i/n)\big)^2 \geq c\log\log(16n/k)
\end{align}
for any $1\leq \ell\neq \ell^\prime\leq \ell_0$, and
\begin{align}\label{eq:KL}
&\sum_{i:(i/n)\in(0, 1/\td{k} ]} \big(f_{\ell}(i/n) - f_{\text{ref}}(i/n)\big)^2 \leq \sum_{i:(i/n)\in(0, 1/\td{k} ]}\big(f_{\ell}(i/n)\big)^2 \nonumber \\
& \leq 2\bigg(\sum_{i:(i/n)\in(0, 1/\td{k} ]}\big(\td{f}_{\ell}(i/n)\big)^2 + \sum_{i:(i/n)\in(0, 1/\td{k} ]}\big(f_{\text{ref}}(i/n)\big)^2\bigg) \nonumber\\
&\leq C\log\log(16n/k).
\end{align}
We now construct the hypotheses in the multiple testing framework. For $j\in[1;\td{k}]$, let $f^j_\ell(\cdot), f^j_{\text{ref}}(\cdot)$ be a set of functions defined on $[(j-1)/\td{k},j/\td{k}]$ as follows. Let $f^1_\ell(x) \equiv f_\ell(x)$ and $f^1_{\text{ref}}(x)\equiv f_{\text{ref}}(x)$ as defined above. Next, for $j\in[2;k]$, we define inductively $f^j_{\text{ref}}(x)\equiv f^{j-1}_{\text{ref}}(x) + f_{\text{ref}}(x-(j-1)/\td{k})$, where $f^{j-1}_{\text{ref}}(x)$ for $x\in[(j-1)/\td{k},j/\td{k}]$ is to be understood as the extension from $[(j-2)/\td{k},(j-1)/\td{k}]$. Also define $f^j_\ell(x)\equiv f^j_{\text{ref}}(x) + f_\ell(x-(j-1)/\td{k})$. Lastly, we piece them together as
\begin{align*}
f^0(x) \equiv \sum_{j=1}^{\td{k}}f^j_{\text{ref}}(x)\bm{1}_{((j-1)/\td{k},j/\td{k}]}(x)
\end{align*} 
and
\begin{align*}
f^{\bm{\ell}}(x) \equiv \sum_{j=1}^{\td{k}}f^j_{\ell_j}(x)\bm{1}_{((j-1)/\td{k},j/\td{k}]}(x),
\end{align*}
where $\bm{\ell} = (\ell_1,\ldots,\ell_{\td{k}})^\top\in[1;\ell_0]^{\td{k}}$. One can readily verify that all of the $f^0$ and $f^{\bm{\ell}}$ belong to the class $\cal{F}_n^*(1,k)$. Indeed, continuity follows directly from the construction and since there are at most $3$ pieces on each of $[(j-1)/\td{k},j/\td{k}]$, there will be at most $3\td{k}=k$ pieces in total. Therefore, the sequence counterparts $\theta^0\equiv (f^0(i/n))_i$ and $\theta^{\bm{\ell}} \equiv (f^{\bm{\ell}}(i/n))_i$ belong to $\Theta^*(1,k)$. 

Let $\rho(\cdot,\cdot)$ denote the Hamming distance. Then, the Gilbert-Varshamov bound (cf. Theorems 5.1.7 and 5.1.9 in \cite{van1999introduction}) entails that with some small $c > 0$, there exists a subset $\cal{S}\subset[1;\ell_0]^{\td{k}}$ with cardinality $|\cal{S}| \asymp \ell_0^{c\td{k}}$ such that $\rho(\bm{\ell},\bm{\ell}^\prime)\geq c\td{k}$ for any $\bm{\ell}\neq \bm{\ell}^\prime\in\cal{S}$. Adopting those in $\cal{S}$ as the truth in the experiment \eqref{eq:model}, we obtain a total of $M\equiv 1+\abs{\mathcal{S}} \asymp \ell_0^{c\td{k}}$ hypotheses, which we denote as $P^0$ and $P^{\bm{\ell}}$, $\bm{\ell}\in\cal{S}$.

It remains to verify: (i) $\|\theta^{\bm{\ell}} - \theta^{\bm{\ell}^\prime}\|^2 \geq ck\log\log(16n/k)$ for any $\bm{\ell}\neq\bm{\ell}^\prime\in\cal{S}$; (ii) $\KL(P^0, P^{\bm{\ell}})\leq C\log|\cal{S}|$ for any $\bm{\ell}\in\cal{S}$. We first verify (i). By definition of $\theta^{\bm{\ell}}$ and $\theta^{\bm{\ell}^\prime}$, on each $[(j-1)/\td{k}, j/\td{k}]$ such that $\ell_j\neq \ell_j^\prime$, we have by \eqref{eq:separation},
\begin{align*}
\sum_{i:\frac{i}{n}\in(\frac{j-1}{\td{k}},\frac{j}{\td{k}}]}(\theta^{\bm{\ell}}_i - \theta^{\bm{\ell}^\prime}_i)^2 &= \sum_{i:\frac{i}{n}\in(\frac{j-1}{\td{k}},\frac{j}{\td{k}}]}\bigg[f_{\ell_j}\bigg(\frac{i}{n} - \frac{j-1}{\td{k}}\bigg) - f_{\ell_j^\prime}\bigg(\frac{i}{n} - \frac{j-1}{\td{k}}\bigg)\bigg]^2\\
&= \sum_{i:\frac{i}{n}\in(0,\frac{1}{\td{k}}]}\bigg[f_{\ell_j}\bigg(\frac{i}{n}\bigg) - f_{\ell_j^\prime}\bigg(\frac{i}{n}\bigg)\bigg]^2 \geq c\log\log(16n/k).
\end{align*}
This entails that
\begin{align*}
\|\theta^{\bm{\ell}} - \theta^{\bm{\ell}^\prime}\|^2 \geq \rho(\bm{\ell},\bm{\ell}^\prime)c\log\log(16n/k) \geq ck\log\log(16n/k).
\end{align*}
Similarly, for (ii), we have by \eqref{eq:KL}
\begin{align*}
\KL(P^0,P^{\bm{\ell}}) &= \|\theta^0 - \theta^{\bm{\ell}}\|^2/2 \leq C\td{k}\cdot\sum_{i:\frac{i}{n}\in(0,\frac{1}{\td{k}}]}\bigg[f_{\ell_j}\bigg(\frac{i}{n}\bigg) - f_{\text{ref}}\bigg(\frac{i}{n}\bigg)\bigg]^2\\
&\leq C\td{k}\log\log(16n/k) \asymp \log|\cal{S}|.
\end{align*}
Application of Theorem 2.5 in \cite{tsybakov2008introduction} then completes the proof.
\end{proof}

\begin{proof}[Proof of Theorem \ref{thm:shape_upper} (lower bound)] This is immediate by realizing that $\Theta^*(d,2)\subset \Theta^*(d,k)$ for $k\geq 2$ and the lower bound construction in the first part of the proof of Proposition \ref{prop:noshape_lower} can be directly applied to establish a lower bound for $\Theta^*(d,2)$. 
\end{proof}

\section{Proof of Theorem \ref{thm:lil}}\label{sec:proof_lil}

\begin{proof}[Proof of Theorem \ref{thm:lil}]
We first claim that there exists some $c=c(d)$ such that for any $t>0$, the event 
\begin{align*}
\cal{E}_1\equiv \bigg\{\max_{1\leq n_1<n_2\leq n}(n_2-n_1)^{-d}(n_2\wedge (n-n_1))^{-1/2}\bigg|\sum_{(n_1;n_2]} (i-n_1)^d\varepsilon_i\bigg|\geq t\bigg\}
\end{align*}
is contained in the event
\begin{align*}
\cal{E}_2 \equiv \bigg\{\max_{1\leq n_1<n_2\leq n}(n_2\wedge (n-n_1))^{-1/2}\bigg|\sum_{(n_1;n_2]} \varepsilon_i\bigg|\geq ct\bigg\}.
\end{align*}
On $\cal{E}_2^c$, for any $1\leq n_1 < n_2 \leq n$, it holds that $\big|\sum_{(n_1;n_2]} \varepsilon_i\big|\leq c(n_2\wedge (n-n_1))^{1/2}t$. Then, 
\begin{align}
\bigg|\sum_{(n_1;n_2]}\varepsilon_i(i-n_1)^d\bigg| &= \biggabs{\sum_{i \in (n_1;n_2]} \epsilon_i \sum_{j=1}^{i-n_1} \big(j^d-(j-1)^d\big)} \notag\\
&= \bigg|\sum_{j=1}^{n_2-n_1}\big(j^d - (j-1)^d\big)\sum_{i\in[n_1+j;n_2]}\varepsilon_i\bigg|\notag\\
&\leq \sum_{\ell=0}^{d-1}{d\choose \ell}\sum_{j=1}^{n_2-n_1}(j-1)^\ell\bigg|\sum_{i\in[n_1+j;n_2]}\varepsilon_i\bigg|\notag\\
&\leq ct\cdot \sum_{\ell=0}^{d-1}{d\choose \ell}\sum_{j=1}^{n_2-n_1}(j-1)^\ell\parr*{\sqrt{n_2}\wedge \sqrt{n-n_1-(j-1)}}\notag\\
&\leq 2ct\cdot \sum_{\ell=0}^{d-1}{d\choose \ell}\int_{0}^{n_2-n_1}x^\ell\parr*{\sqrt{n_2}\wedge \sqrt{n-n_1-x}} \d x \label{eq:han-new1}\\
&\leq 4 ct\cdot \sum_{\ell=0}^{d-1}{d\choose \ell}(n_2-n_1)^{\ell+1}\parr*{n_2\wedge (n-n_1)}^{1/2} \label{eq:han-new2}\\
&\leq c2^{d+2}t\cdot (n_2-n_1)^d\parr*{n_2\wedge (n-n_1)}^{1/2}, \notag
\end{align} 
where the inequality \eqref{eq:han-new1} follows from the fact that the map $x\mapsto x^\ell (\sqrt{n_2}\wedge \sqrt{n-n_1-x})$ first increases and then decreases on $[0,n-n_1]$, and the inequality \eqref{eq:han-new2} follows from a separate discussion of $n_2\leq n-n_1$ and $n_2>n-n_1$ and the following two bounds:  $\int_0^{n_2-n_1}x^\ell\ \d{x} = (\ell+1)^{-1}(n_2-n_1)^{\ell+1}$ and 
\begin{align*}
&\int_0^{n_2-n_1}x^\ell\sqrt{n-n_1-x}\ \d{x} \leq (n_2-n_1)^\ell \int_{n_1}^{n_2}\sqrt{n-x}\ \d{x}\\
&= (n_2-n_1)^\ell\int_{n-n_2}^{n-n_1}\sqrt{x}\ \d{x} = (n_2-n_1)^\ell\cdot\frac{2}{3}\parr*{(n-n_1)^{3/2}-(n-n_2)^{3/2}} \\
&= (n_2-n_1)^\ell\cdot\frac{2}{3}\frac{(n_2-n_1)\big[(n-n_1)^2 + (n-n_1)(n-n_2)+(n-n_2)^2\big]}{(n-n_1)^{3/2}+(n-n_2)^{3/2}}\\
&\leq 2(n_2-n_1)^{\ell+1}(n-n_1)^{1/2}.
\end{align*} 
Therefore the claim holds by choosing $c= 2^{-(d+2)}$. This entails that, for any $t > 0$,
\begin{align*}
\Prob(Z\geq t) &\leq \Prob\bigg(\max_{1\leq n_1<n_2\leq n}(n_2\wedge (n-n_1))^{-1/2}\bigg|\sum_{(n_1;n_2]} \varepsilon_i\bigg|\geq ct\bigg)\\
&\leq \Prob\bigg(\max_{n_1<n_2}\frac{\big|\sum_{(n_1;n_2]} \varepsilon_i\big|}{(n-n_1)^{1/2}}\geq ct\bigg) + \Prob\bigg(\max_{n_1<n_2}\frac{\big|\sum_{(n_1;n_2]} \varepsilon_i\big|}{n_2^{1/2}}\geq ct\bigg)\\
&\equiv (I) + (II).
\end{align*}
Due to symmetry, we only bound $(I)$. By the triangle inequality, 
\begin{align*}
(I)\leq \Prob\bigg(\sup_{n_1<n_2}\frac{\big|\sum_{i=n_1+1}^{n}\varepsilon_i\big|}{(n-n_1)^{1/2}}> ct/2\bigg) + \Prob\bigg(\sup_{n_1<n_2}\frac{\big|\sum_{i=n_2+1}^n\varepsilon_i\big|}{(n-n_1)^{1/2}}> ct/2\bigg).
\end{align*}
By L\'evy's maximal inequality (cf. Theorem 1.1.5 of \cite{de2012decoupling}), the first probability is bounded by
\begin{align*}
\sum_{r=1}^{\ceil{\log_2 n}}\Prob\bigg(\sup_{2^{r-1}\leq (n-n_1)< 2^r}2^{-(r-1)/2}\bigg|\sum_{i=n_1+1}^n \varepsilon_i\bigg|\geq ct/2\bigg)
\leq 9\ceil{\log_2 n}e^{-c' t^2 }.
\end{align*}
Similarly, the second inequality is bounded by
\begin{align*}
&\sum_{r=1}^{\ceil{\log_2 n}}\Prob\bigg(\sup_{\substack{2^{r-1}\leq (n-n_1)<2^r\\ 1\leq n_1<n_2\leq n}}(n-n_1)^{-1/2}\bigg|\sum_{i=n_2+1}^{n}\varepsilon_i\bigg|\geq ct/2\bigg)\\
&\leq \sum_{r=1}^{\ceil{\log_2 n}}\Prob\bigg(\sup_{n-2^r<n_2\leq n}2^{-(r-1)/2}\bigg|\sum_{i=n_2+1}^n \varepsilon_i\bigg|\geq ct/2\bigg)\leq 9\ceil{\log_2 n}e^{-c' t^2}.
\end{align*}
Putting together the pieces, it holds that $\Prob(Z\geq t)\leq 18\ceil{\log_2 n}e^{-c''t^2}$, where we take $c''<c_0$ without loss of generality. Now, if $\psi(\cdot)$ is bounded on $[0,\infty)$ by some $C$, then the result holds trivially. Otherwise, $\psi(x)\uparrow\infty$ as $x\rightarrow\infty$, and integration by parts yields that for any $x_0\geq 0$,
\begin{align*}
\E\psi(Z) &= \int_0^\infty \Prob(\psi(Z)\geq t)\ \d{t} = \int_0^\infty \Prob(Z\geq\psi^{-1}(t))\ \d{t}\\
&\leq \int_0^\infty \bigg\{1\wedge \big[C\log (16n)\cdot e^{-c^{\prime\prime} (\psi^{-1}(t))^2}\big]\bigg\}\ \d{t}\\
&\leq x_0 + C\cdot\int_{x_0
}^\infty \log (16n)\cdot e^{-c^{\prime\prime} (\psi^{-1}(t))^2}\ \d{t}.
\end{align*}
By monotonicity of $\psi^{-1}$, for any $t\geq x_0$, $\psi^{-1}(t)\geq \psi^{-1}(t)/2+\psi^{-1}(x_0)/2$, so the integral above can be further bounded by
\begin{align*}
\int_{x_0}^\infty \big[\log (16n) \cdot e^{-(c''/4)(\psi^{-1}(x_0))^2}\big] e^{-(c''/4)(\psi^{-1}(t))^2}\ \d{t}\leq \int_1^\infty e^{-(c''/4)(\psi^{-1}(t))^2}\ \d{t},
\end{align*}
provided that $x_0\geq 1$ and $\log (16n) \cdot e^{-(c''/4)(\psi^{-1}(x_0))^2}\leq 1$, or equivalently, $x_0 \geq 1\vee \psi\big(\sqrt{(4/c'')\log\log (16n)}\big)$. The claim now follows from the condition (\ref{cond:lil}).
\end{proof}

\section{Proofs for technical results in Section \ref{sec:noshape_upper}}\label{sec:proof_noshape}

\subsection{Proof of Proposition \ref{prop:reduction}}

\begin{proof}[Proof of Proposition \ref{prop:reduction}]
The basic inequality $\|Y-\widehat{\theta}\|^2\leq \|Y-\theta_\ora\|^2$ entails that
\begin{align*}
\|\widehat{\theta} - \theta_0\|^2 \leq \|\theta_\ora - \theta_0\|^2 + 2\varepsilon\cdot \big(\widehat{\theta} - \theta_\ora\big).
\end{align*}
Then we have, for any $\eta > 0$,
\begin{align*}
\varepsilon\cdot\big(\widehat{\theta} - \theta_\ora\big) &= \sum_{j=0}^{k-1}(\varepsilon_{[j]}\cdot (\widehat{\theta} - \theta_\ora)_{[j]}) = \sum_{j=0}^{k-1}(\varepsilon_{[j]}\cdot v_j(\widehat{\theta}))\|(\widehat{\theta} - \theta_\ora)_{[j]}\|\\
&\leq \eta^{-1}\cdot \sum_{j=0}^{k-1}\big(\varepsilon_{[j]}\cdot v_j(\widehat{\theta})\big)^2 + \eta\cdot \sum_{j=0}^{k-1}\|(\widehat{\theta} - \theta_\ora)_{[j]}\|^2\\
&= \eta^{-1}\cdot \sum_{j=0}^{k-1}\big(\varepsilon_{[j]}\cdot v_j(\widehat{\theta})\big)^2 + \eta\cdot\|\widehat{\theta} - \theta_\ora\|^2.
\end{align*}
Applying the inequality $\|\widehat{\theta} - \theta_\ora\|^2\leq 2\big(\|\widehat{\theta} - \theta_0\|^2 + \|\theta_\ora-\theta_0\|^2\big)$ then yields that
\begin{align*}
\|\widehat{\theta} - \theta_0\|^2 \leq \frac{1+2\eta}{1-2\eta}\|\theta_\ora - \theta_0\|^2 + \frac{1}{\eta(1-2\eta)}\sum_{j=0}^{k-1}(\varepsilon_{[j]}\cdot v_j(\widehat{\theta}))^2.
\end{align*}
For any given $\delta > 0$, choosing $\eta = \delta/(2\delta+4)$, upper bounding the right-hand side by the supremum over $\Theta(d,d_0,k)$, and then taking expectation on both sides yield the desired result. 
\end{proof}

\subsection{Proof of Lemma \ref{lemma:para_coef}}
\begin{proof}[Proof of Lemma \ref{lemma:para_coef}]
On the pieces $(n_{i-1}/n, n_i/n]$ and $(n_i/n,n_{i+1}/n]$, the function $f$ can be parametrized as
\begin{align*}
f_{i-1}(x) \equiv \sum_{q=1}^{d+1}a_q^{i-1}\bigg(x-\frac{n_{i-1}}{n}\bigg)^{q-1}, \quad f_i(x) \equiv \sum_{q=1}^{d+1}a_q^i\bigg(x-\frac{n_i}{n}\bigg)^{q-1}.
\end{align*}
By the fact that $0\leq p-1 \leq d_0$ and thus the continuity of the $(p-1)$th derivative at knot $n_i/n$, it holds that $f_{i-1}^{(p-1)}(n_i/n) = f_i^{(p-1)}(n_i/n)$. But
\begin{align*}
&f_{i-1}^{(p-1)}\bigg(\frac{n_i}{n}\bigg) = \sum_{q=1}^{d+1}a_q^{i-1}\frac{\text{d}^{p-1}}{\text{d}x^{p-1}}\bigg(x-\frac{n_{i-1}}{n}\bigg)^{q-1}\bigg|_{x=\frac{n_i}{n}} = \sum_{q=p}^{d+1}a^{i-1}_q\ldot(q-1;p-1)n_{i;i-1}^{q-p},\\
&f_i^{(p-1)}\bigg(\frac{n_i}{n}\bigg) = \sum_{q=1}^{d+1}a^i_q\frac{\text{d}^{p-1}}{\text{d}x^{p-1}}\bigg(x-\frac{n_i}{n}\bigg)^{q-1}\bigg|_{x=\frac{n_i}{n}} = (p-1)!a^i_p.
\end{align*}
This entails that
\begin{align*}
(p-1)!a^i_p = \sum_{q=p}^{d+1} \underline{\odot}(q-1;p-1)a^{i-1}_qn_{i;i-1}^{q-p} = \sum_{q=p}^{d+1} \frac{(q-1)!}{(q-p)!}a^{i-1}_qn_{i;i-1}^{q-p}.
\end{align*}
This implies that $\co{a^i_p; a^{i-1}_q} = (q-1)!/((q-p)!(p-1)!)n_{i;i-1}^{q-p} = {q-1 \choose p-1}n_{i;i-1}^{q-p}$ if $q\geq p$; otherwise it is 0. 
\end{proof}

\subsection{Proof of Lemma \ref{lemma:quad_forms}}
\begin{proof}[Proof of Lemma \ref{lemma:quad_forms}]
The baseline case $s = 0$ follows from the condition $\|\theta\|\leq 1$ and application of Lemma \ref{lemma:diagonal} to the piece $(n_{k_0-1};n_{k_0}]$. The iteration from $s$ to $s+1$ then follows from Lemma \ref{lemma:cancellation}, which is to be stated and proved in Appendix \ref{sec:auxiliary} with its conditions satisfied since $n_{k_0;k_0-1} \geq \max\{n_{2;1},n_{3;2},\ldots, n_{k_0-1;k_0-2}\}$ by \eqref{cond:end_long}.
\end{proof}

\subsection{Proof of Lemma \ref{lemma:last_coef}}
\begin{proof}[Proof of Lemma \ref{lemma:last_coef}]
Fix $i\leq k_0-2$ as in the lemma statement. For simplicity, we again work under the condition $n_{k_0;k_0-1} = \max\{n_{2;1},\ldots, n_{k_0;k_0-1}\}$. We will prove by induction: suppose the desired estimates hold for $a^i_\ell$, $\ell\in[d_0+1;\ell_0]$ for some $\ell_0\in[d_0+1; d]$ and we will prove that the estimate also holds for $a^i_{\ell_0+1}$. The condition of the lemma serves as the baseline $\ell_0 = d_0+1$. For the general induction from $\ell_0$ to $\ell_0+1$, let $L\equiv 1 + (d-d_0)(k_0-1-i)$. Then, Lemma \ref{lemma:quad_forms} entails that
\begin{align*}
1&\gtrsim \frac{(n-n_{k_0-1})^{2(\ell_0+1-L)+1}}{n^{2(\ell_0+1-L)}}\bigg(\sum_{\ell=\ell_0+2-L}^{\ell_0+1}\overline{\beta}^{k_0-1-i}_{\ell_0+2-L,\ell-(\ell_0+2-L)} a^i_{\ell} \bigg)^2.
\end{align*}
On the other hand, we have
\begin{align*}
1\gtrsim \sum_{\ell=\ell_0+2-L}^{\ell_0+1} n_{i+1;i}^{2(\ell-1)}M^{2}(n_{i+1},n_i)(a^i_{\ell})^2,
\end{align*} 
where the summands with $\ell\in[\ell_0+2-L;d_0+1]$ are from the condition of the lemma and those with $\ell\in[d_0+2;\ell_0+1]$ are from the induction assumption. Now, combining the above two estimates and applying Lemma \ref{lemma:two_quad} iteratively to cancel every $a^i_\ell$, $\ell\in[\ell_0+2-L;\ell_0]$, we have
\begin{align*}
1\gtrsim (a^i_{\ell_0+1})^2((I)\wedge (II)),
\end{align*}
where
\begin{align*}
(I)&\equiv \frac{(n-n_{k_0-1})^{2(\ell_0+1-L)+1}}{n^{2(\ell_0+1-L)}}(\overline{\beta}^{k_0-1-i}_{\ell_0+2-L,L-1})^2,\\
(II)&\equiv \bigwedge_{\ell = \ell_0+2-L}^{\ell_0}n_{i+1;i}^{2(\ell-1)}M(n_{i+1},n_i)\frac{(\overline{\beta}^{k_0-1-i}_{\ell_0+2-L,L-1})^2}{(\overline{\beta}^{k_0-1-i}_{\ell_0+2-L, \ell-(\ell_0+2-L)})^2}.
\end{align*}
By Lemma \ref{lemma:property_beta} and the condition $n_{k_0;k_0-1} = \max\{n_{2;1},\ldots, n_{k_0;k_0-1}\}$, we obtain that $(I)\gtrsim n_{i+1;i}^{2\ell_0}M(n_{i+1},n_i)$. Similarly, by Lemma \ref{lemma:property_beta}, as the factors $n_{\cdot;\cdot}$'s in the lower bound of ${(\overline{\beta}^{k_0-1-i}_{\ell_0+2-L,L-1})^2}/{(\overline{\beta}^{k_0-1-i}_{\ell_0+2-L, \ell-(\ell_0+2-L)})^2}$ can all be further bounded below by $n_{i+1;i}$, we obtain by direct calculation that $(II)\gtrsim n_{i+1;i}^{2\ell_0}M(n_{i+1},n_i)$. Putting together the lower bounds for $(I)$, $(II)$ completes the induction.
\end{proof}

\subsection{General statement of Lemma \ref{lemma:coef_est_1}}\label{subsec:coef_est_1}

We restate here Lemma \ref{lemma:coef_est_1} for the case of general $d_0\in[-1;d-1]$. Introduce the following notation:
\begin{align*}
\underline{\odot}n_{\cdot;j}(a,b,c)\equiv  n_{a;j}^c\cdot n_{a-1;j}^c\ldots  n_{a-\floor{b/c}-1;j}^{\Mod(b;c)}
\end{align*}
for positive integers $a,b,c$. Fix $i\geq 2$. Recall the definition $M(a,b) = (a\wedge (n-b))^{1/2}$ for $a,b\in[1;n]$ and the condition (\ref{cond:end_long}). 
\begin{lemma}\label{lemma:coef_est_1-general}
The following estimates hold for all locations $1\leq j\leq i+1$:
\begin{align*}
1 &\gtrsim \max_{1\leq \ell \leq d_0+1} \bigg\{n_{i+2;j}^{2\ell \vee 2(d-d_0)(k_0-i-2)}\\
&\times \underline{\odot}n_{\cdot;j}^2\bigg(i+1, \big\{\ell-(d-d_0)(k_0-i-2)-1\big\}\wedge \big\{(d-d_0)(i+1-j)\big\} , d-d_0\bigg)\\
& \times n_{j+1;j}^{2(\ell-1-(d-d_0)(k_0-j-1))_+}\bigg\}\cdot M^2(n_{j+1},n_j).
\end{align*}
In particular, for $j=i+1$:
\begin{align*}
1 & \geq  c\max_{1\leq \ell\leq d_0+1} \bigg\{(a^{i+1}_\ell)^2 \cdot n_{i+2;i+1}^{2(\ell-1)} \cdot M^2(n_{i+2},n_{i+1}) \bigg\}.
\end{align*}	
\end{lemma}
The proof for this general case is completely analogous to the one presented in Section \ref{subsec:noshape_main}.

\section{Proofs for technical results in Section \ref{sec:shape_upper}}\label{sec:proof_shape}

\subsection{Proof of Lemma \ref{lemma:shape_para}}
\begin{proof}[Proof of Lemma \ref{lemma:shape_para}]
For any $f\in\cal{F}_n^*(d,k)$, let $f_\circ\equiv f_\circ(f)\in \cal{F}^*_n(0,k)$ be such that $f = (I^d_{r_0,\ldots,r_{d-1};0} f_\circ)$ for some real sequence $\{r_\ell\}_{\ell=0}^{d-1}$, with corresponding knots $\{n_j\}_{j=1}^{k-1} = \{n_j(f_\circ)\}_{j=1}^{k-1}$ and magnitudes $\{\mu_j\}_{j=1}^k = \{\mu_j(f_\circ)\}_{j=1}^k$ between $(n_{j-1}/n,n_j/n]$, i.e., $f_\circ(x) = \sum_{j=1}^k \mu_j \bm{1}_{(n_{j-1}/n,n_j/n]}(x)$ for $x \in (0,1]$. Then $\mu_1\leq \ldots \leq \mu_k$. Let
\[
j^* \equiv j^*(f_\circ) \equiv \max\{1\leq j\leq k: \mu_j \leq 0\}.
\]
Define two sequences $\{\td{a}_j\}_{j=1}^{j^*}$ and $\{\td{b}_j\}_{j=j^*}^{k-1}$ as follows: $\td{a}_{j^*} \equiv \mu_{j^*} \leq 0$ and $\td{a}_j \equiv \mu_j - \mu_{j+1}\leq 0$ for $j\in[1;j^*-1]$, $\td{b}_{j^*}\equiv \mu_{j^*+1}\geq 0$ and $\td{b}_j\equiv \mu_{j+1} - \mu_j \geq 0$ for $j\in[j^*+1;k-1]$. 
Then, letting $\tau_j\equiv n_j/n$, $f_\circ$ can be re-parametrized as
\[
f_\circ(x) = \sum_{j=1}^{j^*} \td{a}_j\bm{1}_{(0, \tau_j]}(x) + \sum_{j=j^*}^{k-1} \td{b}_j\bm{1}_{(\tau_j, 1]}(x), \quad x\in (0,1].
\]
Define the function $g^-_\ell(x;\tau) \equiv (\tau-x)_+^\ell$ with any parameter $\tau\in[0,1]$. Then, direct calculation shows that 
\begin{align*}
\int_0^x g^-_\ell(u;\tau)\ \d{u} &= \int_0^{x \wedge \tau} (\tau-u)^\ell\ \d{u}= \int_{\tau-x\wedge \tau}^\tau u^\ell\ \d{u}\\
&= \frac{\tau^{\ell+1}}{\ell+1} - \frac{(\tau-x)_+^{\ell+1}}{\ell+1} = \frac{\tau^{\ell+1}}{\ell+1} + \frac{(-1)}{\ell+1} \cdot g^-_{\ell+1}(x;\tau).
\end{align*}
Similarly, with $g^+_\ell(x;\tau)\equiv (x-\tau)_+^\ell$, it holds that $\int_0^x g^+_\ell(u;\tau)\ \d u = \int_\tau^{x \vee \tau} (u-\tau)^\ell\ \d{u} = \int_0^{x \vee \tau-\tau} u^\ell\ \d{u}= g^+_{\ell+1}(x;\tau)/(\ell+1)$. This entails that 
\begin{align*}
(I^d_{r_0,\ldots,r_{d-1};0}f_\circ)(x) = \sum_{j=1}^{j^*}(-1)^d\frac{\td{a}_j}{d!}(\tau_j - x)_+^d + \sum_{j=j^*}^{k-1}\frac{\td{b}_j}{d!}(x-\tau_j)_+^d + P_{d-1}(x),
\end{align*}
where $P_{d-1}(x)$ is some polynomial of order $d-1$. The proof is then complete by noting that $\{(-1)^d\td{a}_j/d!\}_{j=1}^{j^*}$ has sign $(-1)^{d+1}$ and $\{\td{b}_j/d!\}_{j=j^*}^{k-1}$ is non-negative.
\end{proof}

\subsection{Proof of Lemma \ref{lemma:shape_coef}}

We need the following simple fact that translates the $\ell_2$ constraint on $\theta^\ast$ at the sequence level to an integral $L_2$ constraint on $f^\ast$ at the underlying function level. Its proof can be found after the proof of Lemma \ref{lemma:shape_coef}.

\begin{lemma}\label{lemma:seq2func}
Let $f^*\in \cal{F}_n^*(d,k)$ and $(\theta^*)_i \equiv (f^*(i/n))_i$. Then, if $\|\theta^*\|^2\leq 1$, there exists some $c = c(d)$ such that $1\geq c\cdot n\int_0^1 (f^*)^2(x)\ \d{x}$. Actually, this inequality holds for the larger unshaped spline space $\cal{F}_n(d,d_0,k)$.
\end{lemma}

\begin{proof}[Proof of Lemma \ref{lemma:shape_coef}]
Fix any $\theta\in\Theta^*(d,k)$ and its generating spline $f\in\cal{F}_n^\ast(d,k)$. Then, under the condition $\|\theta\|^2\leq 1$, Lemma \ref{lemma:seq2func} entails that there exists some $K = K(d) > 0$ such that $\int_0^1 f^2(x)\ \d{x} \leq K/n$. Due to scale invariance, it suffices to prove that $|c_\ell(f)|\leq C$ for $\ell\in[0;d-1]$ for some $C = C(d)$ under the condition $\pnorm{f}{2}^2=\int_0^1 f^2(x)\ \d{x} \leq 1$.

For $f \in \mathcal{F}_n^\ast(d,k)$, let $\{n_j=n_j(f)\}_{j\in[0;k]}$ be its knots and $j^*=j^\ast(f)$ be as in its canonical form in Lemma \ref{lemma:shape_para}. Let $
\tau_j\equiv\tau_j(f)\equiv n_j(f)/n$ for $j\in[1;k]$ and $\tau^*\equiv\tau^\ast(f)\equiv \tau_{j^*(f)}(f)$. We will prove that for some $K = K(d)>0$, 
\begin{align*}
\int_0^1 f^2(x)\ \d{x} \geq K\cdot \max_{0\leq \ell \leq d-1} c_\ell^2(f),\quad \text{for any } f \in \mathcal{F}_n^\ast(d,k).
\end{align*}
We focus on the case $\tau^\ast(f) \in [0,1/2]$ and prove that
\begin{align*}
\int_{\tau^*(f)}^{1} f^2(x)\ \d{x} \geq K\cdot \max_{0\leq \ell \leq d-1} c_\ell^2(f),\quad \text{for any } f \in \mathcal{F}_n^\ast(d,k).
\end{align*}
We present the proof for $c_{d-1}(f)$ whenever $c_{d-1}(f)\neq 0$; the bounds for $\{c_\ell(f)\}_{\ell\in[0;d-2]}$ follow from completely analogous arguments. Below we omit notational dependence on $f$ if no confusion could arise. On $[\tau^*,1]$, $f$ has the canonical form
\begin{align*}
f(x) = \sum_{j=j^*}^{k-1}b_j(x-\tau_j)_+^d + \sum_{\ell=0}^{d-1}\frac{c_\ell}{\ell!}x^\ell.    
\end{align*}
This can be alternatively parametrized as $f(x) = \sum_{\ell = 0}^{d-1}c_\ell x^\ell/\ell! + (I^d_{0,\ldots,0;\tau^*(f)}f_\circ)(x)$, where $f_\circ(x)\equiv \sum_{j=j^*}^{k-1}(b_j\cdot d!)\bm{1}_{x>\tau_j} \in \cal{F}_n^*(0,k)$, and $\tau^*(f) = \tau^*(f_\circ)$. Therefore, we have
\begin{align*}
1&\geq \int_{\tau^*(f)}^1 \Big(\sum_{\ell = 0}^{d-1}c_\ell x^\ell/\ell! + (I^d_{0,\ldots,0;\tau^*(f)}f_\circ)(x)\Big)^2\ \d x\\
& = c^2_{d-1} \int_{\tau^*(f_\circ)}^1 \bigg[\sum_{\ell = 0}^{d-2}\frac{c_\ell}{|c_{d-1}|}  \frac{x^\ell}{\ell!} + \mathrm{sgn}(c_{d-1}) \frac{x^{d-1}}{(d-1)!}+ \Big(I^d_{0,\ldots,0;\tau^*(f_\circ)}\frac{f_\circ}{|c_{d-1}|}\Big)(x)\bigg]^2\ \d x\\
&\geq  c_{d-1}^2 \inf_{ \substack{c_0',\ldots,c_{d-2}'\in \R,c_{d-1}^\prime\in\{\pm1\} \\ \td{f}_\circ \in \cup_n\mathcal{F}_n^\ast(0,k), \tau^\ast(\td{f}_\circ)\leq 1/2 } } \int_{\tau^*( \td{f}_\circ )}^1 \bigg[\sum_{\ell=0}^{d-1} \frac{c_\ell' x^\ell}{\ell!}+ (I_{0,\ldots,0;\tau^\ast(\td{f}_\circ)}^d \td{f}_\circ)(x)\bigg]^2\ \d{x}   \\
&=  c_{d-1}^2 \inf_{ \substack{c_0'',\ldots,c_{d-2}'' \in \R, c_{d-1}^{\prime\prime}\in\{\pm1\}\\ \td{f}_\circ \in \cup_n\mathcal{F}_n^\ast(0,k), \tau^\ast(\td{f}_\circ)\leq 1/2 } } \int_{\tau^*( \td{f}_\circ )}^1 \bigg[\sum_{\ell=0}^{d-1} \frac{c_\ell'' (x-\tau^*(\td{f}_\circ))^\ell}{\ell!}+ (I_{0,\ldots,0;\tau^\ast(\td{f}_\circ)}^d \td{f}_\circ)(x)\bigg]^2\ \d{x},
\end{align*}
where in the third line we use the fact that $\td{f}_\circ = f_\circ/|c_{d-1}|\in \cal{F}^*_n(0,k)$ and satisfies $\tau^*(\td{f}_\circ)=\tau^*(f_\circ)\leq 1/2$. Thus, to prove the desired result, it suffices to show that there exists some $K=K(d)>0$ such that
\begin{align}\label{ineq:lower_bound_cl}
\inf_{ \substack{\td{c}_0,\ldots,\td{c}_{d-2} \in \R, \td{c}_{d-1}\in\{\pm1\}\\ \td{f}_\circ \in \cup_n\mathcal{F}_n^\ast(0,k), \tau^\ast(\td{f}_\circ)\leq 1/2, k \in \mathbb{Z}_+ } } \int_{\tau^*(\td{f}_\circ)}^1 \bigg[\sum_{\ell=0}^{d-1} \frac{\td{c}_\ell (x-\tau^*(\td{f}_\circ))^\ell}{\ell!}+ (I_{0,\ldots,0;\tau^*(\td{f}_\circ)}^d \td{f}_\circ)(x)\bigg]^2\ \d{x} \geq K.
\end{align}
Suppose this is not true, then there exist a function sequence $\{\td{f}_{n,\circ}\}_n\subset \cup_{n^\prime,k^\prime}\cal{F}^*_{n'}(0,k^\prime)$ with $\tau_n^* \equiv \tau_n^*(\td{f}_{n,\circ})\subset[0,1/2]$ and real sequences $\{\td{c}_{n,\ell}\}_{n,\ell}$ with $\td{c}_{n,d-1}\in\{\pm1\}$, such that
\begin{align*}
\int_0^1 \bm{1}_{[\tau^*_n,1]}(x)\Big[\sum_{\ell = 0}^{d-1}\frac{\td{c}_{n,\ell}}{\ell!} (x-\tau_n^*)^\ell + (I^d_{0,\ldots,0;\tau_n^*}\td{f}_{n,\circ})(x)\Big]^2\ \d{x}\rightarrow 0.
\end{align*}
Since $L_2$ convergence implies almost everywhere (a.e.) convergence, it follows that
\begin{align*}
\bm{1}_{[\tau^*_n,1]}(x)\cdot\Big[\sum_{\ell = 0}^{d-1}\td{c}_{n,\ell} (x-\tau_n^*)^\ell/\ell! + (I^d_{0,\ldots,0;\tau_n^*}\td{f}_{n,\circ})(x)\Big] \rightarrow 0, \quad \text{a.e. on }[0,1].
\end{align*}
Since the sequence $\{\tau_n^*\}\subset[0,1/2]$ is bounded, $\tau_n^*\rightarrow\tau^*$ along some subsequence for some $\tau^*\in[0,1/2]$, and we work with this subsequence below. As  $\bm{1}_{[\tau^*_n,1]}(x) \to 1$ for any fixed $x\in(\tau^*,1]$, the sequence of functions in the brackets in the above display converges a.e. to $0$ on $(\tau^\ast,1]$. In other words,
\begin{align}\label{eq:ae}
\sum_{\ell = 0}^{d-1}\td{c}_{n,\ell} (x-\tau_n^*)^\ell/\ell! + (I^d_{0,\ldots,0;\tau_n^*}\td{f}_{n,\circ})(x)\rightarrow 0, \quad \text{a.e. on }(\tau^*,1].
\end{align}
We first prove that under \eqref{eq:ae}, $\{\td{c}_{n,\ell}\}_n$ is necessarily bounded for each $\ell\in[0;d-1]$. Since $\{\td{c}_{n,d-1}\}_n\subset\{-1,+1\}$ is already bounded, it suffices to prove the claim for $\ell\in[0;d-2]$. If this is not the case, then there exists some nonempty subset $\cal{L}\subset[0;d-2]$ such that for every $\ell\in\cal{L}$, $\{\td{c}_{n,\ell}\}_n$ is divergent, i.e., $\limsup_n |\td{c}_{n,\ell}| = +\infty$. As $\tau^*_n\rightarrow\tau^*$, we may find some slowly decaying $\epsilon_n\downarrow 0$ such that (i) $\varepsilon_n > (\tau^* - \tau_n^*)_+$, (ii) $\{\td{c}_{n,\ell}\varepsilon_n^\ell\}_n$ is still divergent for every $\ell\in\cal{L}$, and (\ref{eq:ae}) holds with $x_n \equiv \tau_n^* +\varepsilon_n > \tau^*$. Now, by definition of $\td{f}_{n,\circ}(\cdot)$, there exist some $k_n$, $j_n^*\in[1;k_n]$, $0\equiv \tau_{n,0}\leq \ldots\leq \tau_{n,k_n}\equiv 1$, and non-negative sequence $\{\mu_{n,j}\}_{j=j_n^*}^{k_n-1}$ such that $\tau_n^*\equiv \tau_{n,j^*_n} \leq 1/2$ and for $x\in[\tau_n^*,1]$, $\td{f}_{n,\circ}(x) = \sum_{j=j_n^*}^{k_n-1}\mu_{n,j}\bm{1}_{x > \tau_{n,j}}$. Thus by a direct calculation, we have for $x \in [\tau_n^*,1]$
\begin{align*}
(I^d_{0,\ldots,0;\tau_n^*}\td{f}_{n,\circ})(x) = \sum_{j=j_n^*}^{k_n-1}\frac{\mu_{n,j}}{d!}\big(x-\tau_n^*\big)_+^d.
\end{align*}
So by \eqref{eq:ae} and definition of $\{x_n\}$, 
\begin{align*}
\sum_{\ell = 0}^{d-1}\frac{\td{c}_{n,\ell}}{\ell!}\varepsilon_n^\ell + \sum_{j=j_n^*}^{k_n-1}\frac{\mu_{n,j}}{d!}\varepsilon_n^d\rightarrow 0.
\end{align*} 
Let $\ell_0\in\cal{L}$ be the index such that $\{\td{c}_{n,\ell}\varepsilon_n^\ell\}_n$ has the fastest divergence rate, i.e., $\limsup_n |\td{c}_{n,\ell_0}|\varepsilon_n^{\ell_0}/(|\td{c}_{n,\ell}|\varepsilon_n^\ell)\geq \alpha$ for some positive $\alpha$ and every $\ell\in\cal{L}$. Without loss of generality, we further choose $\{\varepsilon_n\}$ such that the maximal divergence rate and the index that achieves this rate are unique, i.e., $\ell_0$ is unique and satisfies $\limsup_n |\td{c}_{n,\ell_0}|\varepsilon_n^{\ell_0}/(|\td{c}_{n,\ell}|\varepsilon_n^\ell)= \infty$ for every $\ell\in \cal{L}\setminus \{\ell_0\}$. This then entails that
\begin{align}\label{eq:Bn}
B_n\equiv \sum_{j=j_n^*}^{k_n-1}\frac{\mu_{n,j}}{d!}\varepsilon_n^d \gtrsim |\td{c}_{n,\ell_0}|\varepsilon_n^{\ell_0}
\end{align}
and is positive and divergent. Next, for the chosen sequence $\{\varepsilon_n\}$, choose $\{\eta_n\}\subset[1,\infty)$ as some slowly growing sequence such that \eqref{eq:ae} holds with the sequence $x_n^\prime \equiv \tau_n^* + \varepsilon_n\eta_n\geq \tau_n^* + \varepsilon_n>\tau^*$, i.e.,  
\begin{align}\label{ineq:eta_n_conv}
\sum_{\ell = 0}^{d-1}\frac{\td{c}_{n,\ell}}{\ell!}(\eta_n\varepsilon_n)^\ell + \sum_{j=j_n^*}^{k_n-1}\frac{\mu_{n,j}}{d!}(\eta_n\varepsilon_n)^d \rightarrow 0,
\end{align}
and that $\{\varepsilon_n\eta_n\}\downarrow 0$ and $\{\td{c}_{n,\ell_0}(\eta_n\varepsilon_n)^{\ell_0}\}$ remains to be the fastest divergent sequence among $\cal{L}$, i.e., $\limsup_n |\td{c}_{n,\ell_0}|(\varepsilon_n\eta_n)^{\ell_0}/(|\td{c}_{n,\ell}|(\varepsilon_n\eta_n)^\ell)= \infty$ for every $\ell\in \cal{L}\setminus \{\ell_0\}$. Similar to \eqref{eq:Bn}, we have $\sum_{j=j_n^*}^{k_n-1}\mu_{n,j}(\eta_n\varepsilon_n)^d/d!\gtrsim |\td{c}_{n,\ell_0}|(\eta_n\varepsilon_n)^{\ell_0}$ and is positive and divergent. But this is impossible since
\begin{align*}
\sum_{j=j_n^*}^{k_n-1}\frac{\mu_{n,j}}{d!}(\eta_n\varepsilon_n)^d
& = (\eta_n)^dB_n \gtrsim (\eta_n)^{d-\ell_0}(\td{c}_{n,\ell_0}(\varepsilon_n\eta_n)^{\ell_0})\\
 & \asymp (\eta_n)^{d-\ell_0}\biggabs{\sum_{\ell = 0}^{d-1}\td{c}_{n,\ell}(\eta_n\varepsilon_n)^\ell/\ell!},
\end{align*}
where the first inequality is by \eqref{eq:Bn} and the last relation is by the maximal divergence rate of $\{\td{c}_{n,\ell_0}(\eta_n\varepsilon_n)^{\ell_0}\}$, and thus
\begin{align*}
&\sum_{\ell = 0}^{d-1}\frac{\td{c}_{n,\ell}}{\ell!}(\eta_n\varepsilon_n)^\ell + \sum_{j=j^*}^{k_n-1}\frac{\mu_{n,j}}{d!}(\eta_n\varepsilon_n)^d \\
& \gtrsim \big[(\eta_n)^{d-\ell_0}-1\big]\biggabs{\sum_{\ell = 0}^{d-1}\td{c}_{n,\ell}(\eta_n\varepsilon_n)^\ell/\ell!} \geq \big[\eta_n-1\big]\biggabs{\sum_{\ell = 0}^{d-1}\td{c}_{n,\ell}(\eta_n\varepsilon_n)^\ell/\ell!}  \to \infty,
\end{align*}
a contradiction to (\ref{ineq:eta_n_conv}). This concludes that $\{\td{c}_{n,\ell}\}_{n}$ are necessarily bounded for every $\ell\in[0;d-1]$. Thus there exists a real sequence $\{c_\ell^*\}_{\ell=0}^{d-1}$ with $c_{d-1}^* \in \{\pm1\}$ such that $ \td{c}_{n,\ell}\rightarrow c^*_\ell$ along some subsequence for each $\ell\in[0;d-1]$. Coming back to \eqref{eq:ae} and noting that $\tau_n^*\rightarrow\tau^*$ along some subsequence, we then conclude that
\begin{align}\label{eq:h_conv}
h_n(x)\equiv (I^d_{0,\ldots,0;\tau_n^*}\td{f}_{n,\circ})(x) \rightarrow \sum_{\ell=0}^{d-1}\frac{-c_\ell^*}{\ell!}(x-\tau^*)^\ell\equiv h^\ast(x)
\end{align}
a.e. on $(\tau^*,1]$ as $n \to \infty$. We will now prove that $\{c_\ell^*\}_{\ell=0}^{d-1}$ are necessarily non-positive. Fix some positive integer $m> d$ and define a regular grid on $(\tau^*,1]$: $t_i \equiv \tau^* + i(1-\tau^*)/m$ for $i\in[0;m]$. Without loss of generality, assume that $\{t_i\}_{i=1}^m$ belongs to the set with full Lebesgue measure such that \eqref{eq:h_conv} holds. Define $(\xi_{n,i})_{i=1}^m\equiv (h_n(t_i))_{i=1}^m$ (resp. $(\xi_{i}^\ast)_{i=1}^m\equiv (h^\ast(t_i))_{i=1}^m$) to be the realization of $h_n(\cdot)$ (resp. $h^\ast(\cdot))$ on this grid. Define $\nabla$ to be the finite difference operator that maps $(y_1,\ldots,y_m)^\top\in\RR^m$ to $(y_2-y_1, \ldots, y_m - y_{m-1})^\top \in\RR^{m-1}$. Then, since $\lim_n \min_{\ell \in [0;d]} h_n^{(\ell)}(x)\geq 0$ for $x \in (\tau^*,1]$, it holds that for each fixed $m\geq d+1$, $\nabla^{\ell}\xi_n \in \RR_{\geq 0}^{m-\ell}$ holds for all $\ell \in [0;d]$ for $n$ large enough. On the other hand, for each $\ell \in [0;d-1]$ and $p\in[\ell;d-1]$ , there exists some positive constant $L_{p,\ell}>0$ for such that
\begin{align*}
\big(\nabla^{\ell}\xi^\ast\big)_1 &= \bigg(\nabla^{\ell} \bigg( \sum_{p=0}^{d-1} \frac{-c_{p}^*}{p!}(t_j-\tau^*)^{p}\bigg)_{j=1}^m\bigg)_1 = \sum_{p=\ell}^{d-1} -c_p^\ast L_{p,\ell} ((1-\tau^*)/m)^p.
\end{align*}
Since for each fixed $m\geq d+1$, $\nabla^{\ell}\xi_n \rightarrow \nabla^{\ell} \xi^\ast$ as $n\rightarrow\infty$ by \eqref{eq:h_conv} and $\nabla^{\ell}\xi_n \in \RR_{\geq 0}^{m-\ell}$ for $n$ large enough, it holds that $\big(\nabla^{\ell}\xi^\ast\big)_1\geq 0$ for each fixed $m\geq d+1$. Multiplying by $m^\ell$ on both sides of the above equation and letting $m \to \infty$ we conclude that $c_\ell^\ast \leq 0$ for $\ell \in [0;d-2]$ and $c_{d-1}^* = -1$.

With $\{c_\ell^*\}_{\ell=0}^{d-1}\in\RR^d_{\leq 0}$, $h_n,h^\ast$ have the property that their derivatives up to order $d - 1$ are all convex functions, so on arbitrary compact interval contained in $(\tau^*,1)$, $D^{(d-1)}h_n$ converges uniformly to $D^{(d-1)}h^\ast\equiv 1$  (cf. Theorem 25.7 of
\cite{rockafellar1997convex} and the remark after its proof). This cannot happen as $D^{(d-1)}h_n(\tau^*_n)= 0$, $\tau^\ast_n \to \tau^\ast$ and $D^{(d-1)}h_n$ is convex. We have therefore established the contradiction and proved (\ref{ineq:lower_bound_cl}).
\end{proof}

\begin{proof}[Proof of Lemma \ref{lemma:seq2func}]
By Lemma \ref{lemma:shape_para}, any $f\in \cal{F}_n^*(d,k)$ has the canonical parametrization
\begin{align*}
f(x) = \sum_{j=1}^{j^*}a_j(\tau_j - x)_+^d + \sum_{j=j^*}^{k-1}b_j(x-\tau_j)_+^d + \sum_{\ell=0}^{d-1}c_\ell x^\ell,
\end{align*}
where $\{\tau_j\}_{j=1}^{k-1} \equiv \{n_j/n\}_{j=1}^{k-1}\subset[0,1]$. Let $\tau^*\equiv \tau_{j^*}$. Then, it holds that $\int_0^1 f^2(x)\ \d x = (I) + (II)$, where
\begin{align*}
(I) \equiv \int_0^{\tau^*} \Big(\sum_{j=1}^{j^*}a_j(\tau_j - x)_+^d + \sum_{\ell=0}^{d-1}c_\ell x^\ell\Big)^2\ \d x,\\
(II) \equiv \int_{\tau^*}^1 \Big(\sum_{j=j^*}^{k-1}b_j(x-\tau_j)_+^d+\sum_{\ell=0}^{d-1}c_\ell x^\ell\Big)^2\ \d x.
\end{align*}
We now upper bound $(II)$ by its sequence counterpart; the bound for $(I)$ is similar. Since
\begin{align*}
(II) &= \sum_{m = j^*}^{k-1}\int_{\tau_m}^{\tau_{m+1}} \Big(\sum_{j=j^*}^{k-1}b_j(x-\tau_j)_+^d+\sum_{\ell=0}^{d-1}c_\ell x^\ell\Big)^2\ \d x\\
&= \sum_{m = j^*}^{k-1}\int_{\tau_m}^{\tau_{m+1}} \Big(\sum_{j=j^*}^mb_j(x-\tau_j)^d+\sum_{\ell=0}^{d-1}c_\ell x^\ell\Big)^2\ \d x,
\end{align*}
we may bound the integral piece by piece. More generally, we show that there exists some $K=K(d)>0$ such that for any $a,b\in[0;n]$ with $b-a\geq d+1$ and $d$-degree polynomial $P(x)\equiv \sum_{\ell=0}^d c_\ell x^\ell$,
\begin{align}\label{eq:l2}
 \int_{a/n}^{b/n} P^2(x)\ \d{x}\leq K\cdot n^{-1}\sum_{i\in(a;b]}P^2(i/n).
\end{align}
The above display holds because
\begin{align*}
&\int_{a/n}^{b/n} P^2(x)\ \d{x} = \int_{a/n}^{b/n}\Big(\sum_{\ell=0}^d c_\ell x^\ell\Big)^2\ \d{x} \lesssim_d \sum_{\ell=0}^dc_\ell^2\cdot \int_{a/n}^{b/n} x^{2\ell}\ \d{x}\\
&\leq \sum_{\ell=0}^d \frac{c_\ell^2}{n}\sum_{i\in(a;b]}\Big(\frac{i}{n}\Big)^{2\ell} = \frac{1}{n}\sum_{i\in(a;b]}\sum_{\ell=0}^d \bigg(c_\ell\bigg(\frac{i}{n}\bigg)^\ell\bigg)^2 \lesssim_d \frac{1}{n}\sum_{i\in(a;b]}\bigg(\sum_{\ell=0}^d  c_\ell\bigg(\frac{i}{n}\bigg)^\ell\bigg)^2,
\end{align*}
where the last inequality is due to Lemma \ref{lemma:diagonal} and the condition $b-a\geq (d+1)$. Then for every $\theta\in\Theta(d,d_0,k)$ with unit norm constraint and the corresponding $f\in\cal{F}_n(d,d_0,k)$, by (\ref{eq:l2}) we have
\begin{align*}
1&\geq \|\theta\|^2 \geq \|\theta\|^2_{(n_{j^*};n]} = \sum_{m=j^*}^{k-1}\|\theta\|^2_{(n_m;n_{m+1}]} = \sum_{m=j^*}^{k-1} \sum_{i\in(n_m;n_{m+1}]} f^2(i/n)\\
&\gtrsim n\sum_{m=j^*}^{k-1}\int_{\tau_m}^{\tau_{m+1}} f^2(x)\ \d x = n \int_{\tau^*}^1 f^2(x)\ \d x.
\end{align*}
The bound for $(II)$ is thus complete. 
\end{proof}

\section{Auxiliary lemmas}\label{sec:auxiliary}
\begin{lemma}\label{lemma:diagonal}
Fix any positive integer $d$. There exists some $c = c(d)$ such that for any integers $n\geq 0$, $m\geq d+1$, and real sequence $\{a_\ell\}_{\ell=1}^{d+1}$, 
\begin{align*}
\sum_{i=1}^m \bigg[a_1 + a_2\bigg(\frac{i}{n}\bigg) +\ldots + a_{d+1}\bigg(\frac{i}{n}\bigg)^d\bigg]^2 \geq c\sum_{\ell=1}^{d+1} a_\ell^2\frac{m^{2\ell-1}}{n^{2(\ell-1)}}.
\end{align*}
\end{lemma}
\begin{proof}[Proof of Lemma \ref{lemma:diagonal}]
As the left hand side of the above inequality equals 
\begin{align*}
&\sum_{i=1}^n \bigg(\sum_{\ell=1}^{d+1} a_{\ell}(i/n)^{\ell-1}\bigg)^2 = \sum_{1\leq \ell,\ell'\leq d+1} a_\ell a_{\ell'} \sum_{i=1}^m (i/n)^{\ell+\ell'-2} \\
&= \sum_{1\leq \ell,\ell'\leq d+1} a_\ell (m/n)^{\ell-1} m^{1/2}\cdot  a_{\ell'} (m/n)^{\ell'-1} m^{1/2}\cdot \bigg[ m^{-(\ell+\ell'-1)} \sum_{i=1}^m i^{\ell+\ell'-2} \bigg],
\end{align*}
using matrix notation, it can be written as $x^\top A x$, where $x\equiv (a_\ell (m/n)^{\ell-1} m^{1/2})_{\ell=1}^{d+1}\in\RR^{d+1}$, and the matrix $(A)_{ij}\equiv (A(m,d))_{ij}\equiv (m^{-(i+j-1)}\sum_{k=1}^mk^{i+j-2})_{ij}\in\RR^{(d+1)\times(d+1)}$.

We first show that $A$ is strictly positive-definite for the fixed $d$ and any $m\geq d+1$. Note that $A$ is actually a moment matrix and can be written as $A_{ij} = \E(X^{i-1}\cdot X^{j-1})$, where $X$ is uniformly distributed on the set $\{1/m,\ldots,m/m\}$. Therefore, for any $c\in\mathbb{S}^d$, writing, with a slight abuse of notation, $Z\equiv \sum_{i=1}^{d+1} c_iX^{i-1}$, it holds that
\begin{align*}
c^\top Ac &= \sum_{1\leq i,j\leq d+1} c_ic_jA_{ij} = \sum_{1\leq i,j\leq d+1} c_ic_j\E(X^{i-1}\cdot X^{j-1}) = \E\bigg(\sum_{i=1}^{d+1}c_iX^{i-1}\bigg)^2 \\
&=\E Z^2 = (\E Z)^2 + \var(Z).
\end{align*}
If $\var(Z) = 0$, then $Z\equiv \alpha$ almost surely for some constant $\alpha$, which is equivalent to that the polynomial
\begin{align*}
T(x)\equiv (c_0 - \alpha) + c_1x + \ldots + c_{d+1}x^{d}
\end{align*}
having distinct roots $\{1/m,\ldots,m/m\}$. If $c_1=\ldots = c_{d+1} = 0$, then $c_0 = \pm1$ since $\|c\| = 1$, which implies that $Z= \pm1$, and thus $c^\top Ac \geq (\E Z)^2 = 1$. Otherwise, we have $c_i \neq 0$ for some $i\in[1;d]$, and hence $T(x)$ is not a constant and thus has at most $d$ roots, which contradicts the condition that $m \geq d+1$. So we conclude that $c^\top A c > 0$ for any $c\in\mathbb{S}^d$ and thus $A$ is strictly positive-definite. 

Next, we show that for any $i\in [1;d+1]$, the $(-i,-i)$-minor of $A$ (i.e. $A$ minus the $i$th row and column) is also strictly positive-definite. For this, define $Q_i$ as the permutation matrix that switches row $i$ with row $i+1$, and define $P_i \equiv Q_iQ_{i+1}\ldots Q_d$ for $i\leq d$ and $P_{d+1} \equiv I_{d+1}$, the $(d+1)$-dimensional identity matrix. Further define $B \equiv P_i^\top AP_i$. Then, the $(-i,-i)$-minor of $A$ is the $(-(d+1),-(d+1))$-minor of $B$. By Sylvester's criterion, it suffices to show that $B$ is strictly positive-definite, but for any $c\in\mathbb{S}^{d}$, it holds that
\begin{align*}
c^\top B c = c^\top P_i^\top A P_i c \equiv \td{c}^\top A \td{c} > 0,
\end{align*}
where in the last inequality we have used the fact that
\begin{align*}
\td{c}^\top \td{c} = c^\top P_i^\top P_i c = c^\top Q_{d}\ldots Q_iQ_i\ldots Q_d c = c^\top c = 1.
\end{align*}
 
Next, we show that $x^\top A x\geq ca_d^2m^{2d+1}/n^{2d}$ for some $c = c(d)$; bounds involving $a_0,\ldots,a_{d-1}$ can be similarly obtained. For this, write $A$ in the block form $[A_{11},A_{12}; A_{21}, A_{22}]$, where $A_{12}\in\RR^{d\times 1}$. Writing $y$ as the first $d$ components of $x$, i.e. $y\equiv (a_0m^{1/2}, a_1m^{3/2}/n,\ldots,a_{d-1}m^{(2d-1)/2}/n^{d-1})^\top$, we have
\begin{align*}
x^\top Ax &= (y, a_dm^{(2d+1)/2}/n^d)^\top
\begin{bmatrix}
A_{11} & A_{12}\\
A_{21} & A_{22}
\end{bmatrix}
(y, a_dm^{(2d+1)/2}/n^d)\\
&= y^\top A_{11}y + 2y^\top A_{21}a_dm^{(2d+1)/2}/n^d + A_{22}(a_dm^{(2d+1)/2}/n^d)^2.
\end{align*}
This is a quadratic form in $y$ and achieves its minimum at $y^* = -A_{11}^{-1}A_{12}a_dm^{(2d+1)/2}/n^d$ (note that  $A_{11}$, the $(-(d+1),-(d+1))$-minor of $A$, is indeed invertible as proved before), which implies that 
\begin{align*}
x^\top Ax \geq a_d^2\frac{m^{2d+1}}{n^{2d}}(A_{22} - A_{21}A_{11}^{-1}A_{12}).
\end{align*}
Therefore if we can show that $A_{22} \geq (1+\varepsilon)A_{21}A_{11}^{-1}A_{12}$ for some positive $\varepsilon = \varepsilon(d)$, then we have
\begin{align*}
x^\top A x &\geq \frac{\varepsilon}{1+\varepsilon}a_d^2\frac{m^{2d+1}}{n^{2d}}A_{22} = \frac{\varepsilon}{1+\varepsilon}a_d^2\frac{m^{2d+1}}{n^{2d}}m^{-(2d+1)}\sum_{k=1}^{m} k^{2d}\\
&\geq \frac{\varepsilon}{1+\varepsilon}a_d^2\frac{m^{2d+1}}{n^{2d}}m^{-(2d+1)}\int_0^mx^{2d}\ \d{x} = \frac{\varepsilon}{(2d+1)(1+\varepsilon)}a_d^2\frac{m^{2d+1}}{n^{2d}}.
\end{align*}
Using the block matrix inverse formula $(A^{-1})_{d+1,d+1} = (A_{22} - A_{21}A_{11}^{-1}A_{12})^{-1}$ and the fact that $(A^{-1})_{d+1,d+1} \leq \|A^{-1}\|_2 = \lambda_{\min}^{-1}(A)$ ($\lambda_{\min}$ takes the smallest eigenvalue), we have
\begin{align*}
&\qquad\quad A_{22} \geq (1+\varepsilon)A_{21}A_{11}^{-1}A_{12} \iff (1+\varepsilon)(A_{22} - A_{21}A_{11}^{-1}A_{12}) \geq \varepsilon A_{22}\\
& \iff (A^{-1})_{d+1,d+1}\leq \frac{1+\varepsilon}{\varepsilon}A_{22}^{-1} \Leftarrow \lambda_{\min}^{-1}(A) \leq \frac{1+\varepsilon}{\varepsilon}\min_{1\leq j\leq d+1}A_{jj}^{-1},
\end{align*}
which is further implied by
\begin{align}
\label{eq:criterion}
\lambda_{\min}(A)\geq \frac{\varepsilon}{1+\varepsilon}\max_{1\leq j\leq d+1}A_{jj}.
\end{align}
For this, we have, for every $j\in[1;d+1]$,
\begin{align*}
A_{jj} &= m^{-(2j-1)}\sum_{k=1}^mk^{2j-2} \leq m^{-(2j-1)}\int_{1}^{m+1}x^{2j-2}\ \d{x}\\
&\leq \frac{1}{2j-1}\bigg(1+\frac{1}{m}\bigg)^{2j-1} \leq 2^{2d+1}.
\end{align*}
It remains to show that there exists some sufficiently small $c^* = c^*(d)$ such that $\lambda_{\min}(A)\geq c^* > 0$, then we can take $\varepsilon = c^*/(2^{2d+1}-c^*)$ in \eqref{eq:criterion}. For this, let $U$ be a random variable uniformly distributed on $[0,1]$ and define matrix $\bar{A}$ as $\bar{A}_{i,j}\equiv \E\parr*{U^{i-1}\cdot U^{j-1}}$. Then, since $d$ is fixed, it holds by the definition of $A,\bar{A}$, and the Portmanteau theorem that $A\rightarrow\bar{A}$ in the matrix spectral norm as $m\rightarrow\infty$. By Weyl's inequality, there exists some positive integer $N = N(d)$ such that for $m\geq N$, $\lambda_{\min}(A)\geq \lambda_{\min}(\bar{A})/2$. On the other hand, a similar argument that establishes the positive definiteness of $A$ yields that $\lambda_{\min}(\bar{A})\geq c > 0$ for some $c = c(d)$. Therefore we can take $c^* = c^*(d) = \min_{d+1\leq m \leq N}\lambda_{\min}(A(m,d))\wedge (c/2)$. This completes the proof. 
\end{proof}

\begin{lemma}\label{lemma:wedge}
Let $\{a_i\}_{i=1}^m,\{b_i\}_{i=1}^m$ be two non-negative sequences. Then, it holds that $\parr*{\bigwedge_{i=1}^m a_i}\cdot\parr*{\bigvee_{i=1}^m b_i}\geq \bigwedge_{i=1}^m a_ib_i$. 
\end{lemma}
\begin{proof}[Proof of Lemma \ref{lemma:wedge}]
Without loss of generality, let $a_1$ be the smallest value among $\{a_i\}_{i=1}^m$. Then, it holds that $\big(\bigwedge_{i=1}^m a_i\big)\cdot\big(\bigvee_{i=1}^m b_i\big) = a_1\cdot\big(\bigvee_{i=1}^m b_i\big)\geq a_1b_1 \geq \big(\bigwedge_{i=1}^m a_ib_i\big).$
\end{proof}

\begin{lemma}\label{lemma:two_quad}
Let $\alpha_1,\alpha_2 > 0$ and $\beta_1,\beta_2$ be real numbers. Then, for any $x\in\RR$, it holds that
\begin{align*}
\alpha_1(x+\beta_1)^2 + \alpha_2(x+\beta_2)^2 \geq (\alpha_1\wedge \alpha_2)(\beta_1-\beta_2)^2/2.
\end{align*}
\end{lemma}
\begin{proof}[Proof of Lemma \ref{lemma:two_quad}]
At $x^*\equiv -(\alpha_1/(\alpha_1+\alpha_2)\cdot \beta_1 + \alpha_2/(\alpha_1+\alpha_2)\cdot \beta_2)$, the quadratic form achieves it minimum value $\frac{\alpha_1\alpha_2}{\alpha_1+\alpha_2}(\beta_1-\beta_1)^2$, which is further lower bounded by $(\alpha_1\wedge \alpha_2)(\beta_1-\beta_2)^2/2$.
\end{proof}

\begin{lemma}\label{lemma:bin_identity}
Let $n$ be any positive integer. Then, for any polynomial $P(\cdot)$ of degree strictly smaller than $n$, it holds that
\begin{align*}
\sum_{j=0}^n {n\choose j}P(j)(-1)^j = 0.
\end{align*}
\end{lemma}
\begin{proof}[Proof of Lemma \ref{lemma:bin_identity}]
We prove by induction. The claim clearly holds for $n = 1$. Suppose the claim holds for some $n$, we will prove that it also holds for $n + 1$. Let $d$ be the degree of $P(\cdot)$. We will prove that the claim holds for all monomials $P(x) \equiv x^d$ where $0\leq d \leq n=(n+1)-1$. The case $d = 0$ follows from the binomial identity:
\begin{align*}
\sum_{j=0}^{n+1} {n+1\choose j}(-1)^j = (1 + (-1))^{n+1} = 0.
\end{align*}
Next, for any $1\leq d \leq n$, it holds that
\begin{align*}
\sum_{j=0}^{n+1} {n+1\choose j} j^d (-1)^j &= \sum_{j=1}^{n+1} {n+1\choose j}j^d (-1)^j = (n+1)\sum_{j=1}^{n+1}{n\choose j-1}j^{d-1}(-1)^j\\
&= (n+1)\sum_{j=0}^n {n\choose j}(j+1)^{d-1}(-1)^j = 0,
\end{align*}
where the last identity follows from the claim for $n$ and the fact that $0\leq d - 1 \leq n-1 < n$.
\end{proof}

For the following lemma, recall the definition of the sequence $\{\overline{\beta}^\cdot_{\cdot,\cdot}\}$ defined before Lemma \ref{lemma:quad_forms}.
\begin{lemma}\label{lemma:cancellation}
Fix $d,d_0$, $k_0$ as defined in \eqref{eq:boundary}, and any $s\in[0;\floor{(d_0+1)/(d-d_0)} - 1]$. Suppose there exists some $c_1 = c_1(d)$ such that
\begin{align}\label{eq:cancel_before}
1\geq c_1\cdot\sum_{k=1}^{(s+1)d_0-sd+1}\frac{(n-n_{k_0-1})^{2k-1}}{n^{2(k-1)}} \bigg(\sum_{\ell=0}^{s(d-d_0)} \overline{\beta}_{k,\ell}^s a_{k+\ell}^{k_0-1-s}\bigg)^2. 
\end{align}
Furthermore, assume that $n_{k_0;k_0-1}\geq n_{k_0-1-s;k_0 - 2 -s}$. Then, there exists some positive constant $c_2 = c_2(d)$ such that
\begin{align*}
1\geq c_2\cdot\sum_{k=1}^{(s+1)d_0-sd+1}\frac{(n-n_{k_0-1})^{2k-1}}{n^{2(k-1)}} \bigg(\sum_{\ell=0}^{(s+1)(d-d_0)} \overline{\beta}_{k,\ell}^{s+1} a_{k+\ell}^{k_0-2-s}\bigg)^2. 
\end{align*}
\end{lemma}
Note that in the above lemma the hypothesis involves only quadratic forms with `shared coefficients' $\{a^{\cdot}_{\ell}\}_{\ell\in[1;d_0+1]}$, while the conclusion involves the ones with both `shared coefficients' $\{a^{\cdot}_{\ell}\}_{\ell\in[1;d_0+1]}$ and `nuisance coefficients' $\{a^{\cdot}_{\ell}\}_{\ell\in[d_0+2;d+1]}$.

Before the proof of Lemma \ref{lemma:cancellation}, we need one further result. For this, some extra notation is needed: 
\begin{align*}
\overline{v}^{s}_{i,j} &\equiv \frac{\overline{\odot}(i+d-d_0-j);s(d-d_0))}{j! \overline{\odot}(i+d-d_0; s(d-d_0))}(-1)^{j}n_{k_0-s;k_0-1-s}^j\\
&\qquad\qquad\times \prod_{m=1}^j \bigg(d-\big(i+(d-1-d_0)\big)-s(d-d_0)+m\bigg),\\
T_k &\equiv \sum_{\ell=0}^{s(d-d_0)} \overline{\beta}_{k,\ell}^s a_{k+\ell}^{k_0-1-s}.
\end{align*}

\begin{lemma}\label{lemma:beta_form}
Fix $d$, $d_0$, and $s$.  It holds for $i \in [1;(s+1)d_0-sd+1]$ that
\[
M\equiv\sum_{k= i}^{(s+1)d_0-sd+1} \overline{v}^{s+1}_{\overline{i},k-i}\cdot T_k = \sum_{k=0}^{(d-d_0)(s+1)}\overline{\beta}^{s+1}_{i,k}a^{k_0-2-s}_{i+k}.
\]
\end{lemma}
\begin{proof}
In order to prove the desired result, we need to show the following two claims:
\begin{itemize}
\item The coefficient of $a_{i+j}^{k_0-2-s}$ in $M$ equals $0$ for $(s+1)(d-d_0)+1\leq j\leq d-i+1$;
\item The coefficient of $a_{i+j}^{k_0-2-s}$ in $M$ equals $\overline{\beta}^{s+1}_{i,j}$ for $0\leq j\leq (s+1)(d-d_0)$.
\end{itemize}
Let
\begin{align*}
i_0&\equiv i_0(d,d_0,s,i)\equiv (s+1)d_0 - sd + 1- i,\\
\overline{i}&\equiv \overline{i}(d,d_0,s,i) \equiv (s+2)d_0-(s+1)d+2-i = i_0-(d-1-d_0),\\
 \Delta n &\equiv n_{k_0-1-s;k_0-2-s}.
\end{align*}
By definition of $M$ and Lemma \ref{lemma:para_coef}, we have 
\begin{align*}
&\co{M;a_{i+j}^{k_0-2-s}} = \sum_{k=i}^{(s+1)d_0-s d+1}\overline{v}^{s+1}_{\overline{i}, k-i}\co{T_k;a^{k_0-2-s}_{i+j}}\\
&= \sum_{k=i}^{(s+1)d_0-sd+1} \frac{\overline{\odot}(\overline{i}+d-d_0-(k-i));(s+1)(d-d_0))(-1)^{k-i}(\Delta n)^{k-i}}{(k-i)! \overline{\odot}(\overline{i}+d-d_0;(s+1)(d-d_0))}\\
&\qquad\qquad \times \prod_{m=1}^{k-i}(d-i_0-(s+1)(d-d_0)+m)\cdot \mathrm{Coef}\bigg[\sum_{\ell=0}^{s(d-d_0)}\overline{\beta}^s_{k,\ell}a^{k_0-1-s}_{k+\ell}; a^{k_0-2-s}_{i+j}\bigg]\\
&= \sum_{k=0}^{i_0} \frac{\overline{\odot}(i_0+1-k;(s+1)(d-d_0))(-1)^k(\Delta n)^k}{k!\overline{\odot}(i_0+1;(s+1)(d-d_0))}\overline{\odot}(i;k)\\
&\qquad\qquad\times \bigg(\sum_{\ell=0}^{s(d-d_0)}\overline{\beta}^s_{i+k,\ell}\binom{i+j-1}{i+k+\ell-1}(\Delta n)^{j-k-\ell}\bigg)\\
&\equiv \sum_{\ell=0}^{s(d-d_0)}(\Delta n)^{j-\ell}\overline{\beta}^s_\ell \cdot A_\ell, 
\end{align*}
where
\begin{align*}
A_\ell\equiv  &\sum_{k=0}^{i_0 }(-1)^k \frac{\overline{\odot}(i_0+1-k;(s+1)(d-d_0))}{k!\overline{\odot}(i_0+1;(s+1)(d-d_0))}\overline{\odot}(i;k)\\
&\qquad\qquad \times {i+j-1\choose i+k+\ell-1}\frac{\overline{\odot}(i+k;\ell)}{\underline{\odot}(d+1-i-k;\ell)},
\end{align*}
and we used $\overline{\beta}_{i+k,\ell}^s = D(i+k,\ell)\overline{\beta}_\ell^s = \frac{\overline{\odot}(i+k;\ell)}{\underline{\odot}(d+1-i-k;\ell)}\overline{\beta}_\ell^s$, with $\overline{\beta}_\ell^s$ defined in (\ref{eq:beta_right}). Let $C(i,j,\ell) \equiv {i+j-1\choose i+\ell-1}\cdot\overline{\odot}(i;\ell)$. Then $C(i,j,\ell) \binom{j-\ell}{k} = \overline{\odot}(i;k)\binom{i+j-1}{i+k+\ell-1}\overline{\odot}(i+k;\ell)/k!$. So $A_\ell$ equals
\begin{align*}
&\quad C(i,j,\ell)\sum_{k=0}^{i_0 }{j-\ell\choose k}\frac{\overline{\odot}(i_0+1-k;(s+1)(d-d_0))}{\overline{\odot}(i_0+1;(s+1)(d-d_0))}(-1)^k\frac{1}{\underline{\odot}(d-i-k+1;\ell)}\\
&=C(i,j,\ell)\sum_{k=0}^{i_0+(s+1)(d-d_0)}{j-\ell\choose k}\frac{\overline{\odot}(i_0+1-k;(s+1)(d-d_0))}{\overline{\odot}(i_0+1;(s+1)(d-d_0))}(-1)^k\frac{1}{\underline{\odot}(d-i-k+1;\ell)}\\
&=C(i,j,\ell)\sum_{k=0}^{j-\ell}{j-\ell\choose k}\frac{\overline{\odot}(i_0+1-k;(s+1)(d-d_0))}{\overline{\odot}(i_0+1;(s+1)(d-d_0))}(-1)^k\frac{1}{\underline{\odot}(d-i-k+1;\ell)}\\
&= \frac{C(i,j,\ell)}{\overline{\odot}(i_0+1;(s+1)(d-d_0))}\sum_{k=0}^{j-\ell}{j-\ell\choose k}(-1)^k\underline{\odot}(d-i-k+1-\ell;(s+1)(d-d_0)-\ell),
\end{align*}
where the first identity follows from the fact that $\overline{\odot}(i_0+1-k;(s+1)(d-d_0)) = 0$ for any $i_0+1\leq k\leq i_0+(s+1)(d-d_0)$, the second identity follows from the fact that $i_0 + (s+1)(d-d_0) = d-i+1\geq j\geq j-\ell$, the third identity follows from the fact that $\ell \leq s(d-d_0)< (s+1)(d-d_0)$.

For the first claim, as $\underline{\odot}(d-i-k+1-\ell;(s+1)(d-d_0)-\ell)$ is a polynomial of degree at most $(s+1)(d-d_0) - \ell < j-\ell$, Lemma \ref{lemma:bin_identity} entails that $A_\ell = 0$ for all $0\leq \ell\leq s(d-d_0)$, thus proving the first claim. We now prove the second claim under the condition $j\leq (s+1)(d-d_0)$. By definition of the $\{\overline{\beta}^\cdot_{\cdot,\cdot}\}$ sequence, we have
\begin{align*}
\overline{\beta}^{s+1}_{i,j} &= D(i,j)\overline{\beta}^{s+1}_{j} = D(i,j)\bigg\{\sum_{\ell=0}^j \binom{(s+1)(d-d_0)-\ell}{j-\ell}(\Delta n)^{j-\ell}\overline{\beta}^s_\ell\bigg\}.
\end{align*}
Therefore, to prove the claim, it suffices to match the coefficients of $\overline{\beta}^s_{\ell}$ for $0\leq \ell\leq s(d-d_0)$, as $\overline{\beta}_{\ell}^s=0$ for $\ell >s(d-d_0)$ from the definition of $\overline{\beta}_{\cdot}^{s}$, and $A_\ell =0$ for $\ell\geq j$. In other words, we only need to show $A_\ell = D(i,j)\binom{(s+1)(d-d_0)-\ell}{j-\ell}$. By using iteratively the identity ${n\choose k} = \binom{n}{k-1} + \binom{n-1}{k-1}$, one has
\begin{align*}
&\sum_{k=0}^{j-\ell}{j-\ell\choose k}(-1)^k\underline{\odot}(d-i-k+1-\ell;(s+1)(d-d_0)-\ell)\\
&= \underline{\odot}((s+1)(d-d_0)-\ell;1)\\
&\qquad\qquad \times \sum_{k=0}^{j-\ell-1}{j-\ell-1\choose k}(-1)^k\underline{\odot}(d-i-k-\ell;(s+1)(d-d_0)-1-\ell)\\
&\qquad\qquad \ldots\\
&= \underline{\odot}((s+1)(d-d_0)-\ell;j-\ell-1)\\
&\qquad\qquad\times \sum_{k=0}^1{1\choose k}(-1)^k\underline{\odot}(d-i-k+2-j;(s+1)(d-d_0)+1-j)\\
&= \underline{\odot}((s+1)(d-d_0)-\ell;j-\ell)\underline{\odot}(d-i+1-j;(s+1)(d-d_0)-j).
\end{align*}
Lastly, by direct calculation, we have
\begin{align*}
A_\ell &= \frac{C(i,j,\ell)}{\overline{\odot}(i_0+1;(s+1)(d-d_0))}\cdot\underline{\odot}((s+1)(d-d_0)-\ell;j-\ell)\\
&\qquad\qquad\times\underline{\odot}(d-i+1-j;(s+1)(d-d_0)-j)\\
&= D(i,j)\binom{(s+1)(d-d_0)-\ell}{j-\ell}.
\end{align*}
The proof is complete.
\end{proof}

\begin{proof}[Proof of Lemma \ref{lemma:cancellation}]
Define for $i\in[1; (s+1)d_0-sd +1]$
\begin{align*}
M_i \equiv \sum_{k = i}^{ (s+1)d_0-sd+1 } \frac{(n-n_{k_0-1})^{2k-1}}{n^{2(k-1)}}T_k^2.
\end{align*}
Inequality \eqref{eq:cancel_before} entails that $1\gtrsim c\sum_{i=1}^{ (s+1)d_0-sd+1 } M_i$ for some $c = c(d)$.
We have for $i \in [1;(s+1)d_0-sd+1]$,
\begin{align*}
M_i &= \frac{(n-n_{k_0-1})^{2i-1}}{n^{2(i-1)}}T_i^2 + \sum_{k= i+1}^{ (s+1)d_0-sd+1 } \frac{(n-n_{k_0-1})^{2k-1}}{n^{2(k-1)}}\frac{\parr*{\overline{v}^{s+1}_{\overline{i},k-i}\cdot T_k}^2}{(\overline{v}^{s+1}_{\overline{i},k-i})^2}\\
&\geq \bigg(\frac{(n - n_{k_0-1})^{2i-1}}{n^{2(i-1)}}\wedge \bigwedge_{k=i+1}^{ (s+1)d_0-sd+1 } \frac{(n-n_{k_0-1})^{2k-1}}{n^{2(k-1)}(\overline{v}^{s+1}_{\overline{i},k-i})^2}\bigg)\Big(T_i^2 + \sum_{k= i+1}^{(s+1)d_0-sd+1 } (\overline{v}^{s+1}_{\overline{i},k-i}\cdot T_k)^2\Big)\\
&\gtrsim \bigg(\frac{(n - n_{k_0-1})^{2i-1}}{n^{2(i-1)}}\wedge \bigwedge_{k=i+1}^{ (s+1)d_0-sd+1 } \frac{(n - n_{k_0-1})^{2k-1}}{n^{2(k-1)}(\overline{v}^{s+1}_{\overline{i},k-i})^2}\bigg)\Big(T_i + \sum_{k= i+1}^{(s+1)d_0-sd+1 } \overline{v}^{s+1}_{\overline{i},k-i}\cdot T_k\Big)^2\\
&= \bigg(\frac{(n - n_{k_0-1})^{2i-1}}{n^{2(i-1)}}\wedge \bigwedge_{k=i+1}^{(s+1)d_0-sd+1 } \frac{(n - n_{k_0-1})^{2k-1}}{n^{2(k-1)}(\overline{v}^{s+1}_{\overline{i},k-i})^2}\bigg)\bigg(\sum_{k=0}^{(d-d_0)(s+1)}\overline{\beta}^{s+1}_{i,k}a^{k_0-2-s}_{i+k}\bigg)^2\\
&\gtrsim \frac{(n - n_{k_0-1})^{2i-1}}{n^{2(i-1)}}\cdot\bigg(\sum_{k=0}^{(d-d_0)(s+1)}\overline{\beta}^{s+1}_{i,k}a^{k_0-2-s}_{i+k}\bigg)^2.
\end{align*}
Here, the second identity follows from Lemma \ref{lemma:beta_form}, and the last inequality follows, by definition of $\{\overline{v}^\cdot_{\cdot,\cdot}\}$ and the condition $n_{k_0;k_0-1}\geq n_{k_0-1-s;k_0-2-s}$, from the calculation:
\begin{align*}
&\frac{(n - n_{k_0-1})^{2i-1}}{n^{2(i-1)}}\wedge \bigwedge_{k=i+1}^{(s+1)d_0-sd+1 } \frac{(n - n_{k_0-1})^{2k-1}}{n^{2(k-1)}(\overline{v}^{s+1}_{\overline{i},k-i})^2}\\
&\asymp \bigwedge_{k=i}^{(s+1)d_0-sd+1 } \frac{(n - n_{k_0-1})^{2k-1}(n_{k_0-1-s} - n_{k_0-2-s})^{-2(k-i)}}{n^{2(k-1)}} \asymp \frac{(n - n_{k_0-1})^{2i-1}}{n^{2(i-1)}}.
\end{align*}
Putting together the lower bounds for $M_i$, $i\in[1;(s+1)d_0-sd+1]$ yields the result.
\end{proof}

\begin{lemma}\label{lemma:property_beta}
Fix any $1\leq s\leq \floor{(d_0+1)/(d-d_0)}$ and $1\leq i\leq sd_0-(s-1)d+1$. For any $0\leq j_1 \leq j_2 \leq s(d-d_0)$, define the following two quantities:
\begin{align*}
&\underline{S}(j_1)\equiv \underline{S}(j_1;d,d_0,s)\equiv \prod_{\ell=1}^{\floor{j_1/(d-d_0)}}n_{k_0-\ell;k_0-1-s}^{d-d_0}\times  n_{k_0-1-\floor{j_1/(d-d_0)};k_0-1-s}^{\Mod(j_1;d-d_0)},\\
&\overline{S}(j_2)\equiv \overline{S}(j_2;d,d_0,s)\equiv \prod_{\ell= -\floor{-j_2/(d-d_0)}+1}^s n_{k_0-\ell;k_0-1-s}^{d-d_0}  \times n_{k_0-(-\floor{-j_2/(d-d_0)});k_0-1-s}^{\Mod(-j_2;d-d_0)}.
\end{align*}
Then, there exists some positive constant $c = c(d)$ such that 
\begin{align*}
\frac{\overline{\beta}^s_{i,j_2}}{\overline{\beta}^s_{i,j_1}}\geq c\frac{ \prod_{\ell=1}^s n_{k_0-\ell;k_0-1-s}^{d-d_0} }{\underline{S}(j_1)\overline{S}(j_2)}.
\end{align*}
When $j_1 = j_2$, the product on the right hand side is to be understood as $1$. 
\end{lemma}
\begin{proof}
We only prove the special case $d_0 = d-1$ (the proof for the general case is completely analogous). Then $k_0 = d+2$, and 
\begin{align*}
\underline{S}(j_1)= \prod_{\ell=1}^{j_1} n_{d+2-\ell;d+1-s},\quad \overline{S}(j_2)= \prod_{\ell=j_2+1}^s n_{d+2-\ell;d+1-s},
\end{align*}
so we only need to prove for $s\in[1;d]$, $i\in[1;d+1-s]$, and $0\leq j_1\leq j_2\leq s$,
\begin{align*}
\frac{\overline{\beta}^s_{i,j_2}}{\overline{\beta}^s_{i,j_1}}\geq c\prod_{k=j_1+1}^{j_2}n_{d+2-k; d+1-s}.
\end{align*}
We prove this by induction on $s$. 

First consider $s=1$. Then $\overline{\beta}_j^1 =  n_{d+1;d}^j$, and $\overline{\beta}_{i,j}^1 = D(i,j)\overline{\beta}_j^1 \asymp n_{d+1;d}^j$. The only non-trivial case is $j_1=0,j_2=1$, so the claim follows.

Suppose the claim holds up to $s - 1$. Fix any $1\leq j_1\leq j_2\leq s$. The claim clearly holds for $j_1 = j_2 = s$. If $j_2 = s$ and $j_1\leq s-1$, then it holds by the recursion formula of $\{\overline{\beta}^s_j\}_{j=0}^s$ in (\ref{eq:beta_right}) that $\overline{\beta}^s_{i,s}/\overline{\beta}^s_{i,j_1} \asymp n_{d+2-s;d+1-s}\overline{\beta}^s_{s-1}/\overline{\beta}^s_{j_1}$, and we can reduce to the following case with $1\leq j_1\leq j_2\leq s-1$. For this case, note that
\begin{align*}
&\prod_{k = j_1+1}^{j_2}n_{d+2-k; d+1-s} = \prod_{k = j_1+1}^{j_2}(n_{d+2-k; d+2-s} + n_{d+2-s;d+1-s})\\
&= \sum_{k=0}^{j_2-j_1}n_{d+2-s;d+1-s}^k\sum_{j_1+1\leq m_1\neq\ldots\neq m_{j_2-j_1-k }\leq j_2} n_{d+2-m_1;d+2-s}\ldots n_{d+2-m_k; d+2-s}\\
&\asymp \bigvee_{k=0}^{j_2-j_1} \bigg\{ n_{d+2-s;d+1-s}^k \prod_{m = 1}^{j_2-j_1-k}n_{d+2-{j_1+m};d+2-s }\bigg\}.
\end{align*}
Treating the above display as a polynomial of $n_{d+2-s;d+1-s}$, it suffices to match the corresponding coefficients of $n_{d+2-s;d+1-s}^k$ for $k\in[0; j_2-j_1]$ in $\overline{\beta}^s_{i,j_2}/\overline{\beta}^s_{i,j_1}$. To this end, we have 
\begin{align*}
\frac{\overline{\beta}^s_{i,j_2}}{\overline{\beta}^s_{i,j_1}} &\asymp  \frac{\bigvee_{\ell=0}^{j_2} n_{d+2-s;d+1-s}^{j_2-\ell} \overline{\beta}_\ell^{s-1} }{ \bigvee_{\ell=0}^{j_1} n_{d+2-s;d+1-s}^{j_1} \overline{\beta}_{\ell}^{s-1} } \asymp \bigvee_{k=0}^{j_2-j_1}  \frac{\bigvee_{\ell=0}^{j_1} n_{d+2-s;d+1-s}^{j_1+k-\ell} \overline{\beta}_{j_2-j_1-k+\ell}^{s-1} }{ \bigvee_{\ell=0}^{j_1} n_{d+2-s;d+1-s}^{j_1-\ell} \overline{\beta}_\ell^{s-1} }  \\
& \geq \bigvee_{k=0}^{j_2-j_1} \bigg\{ n_{d+2-s;d+1-s}^k \bigwedge_{\ell=0}^{j_1}  \frac{\overline{\beta}_{j_2-j_1-k+\ell}^{s-1}}{\overline{\beta}_\ell^{s-1}} \bigg\} \quad (\textrm{by Lemma \ref{lemma:wedge}})\\
&\gtrsim \bigvee_{k=0}^{j_2-j_1} \bigg\{ n_{d+2-s;d+1-s}^k \bigwedge_{\ell=0}^{j_1}  \prod_{m=\ell+1}^{j_2-j_1-k+\ell} n_{d+2-m;d+2-s}  \bigg\} \quad (\textrm{by induction})\\
& = \bigvee_{k=0}^{j_2-j_1} \bigg\{ n_{d+2-s;d+1-s}^k  \prod_{m=j_1+1}^{j_2-k} n_{d+2-m;d+2-s}  \bigg\}\quad (\textrm{minimum at }\ell=j_1),
\end{align*}
matching the calculation in the previous display, completing the proof.
\end{proof}

%\section*{Acknowledgements}

%\bibliographystyle{amsalpha}
\bibliographystyle{alpha}
\bibliography{mybib}

\end{document}